\documentclass[a4paper,10.5pt]{article}
\usepackage{amsmath,amssymb,amsthm}
\usepackage{comment} 
\usepackage{color}
\theoremstyle{definition}
 \newtheorem{dfn}{Definition}[section]
 \newtheorem{remark}[dfn]{Remark}

\theoremstyle{plain}
 \newtheorem{thm}[dfn]{Theorem}
   
 \newtheorem{lem}[dfn]{Lemma}
 \newtheorem{cor}[dfn]{Corollary}

\numberwithin{equation}{section}

\newcommand{\bA}{{\mathbf A}}

\newcommand{\bD}{{\mathbf D}}

\newcommand{\bG}{{\mathbf G}}
\newcommand{\bH}{{\mathbf H}}

\newcommand{\bI}{{\mathbf I}}

\newcommand{\bU}{{\mathbf U}}
\newcommand{\BH}{{\mathbb H}}

\newcommand{\bW}{{\mathbf W}}

\newcommand{\DV}{{\rm Div}\,}
\newcommand{\dv}{\, {\rm div}\,}

\newcommand{\BR}{{\mathbb R}}

\newcommand{\BM}{{\mathbb M}}
\newcommand{\BN}{{\mathbb N}}

\newcommand{\CA}{{\mathcal A}}
\newcommand{\CB}{{\mathcal B}}
\newcommand{\CC}{{\mathcal C}}
\newcommand{\CD}{{\mathcal D}}
\newcommand{\CE}{{\mathcal E}}
\newcommand{\CF}{{\mathcal F}}
\newcommand{\CI}{{\mathcal I}}

\newcommand{\CL}{{\mathcal L}}
\newcommand{\CM}{{\mathcal M}}

\newcommand{\CR}{{\mathcal R}}
\newcommand{\CS}{{\mathcal S}}
\newcommand{\CT}{{\mathcal T}}
\newcommand{\CH}{{\mathcal H}}

\newcommand{\CU}{{\mathcal U}}

\newcommand{\bff}{{\mathbf f}}
\newcommand{\bv}{{\mathbf v}}
\newcommand{\bu}{{\mathbf u}}

\newcommand{\bg}{{\mathbf g}}
\newcommand{\bh}{{\mathbf h}}

\newcommand{\bQ}{{\mathbf Q}}
\newcommand{\bP}{{\mathbf P}}
\newcommand{\bS}{{\mathbb S}}
\newcommand{\bF}{{\mathbf F}}

\newcommand{\pd}{\partial}

\newcommand{\R}{\mathbb{R}}
\newcommand{\N}{\mathbb{N}}
\newcommand{\C}{\mathbb{C}}

\newcommand{\vp}{\varphi}

\newcommand{\fp}{{\mathfrak p}}
\newcommand{\tr}{\mathrm{tr}}
\renewcommand{\Re}{{\rm Re}~}

\title{The $\CR$-boundedness of solution operators for the Q-tensor model of nematic liquid crystals}
\author{
Daniele Barbera
\thanks{Department of Mathematical Sciences "Giuseppe Luigi Lagrange", Politecnico di Torino, 
\endgraf
Corso Duca degli Abruzzi 24, 10129 Torino, Italy
\endgraf 
e-mail address: daniele.barbera96@gmail.com
\endgraf
Partially supported by the project E53D23005450006 “Nonlinear dispersive equations in presence of singularities”- funded by European Union– Next Generation EU within the PRIN 2022 program (D.D.
 104- 02/02/2022 Ministero dell’Università e della Ricerca)
 and
 INDAM, GNAMPA group}
\enskip and \enskip
Miho Murata
\thanks{Department of Mathematical and System Engineering,
Faculty of Engineering,
Shizuoka University, 
\endgraf
3-5-1 Johoku, Chuo-ku, Hamamatsu-shi, Shizuoka,
432-8561, Japan.
\endgraf
e-mail address: murata.miho@shizuoka.ac.jp
\endgraf
Partially supported by JSPS Grant-in-Aid for Early-Career Scientists 21K13819 and Grant-in-Aid for Scientific Research (B) 22H01134
}
}
\date{}

\begin{document}
\maketitle

 \begin{abstract}
In this paper, we consider a resolvent problem 
arising from the Q-tensor model for liquid crystal flows in the half-space.
Our purpose is to show the $\CR$-boundedness for the solution operator families of the resolvent problem when the resolvent parameter lies near the origin.
The definition of the $\CR$-solvability implies the uniform boundedness of the operator
 and, consequently, the resolvent estimates for the linear system. 
\end{abstract}

\section{Introduction}
Let $\R^N_+$ and $\R^N_0$ be the upper half-space and  its boundary for $N \ge 2$, respectively, namely,
\begin{align*}
	\R^N_+&=\{x=(x', x_N) \mid x' \in \R^{N-1}, x_N>0\},\\
	\R^N_0&=\{x=(x', x_N) \mid x' \in \R^{N-1}, x_N=0\},
\end{align*}
where $x'=(x_1, \ldots, x_{N-1})$.
In this paper, we consider the following resolvent problem 
arising from the model for a liquid crystal flow proposed by Beris and Edwards \cite{BE}.
\begin{equation}\label{r0}
\left\{
\begin{aligned}
&\lambda\bu -\Delta \bu + \nabla \fp + \beta \DV (\Delta \bQ -a \bQ)=\bff,
\quad \dv \bu=0& \quad&\text{in $\R^N_+$},\\
&\lambda \bQ - \beta \bD(\bu) - \Delta \bQ + a \bQ =\bG& \quad&\text{in $\R^N_+$},\\
&\bu= \bh, \quad \pd_N \bQ=\bH& \quad&\text{on $\R^N_0$},
\end{aligned}
\right.
\end{equation}
where $\bu = \bu(x) = (u_1(x), \ldots, u_N(x))^\mathsf{T} \footnote{$\bA^\mathsf{T}$denotes the transpose of $\bA$.}$ is the fluid velocity,
$\bQ = \bQ(x)$ is a symmetric and traceless order parameter tensor describing the alignment behavior of molecule orientations,
and $\fp=\fp(x)$ is the pressure.
The right members $\bff=\bff(x)$ and $\bh = \bh(x) = (h_1(x), \dots, h_{N-1}(x), 0)^\mathsf{T}$ are given vector-valued functions,
$\bG=\bG(x)$ and $\bH=\bH(x)$ are given matrix-valued functions;
$a$ and $\beta$ are constants with $a > 0$ and $\beta \neq 0$.
For a vector-valued function $\bv$ and a $N\times N$ matrix-valued function $\bA$ with the $(j, k)$ components $A_{jk}$,
we set 
\[
	\dv \bv = \sum_{j=1}^N\pd_j v_j, \enskip \DV \bA = \left(\sum_{k=1}^N\pd_kA_{1k}, \sum_{k=1}^N\pd_kA_{2k}, \dots, \sum_{k=1}^N\pd_kA_{Nk}\right)^\mathsf{T},
\]
where $\pd_j=\pd/\pd x_j$.
In addition, $\bD(\bv)$ is the deformation rate tensor, namely, 
\[
	\bD(\bv) = \frac12(\nabla \bv +(\nabla \bv)^\mathsf{T}).
\]
The resolvent parameter $\lambda$ is supposed to be contained in a sector 
\[
\Sigma_{\epsilon} 
=\{\lambda \in \C\setminus \{0\} \mid |\arg \lambda| < \pi - \epsilon\}, 
\]
where $\epsilon \in (\epsilon_0, \pi/2)$ with $\tan \epsilon_0 \ge |\beta|/\sqrt 2$.

The resolvent problem \eqref{r0} is derived from a coupled system by the Navier-Stokes equations with the evolution equation of the order parameter tensor $\bQ$:
\begin{equation}\label{nonlinear}
\left\{
\begin{aligned}
&\pd_t \bu -\Delta \bu + \nabla \fp 
+\beta \DV (L\Delta \bQ - a\bQ) = \bff(\bu, \bQ), 
\enskip \dv \bu=0
& \quad&\text{in $\R^N_+$ for $t\in (0, T)$}, \\
&\pd_t \bQ - \beta \bD(\bu) - L\Delta \bQ + a\bQ
=\bG(\bu, \bQ) & \quad&\text{in $\R^N_+$ for $t\in (0, T)$}, \\
&\bu= \bh, \enskip \pd_N \bQ=\bH& \quad&\text{on $\R^N_0$ for $t\in (0, T)$},\\
&(\bu, \bQ)|_{t=0} = (\bu_0, \bQ_0)& \quad&\text{in $\R^N_+$}.
\end{aligned}\right.
\end{equation}
The nonlinear terms are 
\begin{align*}
\bff(\bu, \bQ)
&= -(\bu \cdot \nabla) \bu + 
\DV[2\xi \BH: \bQ(\bQ+\bI/N) - (\xi + 1) \BH\bQ +
 (1-\xi) \bQ\BH - L\nabla \bQ \odot \nabla \bQ]- \beta \DV F(\bQ),\\
\bG(\bu, \bQ)
&= -(\bu \cdot \nabla)\bQ + \xi(\bD(\bu) \bQ + \bQ \bD(\bu))
+\bW(\bu) \bQ - \bQ \bW(\bu) -2\xi (\bQ +\bI/N) \bQ : \nabla \bu + F(\bQ),
\end{align*}
where
$L>0$ is a material-dependent elastic constant, 
a scalar parameter $\xi \in \R$
denotes the ratio between the tumbling 
and the aligning effects that a shear flow would exert over the directors,
$\beta
= 2\xi/N$,
$\bI$ is the $N \times N$ identity matrix,
the $(j, k)$ component of $\nabla \bQ \odot \nabla \bQ$ and $\bW(\bu)$ denote
\[
	(\nabla \bQ \odot \nabla \bQ)_{j k} = \sum^N_{\ell, m = 1} \pd_j Q_{\ell m} \pd_k Q_{\ell m}, \enskip
	\bW(\bu) = \frac12(\nabla \bu -(\nabla \bu)^\mathsf{T}),
\]
respectively.
Furthermore, we set
\begin{align*}
\BH& = L\Delta \bQ - a\bQ + b(\bQ^2-(\tr \bQ^2)\bI/N) -c\tr(\bQ^2) \bQ
\end{align*}
with a material-dependent and temperature-dependent non-zero constant $a$ and material-dependent positive constants $b$ and $c$, 
where $\BH$ is derived from the variational derivative of the free energy functional. See, e.g., \ cite {HD, MZ} for details.
The nonlinear term of $\BH$ is defined as $F(\bQ)$.
In our problem, we set $L=1$ for simplicity. In addition, we assume that $\xi \neq 0$ and $a>0$. The assumption $a>0$ is a mathematical point of view.

The aim of the paper is to prove the resolvent estimate for $\lambda$ near $0$.
In the previous paper \cite{BM}, we considered the case $|\lambda| > r$ with $r>0$.
However, the case $\lambda$ near $0$ is essential in order to prove the global well-posedness for the nonlinear problem.
To understand this point, we consider the linearized system
\[
\left\{
\begin{aligned}
	&\pd_t \bU + \CA \bU=\bF& \quad&\text{in $\R^N_+$ for $t\in (0, T)$}, \\
	&\CB \bU = 0& \quad&\text{on $\R^N_0$ for $t\in (0, T)$},\\
	&\bU(0)=0 & \quad&\text{in $\R^N_+$},
\end{aligned}\right.
\]
where $\CA$ and $\CB$ are linear operators, $\bU=(\bu, \bQ)$, and $\bF=(\bff, \bG)$ is a given function.
As we proved in \cite[Theorem 1.2.2]{BM}, 
let $1< p, q< \infty$,
then there exists $\gamma_0>0$ such that for any $\gamma \ge \gamma_0$ the solution $\bU$ of the linearized problem satisfies
\begin{equation}\label{exp}
\begin{aligned}
	&\|\pd_t \bU\|_{L_p((0, T), L_q(\R^N_+) \times H^1_q(\R^N_+))} + \|\bU\|_{L_p((0, T), H^2_q(\R^N_+) \times H^3_q(\R^N_+))} \\
	&\enskip \le C e^{\gamma T} (\|\bff\|_{L_p((0, T), L_q(\R^N_+))} + \|\bG\|_{L_p((0, T), H^1_q(\R^N_+))})
\end{aligned}
\end{equation}
with some constant $C>0$, where $\gamma_0$ depends on $r$.
The analysis of the resolvent problem for $\lambda$ lying near the origin corresponds to proving \eqref{exp} with $\gamma_0=0$.
{
Consequently, in the following, we consider $\lambda$ in a sector
\[
\Sigma_{\epsilon, c_0} 
=\{\lambda \in \Sigma_\epsilon \mid |\lambda| \le c_0\},
\]
where $c_0$ is some small constant depending on $\epsilon$, $\beta$, and $a$
chosen in Lemma \ref{lem:zl} below.

The resolvent problem for the Q-tensor model was investigated by Schonbek and Shibata \cite{SS1} in the whole space.
They proved the $\CR$-boundedness for the solution operators to the resolvent problem with a modified stress tensor, which implies that the first equation in \eqref{r0} is the Stokes resolvent problem.
Therefore, the $\CR$-boundedness for the Q-tensor model follows from that for the Stokes equations and the heat equation.
Since the $\CR$-boundedness furnishes the maximal $L_p$-$L_q$ regularity for $1< p, q<\infty$, by combining these estimates with the decay estimates of the Stokes and the heat semigroups, they obtained the global well-posedness for the modified system in the whole space. 
The second author and Shibata \cite{MS} proved $\CR$-boundedness for the solution operators for \eqref{r0} in the whole space, where $\lambda$ is far away from the origin. Then \cite{MS} proved the global well-posedness in the maximal $L_p$-$L_q$ regularity class for \eqref{nonlinear} in the whole space.
Xiao \cite{X} assumed $\xi=0$ in a bounded domain; therefore, 
he could refer to \cite{A, D, GS} concerning the maximal $L_p$-$L_q$ regularity for the Stokes and the parabolic operators. Then he proved the existence of global-in-time solutions for \eqref{nonlinear} with $\xi=0$ in the bounded domain.
Hieber, Hussein, and Wrona \cite{HHW} proved that the linear operator is $\CR$-sectorial by proving that the linear operator is invertible and its numerical range lies in a certain sector, which is based on a classical result for unbounded operators in Hilbert spaces (cf. Kato \cite{K}), which implies that the linear operator has the maximal $L_p$-$L_2$ regularity for $p>4/{4-N}$ with $N=2, 3$. Then they established the global well-posedness in a bounded domain for any $\xi$.

This paper is organized as follows: Section 2 states the existence of the $\CR$-bounded solution operator families and the resolvent estimates to the resolvent problem \eqref{r0}.
As a corollary, we state the resolvent estimates in the homogeneous Sobolev spaces for the resolvent problem with homogeneous boundary conditions,
which may be useful to prove the global well-posedness in the maximal $L_p$-$L_q$ regularity class. 
Section 3 improves the result proved by \cite{MS}, which is the $\CR$-boundedness for the solution operator for $\lambda \in \Sigma_\epsilon$ in the whole space.
As a preliminary for the proof of the main theorem, Section 4 develops the estimates of the terms appearing in the Lopatinski determinant in the half-space.
Section 5 proves the main theorem by estimating the symbols of the Fourier multiplier operators.
Finally, Section 6 states the $\CR$-solvability for the resolvent problem with homogeneous boundary conditions.

\section{Notation and Main Theorem}
First, we summarize several symbols and functional spaces used 
throughout the paper.
$\N$, $\R$ and $\C$ denote the sets of 
all natural numbers, real numbers, and complex numbers, respectively. 
We set $\N_0=\N \cup \{0\}$ and $\R_+ = (0, \infty)$. 
Let $q'$ be the dual exponent of $q$
defined by $q' = q/(q-1)$
for $1 < q < \infty$. 
For any multi-index $\alpha = (\alpha_1, \ldots, \alpha_N) 
\in \N_0^N$, we write $|\alpha|=\alpha_1+\cdots+\alpha_N$ 
and $D_x^\alpha=\pd_1^{\alpha_1} \cdots \pd_N^{\alpha_N}$ 
with $x = (x_1, \ldots, x_N)$ and $\pd_j=\pd/\pd x_j$. 
For $k \in \N_0$, scalar function $f$, $N$ vector-valued function $\bg$, 
and $N \times N$ matrix-valued function $\bG$, we set
\begin{gather*}
\nabla^k f = (D_x^\alpha f \mid |\alpha|=k),
\enskip \nabla^k \bg = (D_x^\alpha g_j \mid |\alpha|=k, \enskip j = 1,\ldots, N),\\
\nabla^k \bG = (D_x^\alpha G_{i j} \mid |\alpha|=k, \enskip i, j = 1,\ldots, N).
\end{gather*} 
Hereafter, we use small boldface letters, e.g. $\bff$ to 
denote vector-valued functions and capital boldface letters, e.g. $\bG$
to denote matrix-valued functions, respectively. 
For $\xi'=(\xi_1, \dots, \xi_{N-1})$ and $\alpha=(\alpha_1, \ldots, \alpha_{N-1}) 
\in \N_0^{N-1}$, 
we set
\[
	D^\alpha_{\xi'}  = \frac{\pd^{|\alpha| }}{\pd \xi_1^{\alpha_1} \cdots \pd \xi_{N-1}^{\alpha_{N-1}}},\enskip |\alpha|=\alpha_1+\cdots+\alpha_{N-1}.
\]
For Banach spaces $X$ and $Y$, $\CL(X,Y)$ denotes the set of 
all bounded linear operators from $X$ into $Y$,
$\CL(X)$ is the abbreviation of $\CL(X, X)$, and 
$\rm{Hol}\,(U, \CL(X,Y))$ 
 the set of all $\CL(X,Y)$ valued holomorphic 
functions defined on a domain $U$ in $\mathbb C$. 
For $1 \le p \le \infty$ and $m \in \N$,
$L_p(\Omega)$ and $H_p^m(\Omega)$ 
denote the Lebesgue space and Sobolev space on a domain $\Omega \subset \R^N$;
while $\|\cdot\|_{L_q(\Omega)}$ and $\|\cdot\|_{H_q^m(\Omega)}$
denote their norms, respectively.$C^\infty(\Omega)$ denotes the set of all $C^\infty$ functions defined on $\Omega$. 
$L_p((a, b), X)$ and $H_p^m((a, b), X)$ 
denote the Lebesgue space and the Sobolev space of 
$X$-valued functions defined on an interval $(a,b)$, respectively.
The $d$-product space of $X$ is defined by $X^d$,
while its norm is denoted by 
$\|\cdot\|_X$ instead of $\|\cdot\|_{X^d}$ for the sake of 
simplicity. 
Let $\bS_0 \subset \R^{N^2}$ denotes the set of the order parameter, namely,
\[
 \bS_0=\{\bQ \in \R^{N^2} \mid \tr \bQ=0, \enskip \bQ=\bQ^\mathsf{T}\}.
\]
The norm of a matrix $\bA$ is given by $|\bA|^2 = \tr (A^\mathsf{T} A)$ and hence it follows for all $\bQ \in \bS_0$ that $|\bQ|^2=\tr (\bQ^2)$.
The space for the pressure term is defined as 
\begin{align*}
\widehat H^1_q(\Omega)&=\{\vp \in L_{q, \text{loc}}(\Omega) \mid \nabla \vp \in L_q(\Omega)^N\}.
\end{align*}
Let $\CF_{x'}= \CF$ and $\CF^{-1}_{\xi'} = \CF^{-1}$ 
denote the partial Fourier transform and 
its inverse transform, respectively, which are defined by 
 setting
\begin{align*}
&\hat f (\xi', x_N)
= \CF_{x'}[f](\xi', x_N) = \int_{\R^{N-1}}e^{-ix'\cdot\xi'}f(x', x_N)\,dx', \\
&\CF^{-1}_{\xi'}[g(\xi', x_N)](x') = \frac{1}{(2\pi)^{N-1}}\int_{\R^{N-1}}
e^{ix'\cdot\xi'}g(\xi', x_N)\,d\xi'. 
\end{align*}
The letter $C$ denotes generic constants and the constant 
$C_{a,b,\ldots}$ depends on $a,b,\ldots$. 
The values of constants $C$ and $C_{a,b,\ldots}$ 
may change from line to line. 

Second, we introduce the definition of the $\CR$-boundedness.
\begin{dfn}\label{dfn2}
A family of operators $\CT \subset \CL(X,Y)$ is called $\CR$-bounded 
on $\CL(X,Y)$, if there exist constants $C > 0$ and $p \in [1,\infty)$ 
such that for any $n \in \BN$, $\{T_{j}\}_{j=1}^{n} \subset \CT$,
$\{f_{j}\}_{j=1}^{n} \subset X$ and sequences $\{r_{j}\}_{j=1}^{n}$
 of independent, symmetric, $\{-1,1\}$-valued random variables on $[0,1]$, 
we have  the inequality:
$$
\bigg \{ \int_{0}^{1} \|\sum_{j=1}^{n} r_{j}(u)T_{j}f_{j}\|_{Y}^{p}\,du
 \bigg \}^{1/p} \leq C\bigg\{\int^1_0
\|\sum_{j=1}^n r_j(u)f_j\|_X^p\,du\biggr\}^{1/p}.
$$ 
The smallest such $C$ is called $\CR$-bound of $\CT$, 
which is denoted by $\CR_{\CL(X,Y)}(\CT)$.
\end{dfn}
\begin{remark}\label{rem:def of rbdd}
The $\CR$-boundedness implies that the uniform boundedness of the operator family $\CT$.
In fact, choosing $m=1$ in Definition \ref{dfn2}, we observed that there exists a constant $C$ such that $\|T f\|_Y \le C \|f\|_X$ holds for any $T \in \CT$ and $f \in X$.
\end{remark}

Finally, let us state our main theorem.

\begin{thm}\label{thm:Rbdd H}
Let $N \ge 2$, $a>0$, and $1 < q < \infty$. Let $\epsilon \in (\epsilon_0, \pi/2)$ with $\tan \epsilon_0 \ge |\beta|/\sqrt 2$, and let $c_0$ is a small constant depending on $\epsilon$, $\beta$, and $a$
chosen in Lemma \ref{lem:zl} below. 
Let 
\begin{align*}
X_q(\R^N_+)&=L_q(\R^N_+)^N \times L_q(\R^N_+; \R^{N^3}) \times L_q(\R^N_+)^{N^3+N^2+N} \\
&\quad \times L_q(\R^N_+; \R^{N^4}) \times L_q(\R^N_+; \R^{N^3}) \times L_q(\R^N_+; \bS_0) \times L_q(\R^N_+; \bS_0),\\
Y_q(\R^N_+)&=L_q(\R^N_+)^N \times L_q(\R^N_+; \R^{N^3}) \times L_q(\R^N_+; \bS_0) \times L_q(\R^N_+)^{N^3+N^2+N} \\
&\quad \times L_q(\R^N_+; \R^{N^4}) \times L_q(\R^N_+; \R^{N^3}) \times L_q(\R^N_+; \bS_0) \times L_q(\R^N_+; \R^{N^3}) \times L_q(\R^N_+; \bS_0) \times L_q(\R^N_+; \bS_0), 
\end{align*}
and let 
\begin{align*}
\bF_X&=(\bff, \nabla \bG, \CS_\lambda \bh, \CS_\lambda \bH, \lambda^{1/2}\bH),\\
\bF_Y&=(\bff, \nabla \bG, \bG, \CS_\lambda \bh, \CT_\lambda \bH, \bH)
\end{align*}
with $h_N=0$, where
$\CS_\lambda = (\nabla^2, \lambda^{1/2}\nabla, \lambda)$,
$\CT_\lambda = (\nabla^2, \lambda^{1/2}\nabla, \lambda, \lambda^{1/2}, \nabla)$.
There exist 
operator families 
\begin{align*}
&\CA (\lambda) \in 
{\rm Hol} (\Sigma_{\epsilon, c_0}, 
\CL(X_q(\R^N_+), H^2_q(\R^N_+)^N))\\
&\CB (\lambda) \in 
{\rm Hol} (\Sigma_{\epsilon, c_0}, 
\CL(Y_q(\R^N_+), H^3_q(\R^N_+; \bS_0)))
\end{align*}
such that 
the following assertions hold.
\begin{enumerate}
\item
For any $\lambda = \gamma + i\tau \in \Sigma_{\epsilon, c_0}$,
$\bF_X \in X_q(\R^N_+)$, and $\bF_Y \in Y_q(\R^N_+)$,
\[
\bu = \CA (\lambda) \bF_X,\quad
\bQ = \CB (\lambda) \bF_Y
\]
are unique solutions of \eqref{r0}.
\item
\begin{align*}
&\CR_{\CL(X_q(\R^N_+), A_q(\R^N_+))}
(\{(\tau \pd_\tau)^\ell \CS_\lambda \CA (\lambda) \mid 
\lambda \in \Sigma_{\epsilon, c_0}\}) 
\leq r,\\
&\CR_{\CL(Y_q(\R^N_+), B_q(\R^N_+))}
(\{(\tau \pd_\tau)^\ell \CT_\lambda \CB (\lambda) \mid 
\lambda \in \Sigma_{\epsilon, c_0}\}) 
\leq r
\end{align*}
for $\ell = 0, 1,$
where 
$A_q(\R^N_+) = L_q(\R^N_+)^{N^3 + N^2+N}$,
$B_q(\R^N_+) = H^1_q(\R^N_+; \R^{N^4}) \times H^1_q(\R^N_+; \R^{N^3}) \times H^1_q(\R^N_+; \bS_0) \times L_q(\R^N_+; \R^{N^3}) \times L_q(\R^N_+; \bS_0) $,
and $r$ is a constant independent of $\lambda$.
\end{enumerate}
\end{thm}

\begin{remark}
The condition $h_N=0$ provides the variational equation for $\fp$, which implies the solution formula of the pressure satisfying \eqref{r0} with $(\bff, \bG)=(0, 0)$ (cf. \cite[subsection 2.1]{BM}).
\end{remark}

Remark \ref{rem:def of rbdd} implies the resolvent estimates for \eqref{r0}.
\begin{cor}\label{cor:resolvent}
Let $N \ge 2$, $a>0$, and $1 < q < \infty$.
Let $\epsilon \in (\epsilon_0, \pi/2)$ with $\tan \epsilon_0 \ge |\beta|/\sqrt 2$, and let $c_0$ is a small constant depending on $\epsilon$, $\beta$, and $a$
chosen in Lemma \ref{lem:zl} below.   
Then for any $\lambda \in \Sigma_{\epsilon, c_0}$, $\bff \in L_q(\R^N_+)^N$, $\bG \in H^1_q(\R^N_+)$, 
$\bh \in H^2_q(\R^N_+)^N$ with $h_N=0$, and $\bH \in H^2_q(\R^N_+; \bS_0)$,
there is a unique solution $(\bu, \bQ, \fp)$ for \eqref{r0}, unique up to additive constant on $\fp$, 
with $\bu \in H^2_q(\R^N_+)^N$, $\bQ \in H^3_q(\R^N_+; \bS_0)$, 
$\fp \in \widehat H^1_{q}(\R^N_+)$,
\begin{equation}\label{est:uq} 
\begin{aligned}
&\|(|\lambda|, |\lambda|^{1/2} \nabla, \nabla^2)(\bu, \bQ)\|_{L_q(\R^N_+) \times H^1_q(\R^N_+)}
\\
&\le C (\|\bff\|_{L_q(\R^N_+)}+\|\bG\|_{H^1_q(\R^N_+)}+\|(|\lambda|, |\lambda|^{1/2} \nabla, \nabla^2) \bh\|_{L_q(\R^N_+)}\\
&\quad
+\|\bH\|_{H^2_q(\R^N_+)} + |\lambda|^{1/2}\|\bH\|_{H^1_q(\R^N_+)}+|\lambda|\|\bH\|_{L_q(\R^N_+)}),
\end{aligned}
\end{equation}
and 
\begin{equation}\label{est:p}
\begin{aligned}
\|\nabla \fp\|_{L_q(\R^N_+)} &\le C(\|(\bff, \nabla \bG)\|_{L_q(\R^N_+)}+\|(|\lambda|, |\lambda|^{1/2} \nabla, \nabla^2) \bh\|_{L_q(\R^N_+)}\\
&\quad
+\|\bH\|_{H^2_q(\R^N_+)} + |\lambda|^{1/2}\|\bH\|_{H^1_q(\R^N_+)}+|\lambda|\|\bH\|_{L_q(\R^N_+)}),
\end{aligned}
\end{equation}
where $C$ is a constant depending on $a$, $\beta$, $N$, $\epsilon$, and $q$.
\end{cor}

In particular, the following resolvent estimate holds in the homogeneous Sobolev spaces for \eqref{r0} with $(\bh, \bH) = (0, 0)$. 
\begin{cor}\label{cor:homo}
Let $1 < q < \infty$ and
$\epsilon \in (\epsilon_0, \pi/2)$ with $\tan \epsilon_0 \ge |\beta|/\sqrt 2$, and let $c_0$ is a small constant depending on $\epsilon$, $\beta$, and $a$
chosen in Lemma \ref{lem:zl} below. 
Then for any $\lambda \in \Sigma_{\epsilon, c_0}$, $\bff \in L_q(\R^N_+)^N$, and $\bG \in \dot H^1_q(\R^N_+; \bS_0)$,
there is a unique solution $(\bu, \bQ, \fp)$ for \eqref{r5}, unique up to additive constant on $\fp$, 
with $\bu \in H^2_q(\R^N_+)^N$, $\bQ \in H^3_q(\R^N_+; \bS_0)$, 
$\fp \in \widehat H^1_{q}(\R^N_+)$, and
\begin{align*}
&\|(|\lambda|, |\lambda|^{1/2} \nabla, \nabla^2)(\bu, \bQ)\|_{L_q(\R^N_+) \times \dot H^1_q(\R^N_+)}
+\|\nabla \fp\|_{L_q(\R^N_+)}
\le C \|(\bff, \nabla \bG)\|_{L_q(\R^N_+)},
\end{align*}
where $C$ is a constant depending on $a$, $\beta$, $N$, $\epsilon$, and $q$.
\end{cor}

\section{The whole-space case}
Let $\epsilon \in (\epsilon_0, \pi/2)$ with $\tan \epsilon_0 \ge |\beta|/\sqrt 2$.
As stated by \cite[Theorem 2.3]{MS}, there exist $\CR$-bounded solution operator families for the following problem:
\begin{equation}\label{r RN}
\left\{
\begin{aligned}
&\lambda \bu -\Delta \bu + \nabla \fp +\beta \DV \left(\Delta \bQ - a \bQ \right) = \bff,
\enskip \dv \bu = 0
& \quad&\text{in $\R^N$}, \\
&\lambda \bQ - \beta \bD(\bu) -\Delta \bQ + a\bQ
=\bG & \quad&\text{in $\R^N$}
\end{aligned}
\right.
\end{equation}
for $\lambda \in \Sigma_\epsilon$ with $|\lambda| \ge 1$.
Now we extend the result in the case that $\lambda \in \Sigma_{\epsilon, 1}$.
According to \cite{MS}, the characteristic equation for the velocity $P_2(\xi, \lambda)$ is written by
\begin{align*}
P_2(\xi, \lambda) &=(\lambda+|\xi|^2)(\lambda+|\xi|^2+a) + \frac{\beta^2}{2}(|\xi|^4+a|\xi|^2)\\
&=\lambda^2+2\left(|\xi|^2+\frac{a}{2}\right)\lambda+\left(1+\frac{\beta^2}{2}\right)(|\xi|^4+a|\xi|^2)\\
&= (\lambda - \lambda_+(|\xi|))(\lambda - \lambda_-(|\xi|)),
\end{align*}
where
\[
\lambda_\pm(|\xi|)=-\left(|\xi|^2+\frac{a}{2}\right) \pm \sqrt{\frac{a^2}{4} -a\frac{\beta^2}{2} |\xi|^2 -\frac{\beta^2}{2}|\xi|^4}
\]
which has the following expansions :
\begin{equation}\label{p2 small}
\left\{
\begin{aligned}
&\lambda_+(|\xi|) = -\left(1+\frac{\beta^2}{2}\right)|\xi|^2 + O(|\xi|^4), \\
&\lambda_-(|\xi|) = -\left(1-\frac{\beta^2}{2}\right)|\xi|^2 - a + O(|\xi|^4)  \enskip \text{as } |\xi| \to 0.
\end{aligned}
\right.
\end{equation}
\begin{equation}\label{p2 high}
\lambda_\pm(|\xi|)
= \left(- 1 \pm i\frac{|\beta|}{\sqrt 2}\right)|\xi|^2 +O(1) \enskip \text{as } |\xi| \to \infty.
\end{equation}
The lower bound of $P_2(\xi, \lambda)$ is obtained by the following lemma.
\begin{lem}[e.g. \cite{S} Lemma 3.5.2]\label{spectrum}
Let $0 < \epsilon < \pi/2$.
Then for any $\lambda \in \Sigma_{\epsilon}$
and
$\alpha \geq 0$, the following estimate holds.
\[
|\lambda + \alpha| \geq \left(\sin \frac{\epsilon}{2}\right) (|\lambda| + \alpha).
\]
\end{lem}

\begin{lem}\label{lem:spectrum2}
Let $a>0$. Let $\epsilon \in (\epsilon_0, \pi/2)$ with $\tan \epsilon_0 \ge |\beta|/\sqrt 2$.
Then for any $(\xi, \lambda) \in \R^N \setminus\{0\} \times \Sigma_\epsilon$,
we have     
\begin{equation}\label{spectrum2}
|P_2(\xi, \lambda)| \geq C_{\epsilon, \beta, a} (|\lambda|^{1/2}+ |\xi|)^2 (|\lambda|^{1/2}+1+ |\xi|)^2
\end{equation}
with some constant $C_{\epsilon, \beta, a}$ independent of $\xi$ and $\lambda$.
\end{lem}
\begin{proof}
First, we consider the case that $|\xi|^2 \le r$ with small $r$.
Choosing $r$ so small that $(1+\beta^2/2)|\xi|^2-|O(|\xi|^4)| \ge C_1|\xi|^2$ and $a-\beta^2/2|\xi|^2-|O(|\xi|^4)| \ge C_2$ with some positive constants $C_1$ and $C_2$,
\eqref{p2 small} and Lemma \ref{spectrum} imply that
\begin{equation}\label{p2}
|P_2(\xi, \lambda)| \ge C_{\epsilon, \beta}(|\lambda|^{1/2}+|\xi|)^2 (|\lambda|^{1/2}+1+|\xi|)^2. 
\end{equation}
Second, the case that $|\xi|^2 \ge R$ with large $R$ depending on $\epsilon$ and $\beta$ can be treated by the same proof of \cite[Lemma 2.8]{MS}, then we have
\[
|P_2(\xi, \lambda)| \ge C_{\epsilon, \beta}(|\lambda|^{1/2}+|\xi|)^4
\ge C_{\epsilon, \beta}(|\lambda|^{1/2}+|\xi|)^2 (|\lambda|^{1/2}+1+|\xi|)^2. 
\]
Finally, we consider the case that $r\le |\xi|^2 \le R$.
If $|\lambda| \ge \lambda_0$ with large $\lambda_0$ determined below,
Lemma \ref{spectrum} implies that
\[
|P_2(\xi, \lambda)| \ge (\sin \epsilon/2)^2(|\lambda|+|\xi|^2)(|\lambda|+|\xi|^2+a)-\frac{\beta^2}{2}(|\xi|^4+a|\xi|^2).
\]
Note that 
\[
(|\xi|^4+a|\xi|^2) \le (R+a)|\xi|^2 \le \frac{R+a}{\lambda_0}|\lambda||\xi|^2\le \frac{R+a}{\lambda_0}(|\lambda|+|\xi|^2)(|\lambda|+|\xi|^2+a)
\]
for $|\lambda| \ge \lambda_0$.
Choosing $\lambda_0$ such that $(\sin \epsilon/2)^2 \ge \beta^2(R+a)/2\lambda_0$, we have
\[
|P_2(\xi, \lambda)| \ge C_{\epsilon, \beta, a}(|\lambda|+|\xi|^2) (|\lambda|+1+|\xi|^2),
\]
which implies \eqref{spectrum2}.
For $\lambda \in \overline{\Sigma}_\epsilon$ with $|\lambda| \le \lambda_0$ and $\xi \in \R^N \setminus\{0\}$,
$\lambda-\lambda_\pm(|\xi|) \neq 0$, namely,
$P_2(\xi, \lambda) \neq 0$.
Thus, there exists a constant $c>0$ such that 
\[
|P_2(\xi, \lambda)| \ge c \ge C_{\epsilon, \beta, a}(|\lambda|+|\xi|^2) (|\lambda|+1+|\xi|^2),
\]
which completes the proof of Lemma \ref{lem:spectrum2}.
\end{proof}

Repeating the proof of \cite[Theorem 2.3]{MS} with \eqref{spectrum2}, we have the existence of $\CR$-bounded solution operator families for $\lambda \in \Sigma_\epsilon$.
\begin{thm}\label{thm:Rbdd RN}
Let $N \ge 2$, $a>0$, and $1 < q < \infty$.
Then
for any $\epsilon \in (\epsilon_0, \pi/2)$ with $\tan \epsilon_0 \ge |\beta|/\sqrt 2$,
there exist 
operator families 
\begin{align*}
&\CA_1 (\lambda) \in 
{\rm Hol} (\Sigma_{\epsilon}, 
\CL(L_q(\R^N)^N \times L_q(\R^N; \R^{N^3}), H^2_q(\R^N)^N))\\
&\CB_1 (\lambda) \in 
{\rm Hol} (\Sigma_{\epsilon}, 
\CL(L_q(\R^N)^N \times L_q(\R^N; \R^{N^3})\times L_q(\R^N; \bS_0), H^3_q(\R^N; \bS_0)))\\
&\CC_1 (\lambda) \in 
{\rm Hol} (\Sigma_{\epsilon}, 
\CL(L_q(\R^N)^N \times L_q(\R^N; \R^{N^3}), \widehat H^1_q(\R^N)))
\end{align*}
such that 
for any $\lambda = \gamma + i\tau \in \Sigma_{\epsilon}$
, $\bff \in L_q(\R^N)^N$, and $\bG \in H^1_q(\R^N; \bS_0)$, 
\begin{equation*}
\bu = \CA_1 (\lambda) (\bff, \nabla \bG), \quad
\bQ = \CB_1 (\lambda) (\bff, \nabla \bG, \bG), \quad
\fp = \CC_1(\lambda) (\bff, \nabla \bG)
\end{equation*}
are unique solutions of problem \eqref{r RN}
and 
\begin{align*}
&\CR_{\CL(L_q(\R^N) \times L_q(\R^N; \R^{N^3}), A_q(\R^N))}
(\{(\tau \pd_\tau)^n \CS_\lambda \CA_1 (\lambda) \mid 
\lambda \in \Sigma_{\epsilon}\}) 
\leq r, \\
&\CR_{\CL(L_q(\R^N) \times L_q(\R^N; \R^{N^3}) \times L_q(\R^N; \bS_0), B_q(\R^N))}
(\{(\tau \pd_\tau)^n \CT_\lambda \CB_1 (\lambda) \mid 
\lambda \in \Sigma_{\epsilon}\}) 
\leq r,\\
&\CR_{\CL(L_q(\R^N) \times L_q(\R^N; \R^{N^3}), L_q(\R^N)^N)}
(\{(\tau \pd_\tau)^n \nabla \CC_1 (\lambda) \mid 
\lambda \in \Sigma_{\epsilon}\}) 
\leq r
\end{align*}
for $n = 0, 1$,
where 
$\CS_\lambda$ and $\CT_\lambda$ are defined in Theorem \ref{thm:Rbdd H},
$A_q(\BR^N) = L_q(\BR^N)^{N^3 + N^2+N}$,
$B_q(\BR^N) = H^1_q(\BR^N; \BR^{N^4}) \times H^1_q(\BR^N; \BR^{N^3}) \times  H^1_q(\R^N; \bS_0) \times L_q(\R^N; \R^{N^3}) \times L_q(\R^N; \bS_0)$,
and $r$ is a constant independent of $\lambda$.
\end{thm}
\begin{remark}\label{rem:Rbdd RN}
According to \cite[proof of Theorem 2.3]{MS}, the solution formula for the $(j,k)$ component of the order parameter has the terms:
\begin{equation}\label{form RN}
\begin{aligned}
	&a\CF^{-1} \left[ \frac{\beta^2}{(\lambda+|\xi|^2+a)P_2(\xi, \lambda)}
	\left(i\xi_j\sum^N_{\ell, m=1} i\xi_\ell \widehat{G}_{k \ell}+i\xi_k\sum^N_{\ell, m=1} i\xi_\ell \widehat{G}_{j \ell}\right)\right]\\
	&+ \CF^{-1} \left[\frac{1}{\lambda + |\xi|^2 + a} \widehat{G}_{jk} \right].
\end{aligned}
\end{equation}
In view of the estimates for $(\lambda + |\xi|^2 +a)$ and $P_2(\xi, \lambda)$, $\lambda \CB_1(\lambda)$ and $\lambda^{1/2}  \CB_1(\lambda)$ are defined for $\bG \in H^1_q(\R^N; \bS_0)$, then the $\CR$-boundedness for these operators can be proved in $L_q(\R^N; \bS_0)$.

On the other hand, since \eqref{form RN} can be written as 
\begin{align*}
	&a\CF^{-1} \left[ \frac{\beta^2}{(\lambda+|\xi|^2+a)P_2(\xi, \lambda)}
	\left(i\xi_j\sum^N_{\ell, m=1} \widehat{\pd_\ell G}_{k \ell}+i\xi_k\sum^N_{\ell, m=1} \widehat{\pd_\ell G}_{j \ell}\right)\right]\\
	&-\sum^N_{\ell=1}\CF^{-1} \left[\frac{i\xi_\ell}{(\lambda + |\xi|^2 + a)|\xi|^2} \widehat{\pd_\ell G}_{jk} \right],
\end{align*}
the solution operator of the order parameter may be defined for $\bG \in \dot H^1_q(\R^N; \bS_0)$,
the $\CR$-boundedness is proved in $\dot H^1_q(\R^N; \bS_0)$.
In more detail, the following result holds: 
there exists 
an operator family
\[
\CB_1' (\lambda) \in 
{\rm Hol} (\Sigma_{\epsilon}, 
\CL(L_q(\R^N)^N \times L_q(\R^N; \R^{N^3}), H^3_q(\R^N; \bS_0)))
\]
such that 
for any $\lambda = \gamma + i\tau \in \Sigma_{\epsilon}$
, $\bff \in L_q(\R^N)^N$, and $\bG \in \dot H^1_q(\R^N; \bS_0)$, 
$\bQ = \CB_1' (\lambda) (\bff, \nabla \bG)$
satisfies the problem \eqref{r RN}
and 
\begin{equation}\label{rem:rbdd}
\CR_{\CL(L_q(\R^N) \times L_q(\R^N; \R^{N^3}), B_q'(\R^N))}
(\{(\tau \pd_\tau)^n (\CT_\lambda, 1) \CB_1' (\lambda) \mid 
\lambda \in \Sigma_{\epsilon}\}) 
\leq r
\end{equation}
for $n = 0, 1$,
where $B_q'(\BR^N) = \dot H^1_q(\BR^N; \R^{N^4}) \times \dot H^1_q(\R^N; \R^{N^3}) \times \dot H^1_q(\R^N; \bS_0) \times \dot H^1_q(\R^N; \bS_0) \times \dot H^1_q(\BR^N; \BR^{N^3}) \times \dot H^1_q(\R^N; \bS_0)$.
In particular, it holds that
\[
	\CR_{\CL(L_q(\R^N_+) \times L_q(\R^N_+; \R^{N^3}), \dot H^1_q(\BR^N; \R^{N^4}) \times \dot H^1_q(\R^N; \R^{N^3}) \times \dot H^1_q(\R^N; \bS_0))}
(\{(\tau \pd_\tau)^n \CS_\lambda \CB'_1 (\lambda) \mid 
\lambda \in \Sigma_{\epsilon, c_0}\})
\leq r.
\]
\end{remark}
\section{Lower bound of the Lopatinski determinant in the half space}

Set $A=|\xi'|$, $B_a=\sqrt{\lambda+a+A^2}$, 
and
\begin{equation}\label{root:1}
L_1=L_1(\lambda, \xi')=\sqrt{z_1(\lambda) + A^2}, \enskip L_2=L_2(\lambda, \xi')=\sqrt{z_2(\lambda) + A^2},
\end{equation}
where we have set
\begin{equation}\label{root:2}
\begin{aligned}
&z_1(\lambda)=\frac{\lambda}{1 + \beta^2/2} + \frac a2 +
i\frac{|\beta|/\sqrt 2}{1 + \beta^2/2}\sqrt{\lambda^2 - \frac{a^2(1+\beta^2/2)^2}{2\beta^2}},\\
&z_2(\lambda)=\frac{\lambda}{1 + \beta^2/2} + \frac a2 -
i\frac{|\beta|/\sqrt 2}{1 + \beta^2/2}\sqrt{\lambda^2 - \frac{a^2(1+\beta^2/2)^2}{2\beta^2}}.
\end{aligned}
\end{equation}
As we see in the the proof of Lemma \ref{lem:sol form} below, $-A$, $-B_a$, and $-L_j$ $(j=1, 2)$ are the characteristic roots for the velocity and the order parameter.
Furthermore, the solution formula of the resolvent problem \eqref{r0} with $(\bff, \bG)=(0, 0)$ includes the following terms:
\begin{align*}
\CC_a(\lambda, \xi') &= \frac{1}{\lambda} \left(\CI_1+\frac{\CI_2}{L_2-L_1}\right),\\
\CA_a(\lambda, \xi') &= B_a^3(L_1 + L_2) -A^2 B_a^2 - A^2 L_1 L_2\\
&=B_a(B_a^2-A^2)(L_1+L_2)-A^2(B_a-L_1)(B_a-L_2)
\end{align*}
with
\begin{align*}
\CI_1 & =\beta \frac{A^2}{B_a^2-A^2}\left\{2A^2- \frac{A(B_a^2+A^2)}{B_a}\right\},\\
\CI_2 & =-\beta \frac{L_1(L_2-A)}{B_a^2-L_1^2}\left\{2A^2L_1- \frac{(B_a^2+A^2)(L_1^2+A^2)}{2B_a}\right\}
+\beta \frac{L_2(L_1-A)}{B_a^2-L_2^2}\left\{2A^2L_2- \frac{(B_a^2+A^2)(L_2^2+A^2)}{2B_a}\right\}.
\end{align*}
In this section, we consider the lower bound for $\CC_a$ and $\CA_a$ for $\lambda \in \Sigma_{\epsilon, c_0}$, 
where $c_0$ is a sufficiently small constant determined in Lemma \ref{lem:zl} below.
In particular, we prove 
$|\lambda \CC_a(\lambda, \xi')| \ge C$
with some positive constant $C$, which is the essential estimate to prove the $\CR$-solvability in the case $\lambda$ near $0$.

\subsection{Estimates of $z_j(\lambda)$ and $L_j(\lambda, \xi')$ for $j=1, 2$}\label{zl}
In this subsection, we prove the following lemmas.
\begin{lem}\label{lem:l}
Let $a>0$.
Let $\epsilon \in (\epsilon_0, \pi/2)$ with $\tan \epsilon_0 \ge |\beta|/\sqrt 2$.
Then 
\begin{equation}\label{l neq 0}
L_j (\lambda, \xi')^2 \neq 0
\end{equation}
holds for any $\xi' \in \R^{N-1}$ and $\lambda\in \Sigma_{\epsilon}$. 
\end{lem}

\begin{proof}
Assume that $L_j(\eta, \xi')^2 = z_j(\eta) + A^2 = 0$, which implies that 
\[
\left(\frac{\eta}{1+\beta^2/2} + \frac{a}{2} + A^2 \right)^2 
= \left(\pm i\frac{|\beta|/\sqrt{2}}
{1+\beta^2/2}\sqrt{\eta^2-\frac{a^2(1+\beta^2/2)^2}{2\beta^2}}
\right)^2, 
\]
then
\[
\frac{1}{1+\beta^2/2}\left\{\eta^2 + 2 \left(\frac{a}{2}+A^2\right)\eta
+ (1+\beta^2/2)(A^4+aA^2)\right\}=0. 
\]
Thus, we have
\begin{align*}
\eta = -\left(A^2 + \frac{a}{2}\right) \pm \sqrt{a^2/4- (\beta^2/2)(A^4+aA^2)}
\end{align*}
Assume that $a^2/4- (\beta^2/2)(A^4+aA^2) \geq 0$.
Since $-(A^2 + a/2) \pm \sqrt{a^2/4- (\beta^2/2)(A^4+aA^2)}$
is non-positive real numbers for any $A \geq 0$, we have $\lambda \neq \eta$ for any $\lambda
\in \Sigma_\epsilon$, and so \eqref{l neq 0} holds.

On the other hand, assume that $a^2/4- (\beta^2/2)(A^4+aA^2) < 0$, 
which is $(\beta^2/2)(A^4+aA^2) - a^2/4>0$.
Set 
\[
-\left(A^2 + \frac{a}{2} \right) \pm i\sqrt{\frac{\beta^2}{2}(A^4+aA^2) - \frac{a^2}{4}} = \ell e^{i(\pi-\tau)}
\]
and then
\begin{align*}
\ell^2 &= \left(A^2+ \frac{a}{2}\right)^2 + \frac{\beta^2}{2}(A^4+aA^2)-\frac{a^2}{4}
= \left(1+\frac{\beta^2}{2}\right)(A^4+aA^2),  \\
\tan \tau & = \frac{\sqrt{(\beta^2/2)(A^4+aA^2) - a^2/4}}{A^2+a/2}.
\end{align*}
Note that 
\[
\frac{\beta^2}{2}(A^4+aA^2) -\frac{a^2}{4} 
= \frac{\beta^2}{2}\left(A^2+ \frac{a}{2}\right)^2 - \left(\frac{\beta^2}{2}+1\right)\frac{a^2}{4}
= \frac{\beta^2}{2}\left\{\left(A^2+\frac{a}{2}\right)^2 - \left(1+\frac{2}{\beta^2}\right)\frac{a^2}{4}\right\}.
\]
Set $m = A^2+a/2$ and $n = (1+2/\beta^2)(a^2/4)$, we have
\[
\tan \tau = \frac{|\beta|}{\sqrt{2}}\frac{\sqrt{m^2-n}}{m}
\quad (m \geq a/2).
\] 
Since 
\[
\frac{d}{dm} \frac{\sqrt{m^2-n}}{m}
=\frac{n}{m^2\sqrt{m^2-n}} > 0,
\]
$\tan \tau$ is monotonically increasing as $A \to \infty$, and so
\[
0 \leq \tan \tau \leq \lim_{m\to\infty} \frac{|\beta|}{\sqrt{2}}
\frac{\sqrt{m^2-n}}{m} = \frac{|\beta|}{\sqrt{2}}.
\]
Therefore, we obtain $\lambda \neq \ell e^{i(\pi-\tau)}$ for $\lambda \in \Sigma_\epsilon$, which completes the proof of Lemma \ref{lem:l}.

\end{proof}

\begin{lem}\label{lem:zl}
Let $a>0$.
Let $\epsilon \in (\epsilon_0, \pi/2)$ with $\tan \epsilon_0 \ge |\beta|/\sqrt 2$.
Then there exists a small constant $c_0$ depending on $\epsilon$, $\beta$, and $a$ such that 
\begin{enumerate}
\item\label{z}
\[
\begin{aligned}
c_\beta |\lambda| \le &|z_1(\lambda)| \le C_{\beta, a} |\lambda|\\
c_{\beta, \epsilon, a} (|\lambda|+1) \le &|z_2(\lambda)| \le C_{\beta, a} (|\lambda|+1)
\end{aligned}
\]
\item\label{est:l}
\[
\begin{aligned}
c_{\beta, \epsilon} (|\lambda|+A^2)^{1/2} \le \Re L_1 (\lambda, \xi') \le & |L_1 (\lambda, \xi')| \le C_{\beta, \epsilon}(|\lambda|+A^2)^{1/2}\\
c_{\beta, \epsilon, a} (|\lambda|+1+A^2)^{1/2} \le \Re L_2 (\lambda, \xi') \le & |L_2 (\lambda, \xi')| \le C_{\beta, \epsilon, a}(|\lambda|+1+A^2)^{1/2}
\end{aligned}
\]
\end{enumerate}
for any $\xi' \in \R^{N-1} \setminus \{0\}$ and $\lambda \in \Sigma_{\epsilon, c_0}$,
where $C_{\kappa}$ and $c_{\kappa}$ are positive constants depending on $\kappa$.
\end{lem}

\begin{proof}
First, we prove the first assertion.
Assume that $\lambda \in \Sigma_\epsilon$ satisfies $|\lambda|\le c_1a(1+\beta^2/2)/(\sqrt{2}|\beta|)$, where $c_1$ is a sufficiently small constant.
This assumption implies that $\Bigl|\lambda^2/\dfrac{a^2(1+\beta^2/2)^2}{2\beta^2}\Bigr|\le c_1^2$.
In this case,
\begin{equation}\label{ez} 
\begin{aligned}
z_1(\lambda) & = \frac{\lambda}{1+\beta^2/2} + \frac{a}{2}
-\frac{|\beta|/\sqrt{2}}{1+\beta^2/2}\sqrt{\frac{a^2(1+\beta^2/2)^2}{2\beta^2}}
\sqrt{1-\lambda^2/\frac{a^2(1+\beta^2/2)^2}{2\beta^2}}\\
& = \frac{\lambda}{1+\beta^2/2} + O(\lambda^2), \\
z_2(\lambda) & = \frac{\lambda}{1+\beta^2/2} + \frac{a}{2}
+\frac{|\beta|/\sqrt{2}}{1+\beta^2/2}\sqrt{\frac{a^2(1+\beta^2/2)^2}{2\beta^2}}
\sqrt{1-\lambda^2/\frac{a^2(1+\beta^2/2)^2}{2\beta^2}}\\
& = \frac{\lambda}{1+\beta^2/2} + a + O(\lambda^2).
\end{aligned}
\end{equation}
The above expansion and Lemma \ref{spectrum} furnishes that
there exists a constant $D$ such that
\begin{align*}
|z_1(\lambda)| \ge \frac{|\lambda|}{1+\beta^2/2} - D|\lambda|^2,
\quad
|z_2(\lambda)| \ge c_\epsilon \left(\frac{|\lambda|}{1+\beta^2/2} + a\right) - D|\lambda|^2, 
\end{align*}
where $c_\epsilon=\sin (\epsilon/2)$, then
if $|\lambda|$ is small such that $|\lambda| \leq c_0$ with $c_0=\min\{c_1a(1+\beta^2/2)/(\sqrt{2}|\beta|), c_2\}$, 
$c_2=\dfrac{1}{2D}\dfrac{c_\epsilon}{1+\beta^2/2}$, 
we have \eqref{z}.

Next, we prove the second assertion.
Thanks to \eqref{z}, it is sufficient to consider the lower bound of $\Re L_j(\lambda, \xi')$ with $j=1, 2$.
Set $L_j(\lambda, \xi')=|L_j(\lambda, \xi')|e^{i\theta}$, then $\Re L_j(\lambda, \xi')=|L_j(\lambda, \xi')|\cos \theta$.
Lemma \ref{lem:l} implies that $z_j(\lambda) \neq -A^2$ for any $\xi' \in \R^{N-1}$ and $\lambda \in \Sigma_\epsilon$,
thus we observe that $z_j(\lambda) \neq 0$ and $z_j(\lambda)$ does not exist on the negative real axis, namely,
there exists $\delta \in (0, \pi/2)$ such that $z_j(\lambda) \in \Sigma_\delta$ for any $\lambda \in \Sigma_\epsilon$.
Therefore, we have $L_j (\lambda, \xi')^2 \in \Sigma_\delta$, which implies that
$|\arg L_j (\lambda, \xi')| \le \pi/2-\delta/2$, then we observe that
$\cos \theta \ge \cos (\pi/2-\delta/2) = \sin \delta/2>0$.
Furthermore, in the same manner as \eqref{z},
we have
\begin{align*}
|L_1(\lambda, \xi')|^2& = |L_1(\lambda, \xi')^2| = |z_1(\lambda) + A^2| \geq c_\epsilon \Bigl(\frac{|\lambda|}{1+\beta^2/2}+A^2\Bigr)
-D|\lambda|^2 \nonumber \\
&\ge \Bigl(\frac{c_\epsilon}{1+\beta^2/2}-D|\lambda|\Bigr)|\lambda| + c_\epsilon A^2 \nonumber \\
&\ge \frac{c_\epsilon}{2(1+\beta^2/2)}(|\lambda| + A^2),
\\
|L_2(\lambda, \xi')|^2&=|z_2(\lambda) + A^2| \geq \frac{c_{\epsilon, a}}{2(1+\beta^2/2)}(|\lambda| + 1 + A^2)
\end{align*}
provided that $|\lambda| \leq c_0$.
Therefore, there exist positive constants $c_{\beta, \epsilon}$ and $c_{\beta, \epsilon, a}$ such that 
\[
\frac{\Re L_1 (\lambda, \xi')}{(|\lambda|+A^2)^{1/2}} \ge (\sin \delta/2) c_{\beta, \epsilon}, \quad
\frac{\Re L_2 (\lambda, \xi')}{(|\lambda|+1+A^2)^{1/2}} \ge (\sin \delta/2) c_{\beta, \epsilon, a} 
\]
for $\lambda \in \Sigma_{\epsilon, c_0}$, which completes the proof of \eqref{est:l}.
\end{proof}

\subsection{Lower bound of $\CD(\lambda, \xi')$}
In this subsection, we consider the lower bound of $\CD(\lambda, \xi')$ to obtain the estimates of $\CC_a(\lambda, \xi')$ in Lemma \ref{lem:bound ca}. 
Since the bound of $\CC_a(\lambda, \xi')$ for the middle frequency of $|\xi'|$ can be proved by  Lemma \ref{lem:behav ca}, we only consider the estimates of $\CD(\lambda, \xi')$ for the low and high frequencies of $|\xi'|$.

Set
\begin{align*}
\CD(\lambda, \xi')&=2B_a(B_a^2-L_1^2)(B_a^2-L_2^2)\left(\CI_1+\frac{\CI_2}{L_2-L_1}\right)\\
&=\CD_1(\lambda, \xi')+\CD_2(\lambda, \xi')+\CD_3(\lambda, \xi')
\end{align*}
with
\begin{align*}
\CD_1(\lambda, \xi')&=\frac{2\beta(B_a^2-L_1^2)(B_a^2-L_2^2)A^3\{2AB_a - (B_a^2+A^2)\}}{B_a^2-A^2},\\
\CD_2(\lambda, \xi')&=4\beta A^2 B_a \{L_1L_2(B_a^2+L_1L_2) - AB_a^2(L_1+L_2)\},\\
\CD_3(\lambda, \xi')&=-\beta(B_a^2+A^2)\{(B_a^2+A^2)L_1L_2(L_1+L_2)- AB_a^2 (L_1^2+L_1L_2+L_2^2+A^2) + AL_1L_2(L_1L_2-A^2)\}.
\end{align*}
 Let us consider the lower bound of $\CD(\lambda, \xi')$ for $\lambda \in \Sigma_{\epsilon, c_0}$.
\subsubsection{Case : $\dfrac{c_0+a}{r} \le A^2$.}\label{large}
We firstly consider the case: $(c_0+a)/r \le A^2$
with small $r$, which implies that 
$A^2 \ge (|\lambda| +a)/r \ge \max\{|\lambda +a|/r, |\lambda|/r\}$. 
In this case, by Lemma \ref{lem:zl} \eqref{z} 
\begin{equation}\label{prepare}
\frac{|z_j(\lambda)|}{A^2} 
\le C_\beta \frac{|\lambda|+a}{A^2} 
\le C_\beta r
\end{equation}
with $j=1, 2$.
Note that
\begin{equation}\label{e}
\begin{aligned}
&\sqrt{t^2+1}=1+\frac{t^2}{2}-\frac{1}{8}t^4+\frac{1}{16}t^6+t^6O(|t|^2)
\end{aligned}
\end{equation}
as $t \to 0$.
By \eqref{e} and \eqref{prepare}, 
\begin{align}
B_a&=A\sqrt{\frac{\lambda+a}{A^2}+1} \nonumber\\
&=A\left(1+\frac{1}{2}\frac{\lambda+a}{A^2}-\frac{1}{8}\left(\frac{\lambda+a}{A^2}\right)^2
+\frac{1}{16}\left(\frac{\lambda+a}{A^2}\right)^3+\left(\frac{\lambda+a}{A^2}\right)^3O(|t|^2)\right)\nonumber\\
&=A+\frac{\lambda+a}{2A}-\frac{1}{8}\frac{(\lambda+a)^2}{A^3}
+\frac{1}{16}\frac{(\lambda+a)^3}{A^5}+\frac{(\lambda+a)^3}{A^5}O(|t|^2), \label{b}\\
L_j(\lambda, \xi')&=A\sqrt{\frac{z_j(\lambda)}{A^2}+1}\nonumber\\
&=A\left(1+\frac{1}{2}\frac{z_j(\lambda)}{A^2}
-\frac{1}{8}\left(\frac{z_j(\lambda)}{A^2}\right)^2
+\frac{1}{16}\left(\frac{z_j(\lambda)}{A^2}\right)^3
+\left(\frac{z_j(\lambda)}{A^2}\right)^3O(|t|^2)\right)\nonumber\\
&=A+\frac{z_j(\lambda)}{2A}-\frac{1}{8}\frac{z_j(\lambda)^2}{A^3}
+\frac{1}{16}\frac{z_j(\lambda)^3}{A^5}+\frac{z_j(\lambda)^3}{A^5}O(|t|^2), \label{l}
\end{align}
where $t^2=(\lambda+a)/A^2$.

\noindent
\underline{Behavior of $\CD_1(\lambda, \xi')$}
By \eqref{b},
\begin{align*}
2AB_a-(B_a^2+A^2)
&=2A\left(A+\frac{\lambda+a}{2A}-\frac{1}{8}\frac{(\lambda+a)^2}{A^3}+\frac{(\lambda+a)^2}{A^3}O(|t|^2)\right)-(\lambda+a+2A^2)\\
&=-\frac 14 \frac{(\lambda+a)^2}{A^2}+\frac{(\lambda+a)^2}{A^2}O(|t|^2).
\end{align*}
Thus,
\begin{equation}\label{d1}
\begin{aligned}
\CD_1(\lambda, \xi')&=\frac{2\beta(\lambda+a-z_1(\lambda))(\lambda+a-z_2(\lambda))A^3}{\lambda+a}
\left\{-\frac 14 \frac{(\lambda+a)^2}{A^2}+\frac{(\lambda+a)^2}{A^2}O(|t|^2)\right\}\\
&=-\frac{\beta}{2}(\lambda+a-z_1(\lambda))(\lambda+a-z_2(\lambda))(\lambda+a)(A+A O(|t|^2)).
\end{aligned}
\end{equation}

\noindent
\underline{Behavior of $\CD_2(\lambda, \xi')$}
By \eqref{l},
\begin{equation}\label{l1l2}
\begin{aligned}
L_1L_2&=\left(A+\frac{z_1(\lambda)}{2A}-\frac{1}{8}\frac{z_1(\lambda)^2}{A^3}
+\frac{1}{16}\frac{z_1(\lambda)^3}{A^5}
+\frac{z_1(\lambda)^3}{A^5}O(|t|^2)\right)\\
& \enskip \times 
\left(A+\frac{z_2(\lambda)}{2A}-\frac{1}{8}\frac{z_2(\lambda)^2}{A^3}
+\frac{1}{16}\frac{z_2(\lambda)^3}{A^5}+\frac{z_2(\lambda)^3}{A^5}O(|t|^2)\right)\\
&=A^2+\frac 12 (z_1(\lambda)+z_2(\lambda))-\frac{1}{8 A^2}(z_1(\lambda)-z_2(\lambda))^2\\
&\enskip +\frac{1}{16}(z_1(\lambda)+z_2(\lambda))(z_1(\lambda)-z_2(\lambda))^2\left(\frac{1}{A^4}+\frac{1}{A^4}O(|t|^2)\right).
\end{aligned}
\end{equation}
By \eqref{l1l2},
\begin{equation}\label{d2-1}
\begin{aligned}
&L_1L_2(B_a^2+L_1 L_2)\\
&=\Bigl\{A^2+\frac 12 (z_1(\lambda)+z_2(\lambda))-\frac{1}{8 A^2}(z_1(\lambda)-z_2(\lambda))^2\\
&\quad +\frac{1}{16}(z_1(\lambda)+z_2(\lambda))(z_1(\lambda)-z_2(\lambda))^2
\left(\frac{1}{A^4}+\frac{1}{A^4}O(|t|^2)\right)\Bigr\}\\
&\enskip \times \Bigl\{2A^2+\lambda+a+\frac 12 (z_1(\lambda)+z_2(\lambda))-\frac{1}{8 A^2}(z_1(\lambda)-z_2(\lambda))^2\\
&\quad +\frac{1}{16}(z_1(\lambda)+z_2(\lambda))(z_1(\lambda)-z_2(\lambda))^2
\left(\frac{1}{A^4}+\frac{1}{A^4}O(|t|^2)\right)\Bigr\}\\
&=2A^4+\left\{\lambda+a+\frac{3}{2}(z_1(\lambda)+z_2(\lambda))\right\}A^2\\
&\enskip -\frac 38 (z_1(\lambda)-z_2(\lambda))^2+\frac 12(z_1(\lambda)+z_2(\lambda))(\lambda+a) +\frac 14(z_1(\lambda)+z_2(\lambda))^2\\
&\enskip +\Bigl\{\frac{1}{16}(z_1(\lambda)+z_2(\lambda))(z_1(\lambda)-z_2(\lambda))^2\\
&\quad -\frac{1}{8}(z_1(\lambda)-z_2(\lambda))^2(\lambda+a)\Bigr\}
\left(\frac{1}{A^2}+\frac{1}{A^2}O(|t|^2)\right)
\end{aligned}
\end{equation}
By \eqref{l},
\begin{equation}\label{d2-2}
\begin{aligned}
&AB_a^2(L_1+L_2)\\
&=A(\lambda +a+A^2)\Bigl\{2A+\frac{1}{2A} (z_1(\lambda)+z_2(\lambda))
-\frac{1}{8A^3}(z_1(\lambda)^2+z_2(\lambda)^2)\\
&\enskip +\frac{1}{16}(z_1(\lambda)^3+z_2(\lambda)^3)\left(\frac{1}{A^5}+\frac{1}{A^5}O(|t|^2)\right)\Bigr\}\\
&=2A^4+\left\{\frac 12 (z_1(\lambda)+z_2(\lambda))+2(\lambda+a)\right\}A^2\\
&\enskip +\frac 12(z_1(\lambda)+z_2(\lambda))(\lambda+a)-\frac{1}{8}(z_1(\lambda)^2+z_2(\lambda)^2)\\
&\enskip +\left\{\frac{1}{16}(z_1(\lambda)^3+z_2(\lambda)^3)-\frac{1}{8}(\lambda+a)(z_1(\lambda)^2+z_2(\lambda)^2)\right\}
\left(\frac{1}{A^2}+\frac{1}{A^2}O(|t|^2)\right).
\end{aligned}
\end{equation}
By \eqref{d2-1} and \eqref{d2-2},
\[
\begin{aligned}
&L_1 L_2 (B_a^2+L_1 L_2)-AB_a^2(L_1+L_2)\\
&=\{z_1(\lambda)+z_2(\lambda)-(\lambda+a)\}A^2+\frac 54 z_1(\lambda)z_2(\lambda)\\
&\enskip +\frac{1}{16}z_1(\lambda)z_2(\lambda)
\left\{4(\lambda+a)-(z_1(\lambda)+z_2(\lambda))\right\}
\left(\frac{1}{A^2}+\frac{1}{A^2}O(|t|^2)\right).
\end{aligned}
\]
Thus,
\begin{equation}\label{d2-3}
\begin{aligned}
\CD_2(\lambda, \xi')&=4\beta A^2
\left\{A+\frac{\lambda+a}{2A}-\frac{1}{8}\frac{(\lambda+a)^2}{A^3}
+\frac{1}{16}\frac{(\lambda+a)^3}{A^5}+\frac{(\lambda+a)^3}{A^5}O(|t|^2)\right\}\\
&\enskip \times \Bigl[\{z_1(\lambda)+z_2(\lambda)-(\lambda+a)\}A^2+\frac 54 z_1(\lambda)z_2(\lambda)\\
&\enskip +\frac{1}{16}z_1(\lambda)z_2(\lambda)
\left\{4(\lambda+a)-(z_1(\lambda)+z_2(\lambda))\Bigr\}
\left(\frac{1}{A^2}+\frac{1}{A^2}O(|t|^2)\right)\right]\\
&=4\beta \{z_1(\lambda)+z_2(\lambda)-(\lambda+a)\}A^5\\
&\enskip +4\beta \left\{\frac 54 z_1(\lambda)z_2(\lambda) + \frac 12 (z_1(\lambda)+z_2(\lambda) )(\lambda + a)- \frac 12 (\lambda+a)^2\right\}A^3\\
&\enskip +4\beta\Bigl[\frac{1}{16}z_1(\lambda)z_2(\lambda)
\left\{14(\lambda+a)-(z_1(\lambda)+z_2(\lambda))\right\}\\
&\enskip -\frac 18\{z_1(\lambda)+z_2(\lambda)-(\lambda+a)\}(\lambda+a)^2\Bigr]\left(A+AO(|t|^2)\right).
\end{aligned}
\end{equation}

\noindent
\underline{Behavior of $\CD_3(\lambda, \xi')$}
By \eqref{l} and \eqref{l1l2},
\begin{align}
&(B_a^2+A^2)L_1L_2(L_1+L_2) \nonumber\\
&=(2A^2+\lambda+a)
\Bigl\{A^2+\frac 12 (z_1(\lambda)+z_2(\lambda))-\frac{1}{8 A^2}(z_1(\lambda)-z_2(\lambda))^2 \nonumber\\
&\enskip +\frac{1}{16}(z_1(\lambda)+z_2(\lambda))
(z_1(\lambda)-z_2(\lambda))^2\left(\frac{1}{A^4}+\frac{1}{A^4}O(|t|^2)\right)\Bigr\} \nonumber\\
&\enskip \times
\left\{2A+\frac{1}{2A} (z_1(\lambda)+z_2(\lambda))
-\frac{1}{8A^3}(z_1(\lambda)^2+z_2(\lambda)^2)
+\frac{1}{16}(z_1(\lambda)^3+z_2(\lambda)^3)\left(\frac{1}{A^5}+\frac{1}{A^5}O(|t|^2)\right)\right\} \nonumber\\
&=(2A^2+\lambda+a) \nonumber\\
&\enskip \times
\Bigl[2A^3+\frac 32 (z_1(\lambda)+z_2(\lambda))A
-\left\{\frac{1}{8}(z_1(\lambda)^2+z_2(\lambda)^2)-\frac{1}{4}(z_1(\lambda)+z_2(\lambda))^2
+\frac{1}{4}(z_1(\lambda)-z_2(\lambda))^2\right\}\frac{1}{A} \nonumber\\
&\enskip +\Bigl\{\frac{1}{16}(z_1(\lambda)^3+z_2(\lambda)^3)
-\frac{1}{16}(z_1(\lambda)+z_2(\lambda))(z_1(\lambda)^2+z_2(\lambda)^2)
+\frac{1}{16}(z_1(\lambda)+z_2(\lambda))(z_1(\lambda)-z_2(\lambda))^2\Bigr\} \nonumber\\
&\quad \times \left(\frac{1}{A^3}+\frac{1}{A^3}O(|t|^2)\right)\Bigr] \nonumber\\
&=4A^5+\left\{3(z_1(\lambda)+z_2(\lambda))+2(\lambda+a)\right\}A^3 \nonumber\\
&\enskip+\left\{2z_1(\lambda)z_2(\lambda)-\frac 14 (z_1(\lambda)^2+z_2(\lambda)^2)+\frac 32 (\lambda+a)(z_1(\lambda)+z_2(\lambda))\right\}A \nonumber\\
&\enskip +\Bigl[\frac{1}{8}(z_1(\lambda)+z_2(\lambda))\{(z_1(\lambda)-z_2(\lambda))^2-z_1(\lambda)z_2(\lambda)\} \nonumber\\
&\enskip-\frac 18(z_1(\lambda)^2+z_2(\lambda)^2)(\lambda+a)
+z_1(\lambda)z_2(\lambda)(\lambda+a)\Bigr] \nonumber\\
&\quad \times \left(\frac{1}{A}+\frac{1}{A}O(|t|^2)\right).\label{d3-1}
\end{align}
By \eqref{l1l2},
\begin{equation}\label{d3-2}
\begin{aligned}
&AB_a^2 (L_1^2+L_1L_2+L_2^2+A^2)\\
&=A(\lambda+a+A^2)\Bigl\{4A^2+\frac 32 (z_1(\lambda)+z_2(\lambda))-\frac{1}{8A^2} (z_1(\lambda)-z_2(\lambda))^2\\
&\enskip +\frac{1}{16}(z_1(\lambda)+z_2(\lambda)) (z_1(\lambda)-z_2(\lambda))^2\left(\frac{1}{A^4}+\frac{1}{A^4}O(|t|^2)\right)\Bigr\}\\
&=4A^5+\left\{\frac 32 (z_1(\lambda)+z_2(\lambda))+4(\lambda+a)\right\}A^3\\
&\enskip +\left\{\frac 32(z_1(\lambda)+z_2(\lambda))(\lambda+a)
-\frac 18(z_1(\lambda)-z_2(\lambda))^2\right\}A\\
&\enskip +\frac{1}{16}(z_1(\lambda)-z_2(\lambda))^2\{z_1(\lambda)+z_2(\lambda)-2(\lambda+a)\}\left(\frac{1}{A}+\frac{1}{A}O(|t|^2)\right).
\end{aligned}
\end{equation}
By \eqref{l1l2},
\begin{align}
&AL_1L_2(L_1L_2-A^2) \nonumber\\
&=\Bigl\{A^3+\frac{A}{2} (z_1(\lambda)+z_2(\lambda))-\frac{1}{8 A}(z_1(\lambda)-z_2(\lambda))^2
 \nonumber\\
&\enskip +\frac{1}{16}(z_1(\lambda)+z_2(\lambda)) (z_1(\lambda)-z_2(\lambda))^2
\left(\frac{1}{A^3}+\frac{1}{A^3}O(|t|^2)\right)\Bigr\} \nonumber\\
&\enskip \times \left\{\frac 12 (z_1(\lambda)+z_2(\lambda))-\frac{1}{8 A^2}(z_1(\lambda)-z_2(\lambda))^2
+\frac{1}{16}(z_1(\lambda)+z_2(\lambda))(z_1(\lambda)-z_2(\lambda))^2\left(\frac{1}{A^4}+\frac{1}{A^4}O(|t|^2)\right)\right\} \nonumber\\
&=\frac 12 (z_1(\lambda)+z_2(\lambda))A^3-\frac 18(z_1(\lambda)-z_2(\lambda))^2A
+\frac 14(z_1(\lambda)+z_2(\lambda))^2A \nonumber\\
&\enskip -\frac{1}{16}(z_1(\lambda)+z_2(\lambda))(z_1(\lambda)-z_2(\lambda))^2\left(\frac{1}{A}+\frac{1}{A}O(|t|^2)\right).\label{d3-3}
\end{align}
Thus, by \eqref{d3-1}, \eqref{d3-2}, and \eqref{d3-3},
\begin{equation}\label{d3-4}
\begin{aligned}
\CD_3(\lambda, \xi')&=-\beta(\lambda+a+2A^2)
\Bigl[2\{(z_1(\lambda)+z_2(\lambda))-(\lambda+a)\}A^3
+\frac 52 z_1(\lambda)z_2(\lambda)A\\
&\enskip +\left\{\frac{5}{4}z_1(\lambda)z_2(\lambda) (\lambda +a)-\frac 18 (z_1(\lambda)+z_2(\lambda))z_1(\lambda)z_2(\lambda)\right\}
\left(\frac{1}{A}+\frac{1}{A} O(|t|^2)\right)\Bigr]\\
&=-4\beta\{(z_1(\lambda)+z_2(\lambda))-(\lambda+a)\}A^5\\
&\enskip-\beta\left[2\{(z_1(\lambda)+z_2(\lambda))-(\lambda+a)\}(\lambda+a)+5z_1(\lambda)z_2(\lambda)\right]A^3\\
&\enskip -\beta \left\{5z_1(\lambda)z_2(\lambda) (\lambda +a)-\frac 14 (z_1(\lambda)+z_2(\lambda))z_1(\lambda)z_2(\lambda)\right\}(A+AO(|t|^2).
\end{aligned}
\end{equation}
Therefore,
by \eqref{d1}, \eqref{d2-3}, and \eqref{d3-4},
\begin{equation}\label{d}
\begin{aligned}
\CD(\lambda, \xi')
&=\Bigl[-\frac{\beta}{2}(\lambda+a-z_1(\lambda))(\lambda+a-z_2(\lambda))(\lambda+a)\\
&\enskip +\frac{\beta}{4}z_1(\lambda)z_2(\lambda)
\left\{14(\lambda+a)-(z_1(\lambda)+z_2(\lambda))\right\}
-\frac{\beta}{2} \{z_1(\lambda)+z_2(\lambda)-(\lambda+a)\}(\lambda+a)^2\\
&\enskip -\beta \left\{5z_1(\lambda)z_2(\lambda) (\lambda +a)-\frac 14 (z_1(\lambda)+z_2(\lambda))z_1(\lambda)z_2(\lambda)\right\}\Bigr]
(A+AO(|t|^2)\\
&=-2 \beta z_1(\lambda)z_2(\lambda)(\lambda+a)(A+AO(|t|^2)),
\end{aligned}
\end{equation}
which combined with Lemma \ref{lem:zl} \eqref{z}, we have
\begin{equation}\label{ld}
|\CD(\lambda, \xi')| \ge C_{\beta, \epsilon, a}|\lambda|(|\lambda|+1)^2A.
\end{equation}

\subsubsection{Case : $A^2/|\lambda| \le \dfrac{1}{R}$}
Next, we consider the behavior of $\CD(\lambda, \xi')$ in the case that $A^2/|\lambda| \le 1/R$ with large $R$.
By the expansion:
\begin{equation*}\label{e'}
\begin{aligned}
&\sqrt{t^2+1}=1+O(|t|^2)
\end{aligned}
\end{equation*}
as $t \to 0$ and Lemma \ref{lem:zl} \eqref{z}, we have
\begin{align}
	B_a&=(\lambda+a)^{1/2}\sqrt{1+\frac{A^2}{\lambda+a}}=(\lambda+a)^{1/2}(1+O(|t|^2)),\label{b'}\\
	L_j(\lambda, \xi')&=z_j(\lambda)^{1/2}\sqrt{1+\frac{A^2}{z_j(\lambda)}}=z_j(\lambda)^{1/2}(1+O(|t|^2)),\label{l'}
\end{align}
where $t^2=A^2/\lambda$.
Let us consider the behavior of $\CD_1(\lambda, \xi')$, $\CD_2(\lambda, \xi')$, and $\CD_3(\lambda, \xi')$ by \eqref{b'} and \eqref{l'}.

\noindent
\underline{Behavior of $\CD_1(\lambda, \xi')$}
Since
\[
	2AB_a-(B_a^2+A^2)
	=2A(\lambda+a)^{1/2}(1+O(|t|^2))-(\lambda+a+2A^2),
\]
we have
\begin{equation}\label{d1'}
\begin{aligned}
	&\CD_1(\lambda, \xi')\\
	&=2\beta (\lambda+a-z_1(\lambda))(\lambda+a-z_2(\lambda))A^3
	\left\{-1+\frac{2A}{(\lambda+a)^{1/2}}-\frac{2A^2}{\lambda+a}+\frac{2A}{(\lambda+a)^{1/2}}O(|t|^2)\right\}\\
	&=2\beta (\lambda+a-z_1(\lambda))(\lambda+a-z_2(\lambda))A^3
	(-1 + O(|t|)).
\end{aligned}
\end{equation}

\noindent
\underline{Behavior of $\CD_2(\lambda, \xi')$}
Note that 
\begin{equation}\label{l1l2'}
	L_1 L_2=z_1(\lambda)^{1/2} z_2(\lambda)^{1/2} (1+O(|t|^2)).
\end{equation}
Then it holds that
\[
\begin{aligned}
	&L_1 L_2(B_a^2+L_1 L_2)-AB_a^2(L_1 + L_2)\\
	&= \{z_1(\lambda)^{1/2} z_2(\lambda)^{1/2}(\lambda+A^2+a+z_1(\lambda)^{1/2} z_2(\lambda)^{1/2})
	-A(\lambda+A^2+a)(z_1(\lambda)^{1/2} + z_2(\lambda)^{1/2})\}
	(1+O(|t|^2)).
\end{aligned}
\]
Thus,
\begin{align}
	\CD_2(\lambda, \xi')&=4\beta A^2(\lambda+a)^{1/2} \nonumber\\
	&\enskip \times \{z_1(\lambda)^{1/2} z_2(\lambda)^{1/2}(\lambda+A^2+a+z_1(\lambda)^{1/2} z_2(\lambda)^{1/2})
	-A(\lambda+A^2+a)(z_1(\lambda)^{1/2} + z_2(\lambda)^{1/2})\} \nonumber\\
	&\enskip \times (1+O(|t|^2)) \nonumber\\
	&=4\beta A^2(\lambda+a)^{1/2}z_1(\lambda)^{1/2} z_2(\lambda)^{1/2}(\lambda+a+z_1(\lambda)^{1/2} z_2(\lambda)^{1/2})
	(1+O(|t|)).\label{d2'}
\end{align}

\noindent
\underline{Behavior of $\CD_3(\lambda, \xi')$}
By \eqref{l'} and \eqref{l1l2'}, we have
\begin{equation*}\label{d3-1'}
\begin{aligned}
	&(B_a^2+A^2)L_1 L_2 (L_1 + L_2)
	=(\lambda+a+2A^2)z_1(\lambda)^{1/2} z_2(\lambda)^{1/2}(z_1(\lambda)^{1/2} + z_2(\lambda)^{1/2})
	(1+O(|t|^2)),\\
	&AB_a^2 (L_1^2+L_1 L_2 + L_2^2+A^2)\\
	&=A(\lambda+A^2+a)(z_1(\lambda) + z_1(\lambda)^{1/2}z_2(\lambda)^{1/2} + z_2(\lambda))(1+O(|t|^2))+A^3(\lambda + A^2+a),\\
	&AL_1(\lambda, \xi')L_2(\lambda, \xi')(L_1(\lambda, \xi')L_2(\lambda, \xi')-A^2)\\
	&=A z_1(\lambda) z_2(\lambda)(1+O(|t|^2))
	-A^3z_1(\lambda)^{1/2} z_2(\lambda)^{1/2}
	(1+O(|t|^2)).
\end{aligned}
\end{equation*}
These expansion furnish that
\begin{equation}\label{d3-4'}
\CD_3(\lambda, \xi')=-\beta(\lambda+a)^2
z_1(\lambda)^{1/2} z_2(\lambda)^{1/2}
(z_1(\lambda)^{1/2} + z_2(\lambda)^{1/2})
(1+O(|t|)).
\end{equation}
Summing up
\eqref{d1'}, \eqref{d2'}, and \eqref{d3-4'}, we observe that
\begin{equation*}\label{d'}
	\CD(\lambda, \xi')=-\beta(\lambda+a)^2
	z_1(\lambda)^{1/2} z_2(\lambda)^{1/2}
	(z_1(\lambda)^{1/2} + z_2(\lambda)^{1/2})
	(1+O(|t|)).
\end{equation*}

Now we consider the lower bound of $\CD$.
Lemma \ref{spectrum} and Lemma \ref{lem:zl} \eqref{z} furnish that
\[
	|\CD(\lambda, \xi')| \ge C_{\beta, \epsilon, a} (|\lambda|+1)^2|\lambda|^{1/2}(|\lambda|+1)^{1/2}|z_1(\lambda)^{1/2} + z_2(\lambda)^{1/2}|.
\]
Here, by \eqref{root:2} it holds that
\begin{align*}
	z_1(\lambda)^{1/2} + z_2(\lambda)^{1/2} &= \frac{z_1(\lambda) - z_2(\lambda)}{z_1(\lambda)^{1/2} - z_2(\lambda)^{1/2}} 
	= \frac{\sqrt 2i |\beta|}{1+\beta^2/2} \sqrt{\lambda^2 - \frac{a^2(1+\beta^2/2)^2}{2\beta^2}}\frac{1}{z_1(\lambda)^{1/2} - z_2(\lambda)^{1/2}}\\
	&= -a\sqrt{1-\frac{2\beta^2}{a^2(1+\beta^2/2)}\lambda^2}\frac{1}{z_1(\lambda)^{1/2} - z_2(\lambda)^{1/2}}\\
	&= \frac{-a}{z_1(\lambda)^{1/2} - z_2(\lambda)^{1/2}}(1+O(|\lambda|^2)),
\end{align*}
together with Lemma \ref{lem:zl} \eqref{z}, we have
\[
	|z_1(\lambda)^{1/2} + z_2(\lambda)^{1/2}| \ge C_{\beta, \epsilon, a}
\]
for $|\lambda| \le c_0$.
Therefore, we obtain 
\begin{equation}\label{sd}
\begin{aligned}
	|\CD(\lambda, \xi')| & \ge C_{\beta, \epsilon, a} (|\lambda|+1)^2|\lambda|^{1/2}(|\lambda|+1)^{1/2} \ge C_{\beta, \epsilon, a}|\lambda|^{1/2}.
\end{aligned}
\end{equation}



\subsection{Lower bound of $\CC_a(\lambda, \xi')$}\label{c a}
Recall that
\[
	\lambda \CC_a(\lambda, \xi') = \CI_1+\frac{\CI_2}{L_2-L_1}
\]
with
\begin{align*}
\CI_1 & =\beta \frac{A^2}{B_a^2-A^2}\left\{2A^2- \frac{A(B_a^2+A^2)}{B_a}\right\},\\
\CI_2 & =-\beta \frac{L_1(L_2-A)}{B_a^2-L_1^2}\left\{2A^2L_1- \frac{(B_a^2+A^2)(L_1^2+A^2)}{2B_a}\right\}
+\beta \frac{L_2(L_1-A)}{B_a^2-L_2^2}\left\{2A^2L_2- \frac{(B_a^2+A^2)(L_2^2+A^2)}{2B_a}\right\}.
\end{align*}
To consider the the lower bound of $\CC_a(\lambda, \xi')$ for the middle frequency of $|\xi'|$, we note the behavior of $\lambda C_a(\lambda, \xi')$ as $|\lambda| \to 0$.
\begin{lem}\label{lem:behav ca}
Let $a>0$ and $\xi' \in \R^{N-1} \setminus \{0\}$.
It holds that
\[
	\lambda \CC_a(\lambda, \xi') = -\frac{a^2}{2\beta A(\sqrt{a+A^2}-A)} + o(1)
\]
as $|\lambda| \to 0$.
\end{lem}
\begin{proof}
Set $B_a^0 = \sqrt{a+A^2}$.
By \eqref{ez}, it holds that
\begin{equation}\label{behavior z}
	z_1(\lambda) = \frac{\lambda}{1+\beta^2/2} + o(\lambda), \quad 
	z_2(\lambda) = \frac{\lambda}{1+\beta^2/2} + a + o(\lambda),
\end{equation}
then we know that
\[
	L_1 = A + o(1), \quad L_2 = B_a = B_a^0 + o(1).
\]
Let us consider the behavior of $\CI_1$ and $\CI_2$ as $|\lambda| \to 0$:
\begin{equation}\label{behavior I1}
	\CI_1 = \frac{\beta A^2}{a}\left\{ 2A^2 - \frac{A(a+2A^2)}{B_a^0} \right\} + o(1)
	= \frac{\beta A^3}{a B_a^0}(2A B_a^0 - a - 2A^2) + o(1).
\end{equation}
The first term of $\CI_2$ satisfies
\begin{equation}\label{behavior I2}
\begin{aligned}
	&-\beta \frac{L_1(L_2-A)}{B_a^2-L_1^2}\left\{2A^2L_1- \frac{(B_a^2+A^2)(L_1^2+A^2)}{2B_a}\right\}\\
	&= -\beta \frac{A(B_a^0-A)}{a}\left\{2A^3- \frac{2(a+2A^2)A^2}{2B_a^0}\right\} + o(1)\\
	&= -\beta \frac{A^3(B_a^0-A)}{a B_a^0}(2AB_a^0- a-2A^2) + o(1).
\end{aligned}
\end{equation}
Now we consider the second term of $\CI_2$.
Thanks to \eqref{behavior z}, we have
\[
	L_1-A = \frac{\lambda}{2A(1+\beta^2/2)} + o(\lambda),\quad
	B_a^2-L_2^2 = \frac{\beta^2 \lambda}{2(1+\beta^2/2)} + o(\lambda).
\]
Therefore,
\[
\begin{aligned}
	&\beta \frac{L_2(L_1-A)}{B_a^2-L_2^2}\left\{2A^2L_2- \frac{(B_a^2+A^2)(L_2^2+A^2)}{2B_a}\right\}\\
	&=\frac{B_a^0}{\beta A}\left\{2A^2B_a^0- \frac{(a+2A^2)^2}{2B_a^0}\right\} + o(1)
	=-\frac{a^2}{2\beta A} + o(1),
\end{aligned}
\]
together with \eqref{behavior I1} and \eqref{behavior I2}, we have
\[
	\lambda C_a(\lambda, \xi') = -\frac{a^2}{2\beta A(B_a^0-A)} + o(1),
\] 
which completes the proof of Lemma \ref{lem:behav ca}.
\end{proof}

Let us consider the lower bound of $\CC_a(\lambda, \xi')$.
\begin{lem}\label{lem:bound ca}
Let $a>0$.
\begin{enumerate}
\item\label{bound ca 1}
There exists a constant $C_{\beta, \epsilon, a}$ such that
\begin{equation}\label{est:ca}
|\CC_a(\lambda, \xi')| \ge C_{\beta, \epsilon, a} \frac{1}{|\lambda|},
\end{equation}
for any $\lambda \in \Sigma_{\epsilon, c_0}$ and $\xi' \in \R^{N-1}\setminus \{0\}$.
\item\label{bound ca 2}
There exist a large number $R$ and a constant $C_{\beta, \epsilon, a}$ such that
\begin{equation}\label{est:ca l}
|\CC_a(\lambda, \xi')| \ge C_{\beta, \epsilon, a} \frac{1}{|\lambda|^{3/2}}
\quad \text{if} \enskip 0 < A^2 \le \frac{|\lambda|}{R}
\end{equation}
for any $\lambda \in \Sigma_{\epsilon, c_0}$.


\end{enumerate}
\end{lem}
\begin{proof}
Note that 
\[
	\CC_a(\lambda, \xi')
	= \frac{\CD(\lambda, \xi')}{2\lambda B_a (B_a^2-L_1^2)(B_a^2-L_2^2)}
	=\frac{\CD(\lambda, \xi')}{2\lambda B_a (\lambda+a-z_1(\lambda))(\lambda+a-z_2(\lambda))}.
\]
It holds that
\begin{equation}\label{Ba}
|B_a|\le C(|\lambda|^{1/2}+1+A)
\end{equation}
for any $(\lambda, \xi') \in \Sigma_\epsilon \times \R^{N-1} \setminus \{0\}$. 
Moreover, \eqref{ez} implies that
\[
	|\lambda+a-z_1(\lambda)| \le C_{\beta, a}(|\lambda|+1),\quad
	|\lambda+a-z_2(\lambda)| \le C_{\beta}|\lambda|
\]
for any $(\lambda, \xi') \in \Sigma_{\epsilon, c_0} \times \R^{N-1} \setminus \{0\}$.
Therefore, we have
\begin{equation}\label{est:ca'}
	|\CC_a(\lambda, \xi')|
	\ge C_{\beta, a}\frac{|\CD(\lambda, \xi')|}{|\lambda|^2 (|\lambda|^{1/2}+1+A)(|\lambda|+1)}.
\end{equation}

In order to obtain the first assertion, we divide the proof into two cases: $(c_0+a)/r \le A^2$ and $(c_0+a)/r \ge A^2$, where $r$ is the same small constant in subsection \ref{large}.
In the case $(c_0+a)/r \le A^2$, we note that $A \ge |\lambda|^{1/2}/r^{1/2}$.
By \eqref{est:ca'} and \eqref{ld}, 
we have
\begin{align*}
|\CC_a(\lambda, \xi')|
&\ge C_{\beta, \epsilon, a} \frac{|\lambda|(|\lambda|+1)^2 A}{|\lambda|^2(|\lambda|^{1/2}+1+A)(|\lambda|+1)}
= C_{\beta, \epsilon, a} \frac{(|\lambda|+1) A}{|\lambda|(|\lambda|^{1/2}+1+A)},
\end{align*}
together with $|\lambda|^{1/2} \le A$ and $1 =a^{-1/2} a^{1/2} \le a^{-1/2} \{(c_0+a)/r\}^{1/2} \le a^{-1/2}  A$, then
we observe that 
\[
|\CC_a(\lambda, \xi')|
\ge C\frac{(|\lambda|+1)A}{|\lambda|A}
\ge C\frac{1}{|\lambda|}.
\]
On the other hand, in the case that $A^2 \le (c_0+a)/r$,
Lemma \ref{lem:behav ca} implies that there exists a constant $C>0$ such that
\[
	|\lambda||\CC_a(\lambda, \xi')| = \left|\frac{a^2}{2 \beta A} \frac{1}{\sqrt{a+A^2}-A}  + O(|\lambda|)\right|
	\ge \frac{a^2}{2 \beta A} \frac{1}{\sqrt{a+A^2}-A} - C c_0.
\]
Note that
\[
	\frac{1}{\sqrt{a+A^2}-A} \ge \frac{1}{\sqrt{a+A^2}} \ge C_{\beta, \epsilon, a}
\]
for $A^2 \le (c_0+a)/r$. 
Choosing $c_0$ small enough if necessary, we have
\[
	|\lambda||\CC_a(\lambda, \xi')| 
	\ge C_{\beta, \epsilon, a}
\]
for $A^2 \le (c_0+a)/r$. 
This completes the proof of Lemma \ref{lem:bound ca} \eqref{bound ca 1}. 

Let us prove the second assertion. In the case that $A^2 \le |\lambda|/R$, it holds that
\[
	(|\lambda|^{1/2}+1+A)(|\lambda|+1) \le C(|\lambda|^{1/2}+1)(|\lambda|+1) \le C
\]
for $|\lambda| \le c_0$.
Then \eqref{est:ca'} and \eqref{sd} imply that
\[
|\CC_a(\lambda, \xi')| \ge C_{\beta, \epsilon, a}\frac{|\lambda|^{1/2}}{|\lambda|^2} = C_{\beta, \epsilon, a}\frac{1}{|\lambda|^{3/2}},
\]
which gives Lemma \ref{lem:bound ca} \eqref{bound ca 2}.

\end{proof}

\subsection{Lower bound of $\CA_a(\lambda, \xi')$}
Recall that
\[
	\CA_a(\lambda, \xi') =B_a(B_a^2-A^2)(L_1+L_2)-A^2(B_a-L_1)(B_a-L_2).
\]
In this subsection, we prove the following lemma.
\begin{lem}
Let $a>0$.
There exists a constant $C_{\beta, \epsilon, a}$ such that 
\begin{equation}\label{est:aa}
|\CA_a(\lambda, \xi')| \ge C_{\beta, \epsilon, a}(|\lambda|+1)^2
\end{equation}
for any $\xi' \in \R^{N-1} \setminus \{0\}$ and $\lambda \in \Sigma_{\epsilon, c_0}$.
\end{lem}
The following proof is the same way as the proof in subsection \ref{c a}.

\subsubsection{Case : $\dfrac{c_0+a}{r} \le A^2$.}
We consider the case: $(c_0+a)/r \le A^2$
with small $r$.
Note that $B_a^2-A^2=\lambda+a$.
Furthermore, \eqref{b} and \eqref{l} imply that
\begin{align*}
B_a(L_1+L_2)&=2A^2(1+O(|t|^2)),\\
A^2(B_a-L_1)(B_a-L_1)&=\frac14 (\lambda+a-z_1(\lambda))(\lambda+a-z_2(\lambda))(1+O(|t|^2)),
\end{align*}
where $t^2=(\lambda+a)/A^2$, then by Lemma \ref{lem:zl} \eqref{z} we have
\begin{align*}
	\CA_a(\lambda, \xi')
	&=2(\lambda+a)A^2(1+O(|t|^2)) - \frac14 (\lambda+a-z_1(\lambda))(\lambda+a-z_2(\lambda))(1+O(|t|^2))\\
	&=2(\lambda+a)A^2(1+O(|t|^2)).
\end{align*}
Choosing $r$ small enough if necessary, we have
\[
|\CA_a(\lambda, \xi')| \ge C_\epsilon (|\lambda|+a)A^2,
\]
together with $(|\lambda|+a)/r \le A^2$ provided by
$(c_0+a)/r \le A^2$ and $|\lambda| \le c_0$, it holds that
\[
|\CA_a(\lambda, \xi')| \ge \frac{C_\epsilon}{r} (|\lambda|+a)^2
\ge C_{\beta, \epsilon, a}(|\lambda|+1)^2
\]
with small $r$.

\subsubsection{Case : $A^2 \le \dfrac{1}{R}$.}
We consider the case: $A^2 \le 1/R$ with large $R$.
Note that
\begin{equation}\label{lower b}
|B_a|\ge C(|\lambda|^{1/2}+1+A)
\end{equation}
holds for any $(\lambda, \xi') \in \Sigma_\epsilon \times \R^{N-1} \setminus \{0\}$.
Furthermore, using the identity:
\[
L_1+L_2=\frac{L_1^2-L_2^2}{L_1-L_2}=\frac{z_1(\lambda)-z_2(\lambda)}{L_1-L_2},
\]
\eqref{ez}, and Lemma \ref{lem:zl} \eqref{est:l}, we have 
\begin{equation}\label{l1 l2}
|L_1+L_2| \ge C_{\beta, \epsilon, a}
\end{equation}
for any $(\lambda, \xi') \in \Sigma_\epsilon \times \R^{N-1} \setminus \{0\}$.
From the definition of $\CA_a$, together with \eqref{lower b} and \eqref{l1 l2}, we have
\[
\CA_a(\lambda, \xi') = (\lambda+1) B_a (L_1+L_2) (1+O(A)).
\]
Then it holds that
\begin{equation}\label{s aa}
\begin{aligned}
	|\CA_a(\lambda, \xi')| &\ge C_{\beta, \epsilon, a}(|\lambda|+1)(|\lambda|^{1/2}+1+A) \ge C_{\beta, \epsilon, a}(|\lambda|+1)^{3/2}\\
	& \ge C_{\beta, \epsilon, a} \frac{1}{(c_0+1)^{1/2}}(|\lambda|+1)^2.
\end{aligned}
\end{equation}

\subsubsection{Case : $\dfrac{1}{R} \le A^2 \le \dfrac{c_0+a}{r}$.}
Set
\[
U=\{(\lambda, \xi') \in \overline{\Sigma_\epsilon} \times \R^{N-1}\mid |\lambda| \le c_0, 1/R^{1/2} \le |\xi'| \le r^{-1/2}(c_0+a)^{1/2}\}.
\]
Thanks to \cite{BM}, $\CA_a(\lambda, \xi') \neq 0$ for $(\lambda ,\xi') \in U$, which implies that
there exists a positive constant $C_{\beta, \epsilon}$ such that
$|\CA_a(\lambda, \xi')| \ge C_{\beta, \epsilon}$
for $(\lambda, \xi') \in U$.
Thus, we have 
\[
|\CA_a(\lambda, \xi')|\ge C_{\beta, \epsilon} \ge C_{\beta, \epsilon}\frac{1}{(c_0+1)^2}(|\lambda|+1)^2
\]
for $(\lambda, \xi') \in \Sigma_\epsilon \times \R^{N-1}$ such that $1/R \le A^2 \le (c_0+a)/r$ and $|\lambda| \le c_0$.

\section{$\CR$-boundedness}

In this section, we prove the main theorem.
By a suitable extension of $\bff$ and $\bG$, the resolvent problem \eqref{r0} with $(\bh, \bH)=(0, 0)$ reduces to the whole-space problem.
We discuss the details in subsection \ref{subsec:proof of main}. 
Therefore, our main task is to prove $\CR$-boundedness for the solution operators of the \eqref{r0} with $\bff=0$, $\bG=0$, namely,
\begin{equation}\label{r}
\left\{
\begin{aligned}
&\lambda\bu -\Delta \bu + \nabla \fp + \beta \DV (\Delta \bQ -a \bQ)=0,
\enskip \dv \bu=0& \quad&\text{in $\R^N_+$},\\
&\lambda \bQ - \beta \bD(\bu) - \Delta \bQ + a \bQ =0& \quad&\text{in $\R^N_+$},\\
&\bu= \bh, \enskip \pd_N \bQ=\bH& \quad&\text{on $\R^N_0$}.
\end{aligned}
\right.
\end{equation}
Let 
\begin{align*}
\widetilde X_q(\R^N_+) &= L_q(\R^N_+)^{N^3+N^2+N} \times L_q(\R^N_+; \R^{N^4}) \times L_q(\R^N_+; \R^{N^3}) \times L_q(\R^N_+; \bS_0) \times L_q(\R^N_+; \bS_0),\\
\widetilde Y_q(\R^N_+) &= \widetilde X_q(\R^N_+) \times L_q(\R^N_+;\R^{N^3})\times L_q(\R^N_+; \bS_0).
\end{align*}
\begin{thm}\label{thm:Rbdd}
Let $N \ge 2$ and $a>0$.
Let $1 < q < \infty$. Let $\epsilon \in (\epsilon_0, \pi/2)$ with $\tan \epsilon_0 \ge |\beta|/\sqrt 2$, and let $c_0$ be a small positive constant determined in Lemma \ref{lem:zl}.
Then
there exist 
an operator families 
\begin{align*}
&\CA_2 (\lambda) \in 
{\rm Hol} (\Sigma_{\epsilon, c_0}, 
\CL(\widetilde X_q(\R^N_+), H^2_q(\R^N_+)^N))\\
&\CB_2 (\lambda) \in 
{\rm Hol} (\Sigma_{\epsilon, c_0}, 
\CL(\widetilde Y_q(\R^N_+), H^3_q(\R^N_+; \bS_0)))
\end{align*}
such that 
for any $\lambda = \gamma + i\tau \in \Sigma_{\epsilon, c_0}$ with $|\lambda| \le c_0$
, $\bh \in H^2_q(\R^N_+)^N$ with $h_N=0$, and $\bH \in H^2_q(\R^N_+; \bS_0)$, 
\begin{align*}
\bu &= \CA_2 (\lambda) (\CS_\lambda\bh, \CS_\lambda\bH, \lambda^{1/2}\bH), \\
\bQ &= \CB_2 (\lambda) (\CS_\lambda\bh, \CT_\lambda\bH, \bH)
\end{align*}
are solutions of problem \eqref{r}
and 
\begin{align*}
&\CR_{\CL(\widetilde X_q(\R^N_+), A_q(\R^N_+))}
(\{(\tau \pd_\tau)^\ell \CS_\lambda \CA_2 (\lambda) \mid 
\lambda \in \Sigma_{\epsilon, c_0}\}) 
\leq r, \\
&\CR_{\CL(\widetilde Y_q(\R^N_+), B_q(\R^N_+))}
(\{(\tau \pd_\tau)^\ell \CT_\lambda \CB_2 (\lambda) \mid 
\lambda \in \Sigma_{\epsilon, c_0}\}) 
\leq r 
\end{align*}
for $\ell = 0, 1$,
where 
$\CS_\lambda$, $\CT_\lambda$, $A_q(\R^N_+)$, and $B_q(\R^N_+)$ are defined in Theorem \ref{thm:Rbdd H},
$r$ is a constant independent of $\lambda$.
\end{thm}

\subsection{Preliminary}
First, we define the classes of multipliers.
\begin{dfn}\label{def:m}
Let $\Sigma$ be a domain in $\C$ and let $m(\xi', \lambda): \R^{N-1} \setminus\{0\} \times \Sigma \to \C$ be a $C^\infty$ function with respect to $\xi' \in \R^{N-1} \setminus\{0\}$ and a $C^1$ function with respect to $\tau \in \R \setminus\{0\}$, where $\lambda=\gamma+i\tau$. 
Let $\ell=0, 1$.
\begin{enumerate}
\item
If there is $s \in \R$ such that for any $\xi' \in \R^{N-1}\setminus\{0\}$ and for any $\lambda\in \Sigma$, it holds
\[
    \left|(\tau\partial_\tau)^\ell D^\alpha_{\xi'} m(\xi',\lambda) \right| \le C(|\lambda|^\frac{1}{2}+|\xi'|)^{s-|\alpha|} \quad \forall \alpha \in \N_0^{N-1}
\]
with some constant $C$,
then we write $m(\xi',\lambda) \in \BM_{s, 1}$.
If for any $\xi' \in \R^{N-1}\setminus\{0\}$ and for any $\lambda \in \Sigma$, 
it holds
\[
\left|(\tau\partial_\tau)^\ell D^\alpha_{\xi'}m(\xi',\lambda)\right| \le C (|\lambda|^\frac{1}{2}+1+|\xi'|)^s(|\lambda|^\frac{1}{2}+|\xi'|)^{-|\alpha|} \quad \forall \alpha \in \N_0^{N-1}
\]
with some constant $C$,
then we write $m(\xi',\lambda) \in \widetilde\BM_{s, 1}$.

\item
If there is $s \in \R$ such that for any $\xi' \in \R^{N-1}\setminus\{0\}$ and for any $\lambda \in \Sigma$, 
it holds
\[
\left|(\tau\partial_\tau)^\ell D^\alpha_{\xi'}m(\xi',\lambda)\right| \le C (|\lambda|^\frac{1}{2}+|\xi'|)^{s}|\xi'|^{-|\alpha|} \quad \forall \alpha \in \N_0^{N-1}
\]
with some constant $C$,
then we write $m(\xi',\lambda) \in \BM_{s, 2}$.
If for any $\xi' \in \R^{N-1}\setminus\{0\}$ and for any $\lambda \in \Sigma$, 
it holds
\[
\left|(\tau\partial_\tau)^\ell D^\alpha_{\xi'}m(\xi',\lambda)\right| \le C (|\lambda|^\frac{1}{2}+1+|\xi'|)^{s}|\xi'|^{-|\alpha|} \quad \forall \alpha \in \N_0^{N-1}
\]
with some constant $C$,
then we write $m(\xi',\lambda) \in \widetilde\BM_{s, 2}$.

\item
If there is $s \in \R$ such that for any $\xi' \in \R^{N-1}\setminus\{0\}$ and for any $\lambda\in \Sigma$, it holds
\[
    \left|(\tau\partial_\tau)^\ell D^\alpha_{\xi'} m(\xi',\lambda) \right| \le C(|\lambda|^\frac{1}{2}+1+|\xi'|)^{s-|\alpha|} \quad \forall \alpha \in \N_0^{N-1}
\]
with some constant $C$,
then we write $m(\xi',\lambda) \in \BM_{s, 1}'$.
\end{enumerate}
\end{dfn}

\begin{remark}
By Definition \ref{def:m}, $\BM_{s, 1}' \subset \widetilde \BM_{s, 1}$ holds.
\end{remark}

The following lemmas will be used to prove the $\CR$-boundedness for the solution operators of the velocity and the order parameter.

\begin{lem}\label{m}
Let $a>0$, $s \in \R$.
\begin{enumerate}
\item\label{m-1}
Let $\epsilon \in (0, \pi/2)$.
Then
$A^s \in \BM_{s, 2}$ for $s \ge 0$.
In particular, $|D^\alpha_{\xi'} A^2| \le 2A^{2-|\alpha|}~(|\alpha| \le 2)$,
$|D^\alpha_{\xi'} A^2| = 0~(|\alpha| \ge 3)$.
\item\label{m-2}
Let $\epsilon \in (0, \pi/2)$ and let $\lambda \in \Sigma_\epsilon$.
Then
$B_a^s \in \widetilde\BM_{s, 1}$.
\item
Let $\epsilon \in (\theta_0, \pi/2)$ with $\tan \epsilon_0 \ge |\beta|/\sqrt 2$, and let $\lambda \in \Sigma_{\epsilon, c_0}$, where $c_0$ is a small constant determined in Lemma \ref{lem:zl}.
Then
\begin{enumerate}
\item\label{m3-a}
$L_1^s \in \BM_{s, 1}$. In particular, $L_1^s \in \widetilde \BM_{s, 1}$ for $s \ge 0$.
\item\label{m3-b}
$L_2^s \in \widetilde\BM_{s, 1}$.
\end{enumerate}
\end{enumerate}

\end{lem}

\begin{proof}
The first assertion was proved by \cite[Lemma 5.2]{ShiS}. As we mentioned in \cite[Lemma 3.2.4]{BM}, the second assertion follows from the same method as \cite[Lemma 5.2]{ShiS} and \cite[Lemma 4.4]{SS}.
Furthermore, the first assertion provide that $A^2 \in \BM_{2, 1}$ and $A^2 \in \widetilde \BM_{2, 1}$, together with Bell's formula and Lemma \ref{lem:zl} \eqref{est:l}, give us the third assertion.  
\end{proof}
\begin{lem}\label{m2}
Let $a>0$ and let $\lambda \in \Sigma_{\epsilon, c_0}$, where $c_0$ is a small constant determined in Lemma \ref{lem:zl}.
Let $s \in \R$.
Then the following assertions hold.
\begin{enumerate}
\item\label{m2-1}
$(A+B_a)^s \in \widetilde \BM_{s, 2}$.
\item\label{m2-2}
$(B_a+L_j)^s \in \widetilde \BM_{s, 1}$ for $j=1, 2$.
\item\label{m2-3}
$(A+L_1)^s \in \BM_{s, 2}$, \enskip $(A+L_2)^s \in \widetilde\BM_{s, 2}$.
\item\label{m2-4}
$B_a-A \in \widetilde \BM_{-1, 2}$, \enskip 
$L_2-A \in \widetilde \BM_{-1, 2}$, \enskip 
$L_j-B_a \in \widetilde \BM_{-1, 1}$.
In particular, $L_2-A \in \BM_{0, 2}$, \enskip 
$L_j-B_a \in \BM_{0, 1}$.
\item\label{m2-5}
$\dfrac{L_1-A}{B_a^2-L_j^2} \in \BM_{-1, 2}$ for $j=1, 2$, \enskip 
$\dfrac{L_2-A}{B_a^2-L_1^2} \in \widetilde\BM_{-1, 2}$, \enskip 
$\dfrac{B_a-A}{B_a^2-L_1^2} \in \widetilde\BM_{-1, 2}$, \enskip 
$\dfrac{B_a-L_2}{B_a^2-L_1^2} \in \widetilde\BM_{-1, 1}$.
\item\label{m2-6}
$\dfrac{\lambda}{\lambda+1}\dfrac{L_2-A}{B_a^2-L_2^2} \in \widetilde\BM_{-1, 2}$,
\enskip
$\dfrac{\lambda}{\lambda+1}\dfrac{B_a-A}{B_a^2-L_2^2} \in \widetilde\BM_{-1, 2}$.

\end{enumerate}

\end{lem}
\begin{proof}
Since the proof is essentially the same as Lemma 3.2.10 of \cite{BM}, we write the part of the proof.
\begin{enumerate}
\item
This was proved by \cite[Lemma 5.2]{ShiS}.
\item
The second assertion will be obtained by Bell's formula and the following estimate 
\begin{equation}\label{bl}
c(|\lambda|^{1/2}+1+A) \le |B_a + L_j| \le C(|\lambda|^{1/2}+1+A)
\end{equation}
for any $(\xi', \lambda) \in \R^{N-1}\setminus\{0\} \times \Sigma_{\epsilon, c_0}$
with some constant $c$ and $C$ depending on $\beta$, $\epsilon$, and $a$.
Let us prove \eqref{bl}.
Since $\Re L_j >0$,
\[
|B_a+L_j| \ge \Re (B_a+L_j) \ge \Re B_a.
\] 
If $\lambda \in \Sigma_\epsilon$, we have $\lambda+a+A^2 \in \Sigma_\epsilon$, then $|\arg B_a| \le \pi/2 - \epsilon/2$.
Thus,
\[
\Re B_a \ge |\lambda+a+A^2|^{1/2} \cos(\pi/2 - \epsilon/2) = |\lambda+a+A^2|^{1/2}\sin (\epsilon/2) \ge C_{\epsilon, a}(|\lambda|+1+A^2)^{1/2}
\]
for any $(\xi', \lambda) \in \R^{N-1}\setminus\{0\} \times \Sigma_\epsilon$.
On the other hand, Lemma \ref{m} implies that $|B_a + L_j| \le C(|\lambda|^{1/2}+1+A)$ for any $(\xi', \lambda) \in \R^{N-1}\setminus\{0\} \times \Sigma_{\epsilon, c_0}$
, then we have \eqref{bl}.
\item
Since $|A+L_1| \ge \Re(A+L_1) \ge \Re L_1$
and $|A+L_2| \ge \Re(A+L_2) \ge \Re L_2$, together with Lemma \ref{lem:zl} \eqref{est:l}, we also have the third assertion.

\item
The first, second, and third assertions, together with
\begin{align*}
B_a - A &= \frac{B_a^2-A^2}{B_a+A}= \frac{\lambda+a}{B_a+A}, \enskip
L_2 - A = \frac{L_2^2-A^2}{L_2+A} = \frac{z_2(\lambda)}{L_2+A}, \\
L_j - B_a &=\frac{L_j^2 - B_a^2}{L_j +B_a} = \frac{z_j(\lambda)-(\lambda+a)}{L_j+B_a},
\end{align*}
furnish that $B_a-A \in \widetilde \BM_{-1, 2}$, $L_2-A \in \widetilde \BM_{-1, 2}$, and $L_j - B_a \in \widetilde \BM_{-1, 1}$.
Since $\BM_{-1, j} \subset \BM_{0, j}$ for $j=1, 2$, the fourth assertions hold for $\lambda \in \Sigma_{\epsilon, c_0}$.

\item
Note that
\[
B_a-A=\frac{B_a^2-A^2}{B_a+A}=\frac{\lambda+a}{B_a+A},
\quad
L_j-A=\frac{L_j^2-A^2}{L_j+A}=\frac{z_j(\lambda)}{L_j+A}~(j=1, 2).
\]
By the above assertions \eqref{m2-1} and \eqref{m2-3}, we have
\begin{align*}
\frac{L_1-A}{B_a^2-L_j^2}&=\frac{z_1(\lambda)}{\lambda+a-z_j(\lambda)} \frac{1}{L_1+A} \in \BM_{-1, 2},\\
\frac{L_2-A}{B_a^2-L_1^2}&=\frac{z_2(\lambda)}{\lambda+a-z_1(\lambda)} \frac{1}{L_2+A} \in \widetilde\BM_{-1, 2},\\
\frac{B_a-A}{B_a^2-L_1^2}&=\frac{\lambda+a}{\lambda+a-z_1(\lambda)} \frac{1}{B_a+A} \in \widetilde\BM_{-1, 2},\\
\frac{B_a-L_2}{B_a^2-L_1^2}&=\frac{\lambda+a-z_2(\lambda)}{\lambda+a-z_1(\lambda)} \frac{1}{B_a+L_2} \in \widetilde\BM_{-1, 1}.
\end{align*}
\item
This assertion holds by the same way as \eqref{m2-5}.
\end{enumerate}
\end{proof}

\begin{lem}\label{lem:ca}
Let $a>0$. Let $\xi' \in \R^{N-1} \setminus \{0\}$ and $\lambda \in \Sigma_{\epsilon, c_0}$, where $c_0$ is a small constant determined in Lemma \ref{lem:zl}.
Then 
$\lambda \CC_a \in \BM_{0, 2}$.
In addition, let $\ell=0, 1$. Then there exists a large number $R$ such that 

\begin{enumerate}
\item
\begin{equation}\label{lem:ca h}
|\lambda|^{-1} |(\tau \pd_{\tau})^\ell D^\alpha_{\xi'} \CC_a(\lambda, \xi')^{-1}| \le C_\alpha |\xi'|^{-|\alpha'|}.
\end{equation}
\item
\begin{equation}\label{lem:ca l}
|\lambda|^{-3/2} |(\tau \pd_{\tau})^\ell D^\alpha_{\xi'} \CC_a(\lambda, \xi')^{-1}| \le C_\alpha |\xi'|^{-|\alpha'|}
\quad \text{if} \enskip A^2 \le \frac{|\lambda|}{R}.
\end{equation}


\end{enumerate}
 
\end{lem}

\begin{proof}
We rewrite $\CC_a=\CC_a(\lambda, \xi')$ as follows:
\[
\mathcal{C}_a=\frac{1}{\lambda}\left(\mathcal{I}_1+\frac{\mathcal{I}_2}{L_2-L_1}\right)
\]
with 
\begin{equation}\label{f.alt.C-it}
    \begin{aligned}
        & \mathcal{I}_1+\frac{\mathcal{I}_2}{L_2-L_1}= -\frac{\beta A^3(B_a^2-A^2)}{B_a(B_a+A)^2}\\
        & + \frac{\beta A(L_1-A)(L_2-A)[(L_2A-B_a^2)(L_1-A)+L_2(A^2-B_a^2)]}{(B_a^2-L_1^2)(B_a^2-L_2^2)} \\
        & + \frac{\beta (B_a-A)^2[(A^2B_a^2+A^2L_1L_2+B_a^2L_1L_2+L_2^2B_a^2)(A-L_1)+(B_a^2-L_2^2)AL_1(L_1+A)]}{2B_a(B_a^2-L_1^2)(B_a^2-L_2^2)}.
    \end{aligned}
\end{equation}
Note that 
\[
(L_2 A-B_a^2)(L_1-A) = (L_2-B_a)(A+B_a) + (A-L_2)B_a,
\]
which appears on the second line of \eqref{f.alt.C-it}.
Then Lemma \ref{m} and Lemma \ref{m2} \eqref{m2-4}, \eqref{m2-5}
imply that
$\lambda \CC_a \in \BM_{0, 2}$.
Let $s \in \N$. 
Bell's formula and $\lambda \CC_a \in \BM_{0, 2}$, together with \eqref{est:ca}, 
furnish that
\begin{equation}\label{ca-1}
\begin{aligned}
|D^{\alpha}_{\xi'} \CC_a^{-s}| &\le C_\alpha \sum^{|\alpha|}_{\ell=1} |\CC_a|^{-s-\ell} \sum_{\substack{\alpha_1+\cdots \alpha_\ell=\alpha, \\ |\alpha_k|\ge 1}}
|D^{\alpha_1}_{\xi'}C_a| \cdots |D^{\alpha_\ell}_{\xi'}C_a|\\
&\enskip \le C_{\alpha} \sum^{|\alpha|}_{\ell=1} |\lambda|^{s+\ell} |\lambda|^{-\ell} |\xi'|^{-|\alpha|}
\enskip \le C_{\alpha} |\lambda|^s |\xi'|^{-|\alpha|}.
\end{aligned}
\end{equation}
Now we consider the estimate of $\tau \pd_{\tau} \CC_a^{-1}=-(\tau \pd_{\tau} \CC_a) \CC_a^{-2}$.
By Leibniz rule, $\lambda \CC_a \in \BM_{0, 2}$, and \eqref{ca-1} with $s=2$,
\[
|D^{\alpha}_{\xi'} \tau \pd_{\tau} \CC_a^{-1}| \le C \sum_{\beta+\gamma = \alpha}|D^{\beta}_{\xi'} \tau \pd_{\tau} \CC_a| 
|D^{\gamma}_{\xi'} \CC_a^{-2}| \le C_\alpha |\lambda| |\xi'|^{-|\alpha|},
\]
which furnishes that \eqref{lem:ca h}.

In the case that $A^2 \le |\lambda|/R$,
\eqref{est:ca l} furnishes that
\begin{equation}\label{ca-2}
\begin{aligned}
|D^{\alpha}_{\xi'} \CC_a^{-s}| &\le C_\alpha \sum^{|\alpha|}_{\ell=1} |\CC_a|^{-s-\ell} \sum_{\substack{\alpha_1+\cdots \alpha_\ell=\alpha, \\ |\alpha_k|\ge 1}}
|D^{\alpha_1}_{\xi'}C_a| \cdots |D^{\alpha_\ell}_{\xi'}C_a|\\
& \le C_{\alpha} \sum^{|\alpha|}_{\ell=1} |\lambda|^{3(s+\ell)/2} |\lambda|^{-\ell} |\xi'|^{-|\alpha'|}
\le C_{\alpha} |\lambda|^{3s/2} |\xi'|^{-|\alpha'|}
\end{aligned}
\end{equation}
for $|\lambda| \le c_0$. Similarly, \eqref{ca-2} with $s=2$ implies that
\[
|D^{\alpha}_{\xi'} \tau \pd_{\tau} \CC_a^{-1}| \le C \sum_{\beta+\gamma = \alpha}|D^{\beta}_{\xi'} \tau \pd_{\tau} \CC_a| 
|D^{\gamma}_{\xi'} \CC_a^{-2}| \le C_\alpha |\lambda|^2 |\xi'|^{-|\alpha|} \le C_\alpha |\lambda|^{3/2} |\xi'|^{-|\alpha|},
\]
then we have \eqref{lem:ca l}, which completes the proof of Lemma \ref{lem:ca}.
\end{proof}

Lemma \ref{lem:ca} implies that the following corollary.
\begin{cor}\label{cor:ca}
Let $a>0$ and $\lambda \in \Sigma_{\epsilon, c_0}$, where $c_0$ is a small constant determined in Lemma \ref{lem:zl}.
Then 
$\lambda^{-1}\CC_a^{-1} \in \BM_{0, 2}$.
\end{cor}

\begin{lem}\label{lem:aa}
Let $a>0$ and $\lambda \in \Sigma_{\epsilon, c_0}$, where $c_0$ is a small constant determined in Lemma \ref{lem:zl}.
Then
$(\lambda+a)^{-1}\CA_a \in \widetilde \BM_{2, 1}$ and $(\lambda+a)\CA_a^{-1} \in \widetilde \BM_{-2, 1}$.
\end{lem}

\begin{proof}
First, we prove $(\lambda+a)^{-1}\CA_a \in \widetilde \BM_{2, 1}$.
Recall that 
\begin{align*}
\CA_a=\CA_a(\lambda, \xi')
&=B_a(B_a^2-A^2)(L_1+L_2)-A^2(B_a-L_1)(B_a-L_2)\\
&=(\lambda+a)B_a(L_1+L_2)-(\lambda+a-z_1(\lambda))\frac{A^2(B_a-L_2)}{B_a+L_1}.
\end{align*}
Lemma \ref{m} implies that
$B_a, L_1+L_2 \in \widetilde \BM_{1, 1}$, 
$A^2 \in \widetilde \BM_{2, 1}$.
Furthermore, it holds
\[
\frac{B_a-L_2}{B_a+L_1}=\frac{B_a^2-L_2^2}{(B_a+L_1)(B_a+L_2)}=\frac{(\lambda+a-z_2(\lambda))^{1/2}}{B_a+L_1}\frac{(\lambda+a-z_2(\lambda))^{1/2}}{B_a+L_2} \in \BM_{0, 1}
\]
by Lemma \ref{m2} \eqref{m2-2} and \eqref{ez}.
Then
we have $(\lambda+a)^{-1}\CA_a \in \widetilde \BM_{2, 1}$.

Next, we prove $(\lambda+a)\CA_a^{-1} \in \widetilde \BM_{-2, 1}$.
By \eqref{est:aa}, there exists a positive constant $C$ such that
\begin{equation}\label{bound aa}
\left|\frac{\CA_a(\lambda, \xi')}{(\lambda+a) B_a^2}\right| \ge C.
\end{equation}
In fact, in the case that $(c_0+a)/r \le A^2$, by the same manner as the estimates of $\CD(\lambda, \xi')$,
\begin{align*}
&(\lambda+a)^{-1}\CA_a(\lambda, \xi')\\
&\enskip =B_a(L_1+L_2)-\frac{1}{\lambda+a}A^2(B_a-L_1)(B_a-L_2)\\
&\enskip=A^2-\frac12 \{(\lambda+a)-3(z_1(\lambda)+z_2(\lambda))+z_1(\lambda)z_2(\lambda)(\lambda+a)^{-1}\} (1+O(|t|^2))\\
&\enskip =A^2(1+O(|t|^2)),
\end{align*}
where $t^2=(\lambda+a)/A^2$.
Thus, we have $|(\lambda+a)^{-1} B_a^{-2}\CA_a(\lambda, \xi')| \ge C$.
In the case that $A^2 \le 1/R$, by \eqref{s aa} and $|B_a^2| \le |\lambda|+a + 1/R \le C(|\lambda|+a)$,
we have
\[
\left|\frac{\CA_a(\lambda, \xi')}{(\lambda+a) B_a^2}\right| \ge C\frac{(|\lambda|+1)^2}{(|\lambda|+a)^2} \ge C.
\]
In the case that $1/R \le A^2 \le (c_0+a)/r$, $\CA_a, B_a \neq 0$ implies that \eqref{bound aa}.
Let $s\in \N$.
Bell's formula provides that
\begin{equation}\label{aa-1}
\begin{aligned}
&|D^{\alpha}_{\xi'} ((\lambda+a)B_a^2 \CA_a^{-1})^s|=|D^{\alpha}_{\xi'} ((\lambda+a)^{-1}B_a^{-2} \CA_a)^{-s}| \\
&\enskip \le C_\alpha \sum^{|\alpha|}_{\ell=1} |(\lambda+a)^{-1}B_a^{-2} \CA_a|^{-s-\ell} \sum_{\substack{\alpha_1+\cdots \alpha_\ell=\alpha, \\ |\alpha_k|\ge 1}}
|D^{\alpha_1}_{\xi'}(\lambda+a)^{-1}B_a^{-2} \CA_a| \cdots |D^{\alpha_\ell}_{\xi'}(\lambda+a)^{-1}B_a^{-2} \CA_a|\\
&\enskip \le C_{\alpha} (|\lambda|^{1/2}+|\xi'|)^{-|\alpha'|},
\end{aligned}
\end{equation}
where we have used $(\lambda+a)^{-1}B_a^{-2} \CA_a \in \BM_{0, 1}$ provided by $(\lambda+a)^{-1}\CA_a \in \widetilde \BM_{2, 1}$.
Note that
\[
\tau \pd_{\tau}\left(\frac{(\lambda+a)B_a^2 }{\CA_a}\right)
=-\left(\tau \pd_{\tau} \frac{\CA_a}{(\lambda+a) B_a^2}\right) \left(\frac{(\lambda+a)B_a^2}{\CA_a}\right)^2,
\]
then \eqref{aa-1} with $s=2$ and $(\lambda+a)^{-1}B_a^{-2} \CA_a \in \BM_{0, 1}$ furnish that
$(\lambda+a)B_a^2 \CA_a^{-1} \in \BM_{0, 1}$, which implies that $(\lambda+a)\CA_a^{-1} \in \widetilde \BM_{-2, 1}$.
\end{proof}

Next, we consider the estimate for $\CE$:
\begin{align*}
\CE&=\frac{\beta}{L_2-L_1}\left\{\frac{L_1}{B_a^2-L_1^2}\left(2A^2L_1-\frac{(B_a^2+A^2)(L_1^2+A^2)}{2B_a}\right)-\frac{L_2}{B_a^2-L_2^2}\left(2A^2L_2-\frac{(B_a^2+A^2)(L_2^2+A^2)}{2B_a}\right)\right\}\\
&=\frac{\beta A[(B_a-L_2)\{2B_aL_1(A-L_1)+L_1^2(B_a-L_2)\}-L_1L_2(B_a-A)^2-B_a^2(L_2-A)^2]}{(B_a^2-L_1^2)(B_a^2-L_2^2)} \\
&\enskip +\frac{\beta(B_a-A)^2[-A^2B_a^2-L_1L_2A^2-B_a^2(L_1^2+L_1L_2+L_2^2)+L_1^2L_2^2]}{2B_a(B_a^2-L_1^2)(B_a^2-L_2^2)}. 
\end{align*}

\begin{lem}\label{lem:e}
Let $a>0$ and $\lambda \in \Sigma_{\epsilon, c_0}$, where $c_0$ is a small constant determined in Lemma \ref{lem:zl}.
Then
$\CE=(1+1/\lambda)(A m_0+m_1)$, where $m_0 \in \BM_{0, 2}$ and $m_1 \in \widetilde \BM_{1, 2}$.
\end{lem}
\begin{proof}
Lemma \ref{m} and Lemma \ref{m2} imply that
\begin{align*}
&\frac{A(B_a-L_2)\{2B_aL_1(A-L_1)+L_1^2(B_a-L_2)\}}{(B_a^2-L_1^2)(B_a^2-L_2^2)} =Am_0,\\
&\frac{-A\{L_1L_2(B_a-A)^2-B_a^2(L_2-A)^2\}}{(B_a^2-L_1^2)(B_a^2-L_2^2)}
+\frac{(B_a-A)^2(-A^2B_a^2-L_1L_2A^2)}{2B_a(B_a^2-L_1^2)(B_a^2-L_2^2)}=\frac{\lambda+1}{\lambda} A m_0\\
&\frac{(B_a-A)^2\{-B_a^2(L_1^2+L_1L_2+L_2^2)+L_1^2L_2^2\}}{2B_a(B_a^2-L_1^2)(B_a^2-L_2^2)}
=\frac{\lambda+1}{\lambda}m_1,
\end{align*}
where $m_0 \in \BM_{0, 2}$ and $m_1 \in \widetilde \BM_{1, 2}$, which completes the proof of Lemma \ref{lem:e}.
\end{proof}

We also consider the estimates for $E^h_k$, $E_k^{H_{NN}}$, and $E^{H_{jl}}_k$ defined by
\begin{align*}\label{f.alt.E}
E_k^h&= i\xi_k\left\{\frac{\mathcal{E}}{\lambda\mathcal{C}_a}\left[\frac{2AB_a}{(B_a+A)^2} -\frac{2[(L_2B_a^2-A^2B_a)(L_1-A)+(AL_1L_2+AB_aL_1)(A-B_a)]}{(B_a^2-A^2)(B_a+L_1)(B_a+L_2)}\right.\right. \\
&\left.\left.-\frac{AB_a+L_1L_2}{(B_a+L_1)(B_a+L_2)}\right] -\frac{2(B_a^2(L_1+L_2)-A^2B_a+L_1L_2B_a)}{(B_a^2-A^2)(B_a+L_1)(B_a+L_2)}+\frac{B_a}{(B_a+L_1)(B_a+L_2)}\right\}, \\
E_k^{H_{NN}}&= \frac{ i\xi_k\mathcal{B}}{\lambda\mathcal{C}_a}A^2+\frac{2i\xi_kB_a^2}{\beta(B_a^2-A^2)},\\
E^{H_{jN}}_k&= -\frac{(B_a^2+A^2)i\xi_ki\xi_j\mathcal{B}}{B_a\lambda\mathcal{C}_a}-\frac{4i\xi_ji\xi_k B_a}{\beta(B_a^2-A^2)}+\frac{2B_a}{\beta}\delta_{jk},\\
E^{H_{j\ell}}_k&=\frac{i\xi_ji\xi_ki\xi_{\ell}\mathcal{B}}{\lambda\mathcal{C}_a}+\frac{2i\xi_ji\xi_ki\xi_{\ell}}{\beta(B_a^2-A^2)}-\frac{2}{\beta}i\xi_j\delta_{k \ell}-\frac{2}{\beta}i\xi_{\ell}\delta_{jk}
\end{align*}
for $j, k, \ell =1, \dots N-1$ with
\[
\mathcal{B}= \frac{2A^2}{(B_a+A)^2}+\frac{2[(L_2B_a^2-A^2B_a)(L_1-A)+(AL_1L_2+AB_aL_1)(A-B_a)]}{(B_a+L_1)(B_a+L_2)(B_a^2-A^2)}
 +\frac{AB_a+L_1L_2}{(B_a+L_1)(B_a+L_2)}.
\]
Note that the second and third terms of $\CB$ is the same as the second and third terms of $E^h_k$.
The following lemma follows from Lemmas \ref{m}, \ref{m2}, and \ref{lem:e}.
\begin{lem}\label{lem:eh}
Let $a>0$ and $\lambda \in \Sigma_{\epsilon, c_0}$, where $c_0$ is a small constant determined in Lemma \ref{lem:zl}.
Then
\begin{align*}
&E^h_k = i\xi_k \left(\frac{\CE}{\lambda \CC_a} m_0 + \frac{m_1}{\lambda+a}\right), ~
E^{H_{NN}}_k  = i\xi_k \left(\frac{m_2}{\lambda \CC_a} + \frac{m_3}{\lambda+a}\right), \\
&E^{H_{jN}}_k = i\xi_j i\xi_k \left(\frac{m_4}{\lambda \CC_a} + \frac{m_5}{\lambda+a}\right) + m_6,\\
&E^{H_{j\ell}}_k= i\xi_j i\xi_k i\xi_\ell \left(\frac{m_7}{\lambda \CC_a} + \frac{m_8}{\lambda+a}\right) + i\xi_j m_9 + i\xi_\ell m_{10},
\end{align*}
 where $m_0, m_7, m_8, m_9, m_{10} \in \BM_{0, 2}$, $m_2, m_6 \in \BM_{2, 2}$, $m_1, m_5, m_6 \in \widetilde\BM_{1, 1}$,
$m_3 \in \widetilde\BM_{2, 1}$, $m_4 \in \widetilde\BM_{1, 2}$.
\end{lem}
\begin{proof}
Let us consider $E^h_k$.
The first and third terms of $E^h_k$:
\[
	\frac{2AB_a}{(B_a+A)^2} -\frac{AB_a+L_1L_2}{(B_a+L_1)(B_a+L_2)}
\]
belong to $\BM_{0, 2}$ by Lemma \ref{m} and Lemma \ref{m2} \eqref{m2-1} \eqref{m2-2}.
For the second term of $E^h_k$,
note that $(L_2 B_a-A^2)=(L_2-A)(B_a+A) + A(B_a-L_2)$, then
Lemma \ref{m2}, \eqref{ez}, and Lemma \ref{lem:zl} \eqref{z} imply that
\begin{align*}
&\frac{(L_2B_a^2-A^2B_a)(L_1-A)}{(B_a^2-A^2)(B_a+L_1)(B_a+L_2)}\\
&\enskip =\frac{z_1(\lambda) z_2(\lambda) B_a (B_a+A)}{(\lambda+a)(B_a+L_1)(B_a+L_2)(L_1+A)(L_2+A)}
+\frac{z_1(\lambda) (\lambda+a-z_2(\lambda)) B_a A}{(\lambda+a)(B_a+L_1)(B_a+L_2)^2(L_1+A)},\\
&\frac{(AL_1L_2+AB_aL_1)(A-B_a)}{(B_a^2-A^2)(B_a+L_1)(B_a+L_2)}
=-\frac{AL_1L_2+AB_aL_1}{(B_a+L_1)(B_a+L_2)(B_a+A)}
\end{align*}
belong to $\BM_{0, 2}$.
For the last two terms, Lemma \ref{m} \eqref{m-1} implies that 
$A^2 \in \widetilde \BM_{2, 1}$, thus
\[
-\frac{1}{B_a^2-A^2}\left\{\frac{2(B_a^2(L_1+L_2)-A^2B_a+L_1L_2B_a)}{(B_a+L_1)(B_a+L_2)}-\frac{(B_a^2-A^2)B_a}{(B_a+L_1)(B_a+L_2)}\right\}=\frac{m_1}{\lambda+a}
\]
with $m_1 \in \widetilde\BM_{1, 1}$. Therefore, we have
\[
E^h_k = i\xi_k \left(\frac{\CE}{\lambda \CC_a} m_0 + \frac{m_1}{\lambda+a}\right)
\]
with $m_0 \in \BM_{0, 2}$ and $m_1 \in \widetilde\BM_{1, 1}$.

The assertions for the other terms are obtained by Lemma \ref{m} and the fact that 
$\CB \in \BM_{0, 2}$, which completes the proof of Lemma \ref{lem:eh}.
\end{proof}

Finally, we prove the estimate of $B_a L_1-A^2$.

\begin{lem}\label{lem:bl}
Let $a>0$ and $\lambda \in \Sigma_{\epsilon, c_0}$, where $c_0$ is a small constant determined in Lemma \ref{lem:zl}.
Then
$B_a L_1-A^2=(\lambda+1)m$, where $m \in \BM_{0, 2}$.

\end{lem}

\begin{proof}
Note that
\[
B_aL_1-A^2=(B_a-A)(L_1+A)+(L_1-B_a)A.
\]
The identities
\begin{align*}
B_a-A=\frac{\lambda+a}{B_a+A}, \enskip L_1-B_a=\frac{z_1(\lambda)-(\lambda+a)}{L_1+B_a},
\end{align*}
together with Lemma \ref{m2} \eqref{m2-1}, \eqref{m2-2}, and \eqref{ez},
 we have Lemma \ref{lem:bl}.
\end{proof}

\subsection{$\CR$-boundedness for the solution operator of the velocity}\label{subsection:v}
Set
\[
\CM(\gamma_1, \gamma_2, x_N)=\frac{e^{-\gamma_1 x_N}-e^{-\gamma_2 x_N}}{\gamma_1-\gamma_2}.
\]
The following lemmas are proved by the same way as \cite[Lemma 5.3 and Lemma 5.4]{ShiS}.
\begin{lem}
Let $\ell = 0, 1$, let $\beta \in \mathbb{R}\setminus\{0\}$, $\epsilon \in (\epsilon_0, \pi/2)$ with $\tan \epsilon_0\ge |\beta|/\sqrt{2}$.
Then for any $\alpha \in \N_0^{N-1}$, $\lambda \in \Sigma_\epsilon$, 
$\xi \in \R^{N-1}\setminus \{0\}$, and $x_N \in (0, \infty)$, the following estimates hold.
\begin{align*}
|D^\alpha_{\xi'} \{ (\tau \pd_\tau)^\ell e^{-Ax_N}| &\le C e^{-(1/2) Ax_N}A^{-|\alpha|} ,\\
|D^\alpha_{\xi'} \{ (\tau \pd_\tau)^\ell e^{-L_1 x_N}| &\le C e^{-d (|\lambda|^{1/2}+A)x_N} (|\lambda|^{1/2}+A)^{-|\alpha|},\\
|D^\alpha_{\xi'} \{ (\tau \pd_\tau)^\ell e^{-L_2 x_N}| &\le C e^{-d (|\lambda|^{1/2}+1+A)x_N} (|\lambda|^{1/2}+1+A)^{-|\alpha|}\\
&\le C e^{-d (|\lambda|^{1/2}+A)x_N} (|\lambda|^{1/2}+A)^{-|\alpha|},\\
|D^\alpha_{\xi'} \{ (\tau \pd_\tau)^\ell \CM(L_1, A, x_N)| &\le C e^{-d Ax_N}
(|\lambda|^{-1/2} \text{~or~} x_N) A^{-|\alpha|},\\
|D^\alpha_{\xi'} \{ (\tau \pd_\tau)^\ell \CM(L_2, L_1, x_N)| &\le C e^{-d (|\lambda|^{1/2}+A)x_N}
(1 \text{~or~} x_N) (|\lambda|^{1/2}+A)^{-|\alpha|},
\end{align*}
where $d$ is a some positive constant independent of $\alpha$.
\end{lem}
\begin{lem}\label{l.R-bound.int.1d}
Let $a>0$, $\beta \in \mathbb{R}\setminus\{0\}$, $\epsilon \in (\epsilon_0, \pi/2)$ with $\tan \epsilon_0\ge |\beta|/\sqrt{2}$, let $q\in(1,\infty)$, let $k_1 \in \BM_{0, 1}$ and $k_2 \in \BM_{0, 2}$,
let us define the following operators in $\CL(L_q(\mathbb{R}^N_+))$:
\[
 \begin{array}{l}
    (K_1(\lambda)g)(x)= \int_0^\infty \mathcal{F}^{-1}_{\xi^\prime}[k_2(\lambda,\xi^\prime)Ae^{-A(x_N+y_N)}\widehat{g}(\xi^\prime,y_N)](x^\prime)dy_N, \\
    (K_2(\lambda)g)(x)= \int_0^\infty \mathcal{F}^{-1}_{\xi^\prime}[k_2(\lambda,\xi^\prime)Ae^{-L_i (x_N+y_N)}\widehat{g}(\xi^\prime,y_N)](x^\prime)dy_N~(i=1, 2), \\
    (K_3(\lambda)g)(x)= \int_0^\infty \mathcal{F}^{-1}_{\xi^\prime}[k_1(\lambda,\xi^\prime)|\lambda|^{1/2}e^{-L_i(x_N+y_N)}\widehat{g}(\xi^\prime,y_N)](x^\prime)dy_N~(i=1, 2),\\
	(K_4(\lambda)g)(x)= \int_0^\infty \mathcal{F}^{-1}_{\xi^\prime}[k_2(\lambda,\xi^\prime)A^2\mathcal{M}(L_1,A,x_N+y_N)\widehat{g}(\xi^\prime,y_N)](x^\prime)dy_N, \\
    (K_5(\lambda)g)(x)= \int_0^\infty \mathcal{F}^{-1}_{\xi^\prime}[k_2(\lambda,\xi^\prime)A|\lambda|^\frac{1}{2}\mathcal{M}(L_1,A,x_N+y_N)\widehat{g}(\xi^\prime,y_N)](x^\prime)dy_N, \\   
     (K_6(\lambda)g)(x)= \int_0^\infty \mathcal{F}^{-1}_{\xi^\prime}[k_2(\lambda,\xi^\prime)A\mathcal{M}(L_2,L_1,x_N+y_N)\widehat{g}(\xi^\prime,y_N)](x^\prime)dy_N,\\
    (K_7(\lambda)g)(x)= \int_0^\infty \mathcal{F}^{-1}_{\xi^\prime}[k_1(\lambda,\xi^\prime)|\lambda|^{1/2}\mathcal{M}(L_2,L_1,x_N+y_N)\widehat{g}(\xi^\prime,y_N)](x^\prime)dy_N,\\
   (K_8(\lambda)g)(x)= \int_0^\infty \mathcal{F}^{-1}_{\xi^\prime}[k_2(\lambda,\xi^\prime)A^2\mathcal{M}(L_2,L_1,x_N+y_N)\widehat{g}(\xi^\prime,y_N)](x^\prime)dy_N
\end{array} 
\]
for $\lambda \in \Sigma_\epsilon$.
Then for $\ell=0,1$ and for $j=1,\ldots, 8$ the set
$\left\{(\tau\partial_\tau)^\ell(K_j(\lambda))\mid \lambda\in\Sigma_{\epsilon}\right\}$ 
is $\mathcal{R}$-bounded on $\CL(L_q(\R^N_+))$.
\end{lem}

Lemma \ref{l.R-bound.int.1d} implies the following lemma.

\begin{lem}\label{l.R-bound.int.2d}
Let $a>0$, $\beta \in \mathbb{R}\setminus\{0\}$, $\epsilon \in (\epsilon_0, \pi/2)$ with $\tan \epsilon_0\ge |\beta|/\sqrt{2}$, let $q\in(1,\infty)$, let $k_1 \in  \BM_{-1, 1}$, $k_2 \in  \BM_{-2, 2}$, $k_3 \in  \widetilde\BM_{-2, 2}$, $k_4 \in \widetilde\BM_{-2, 1}$, $k_5 \in \widetilde \BM_{-1, 2}$, and $k_6 \in \widetilde \BM_{-1, 1}$,
let us define the following operators in $\CL(L_q(\mathbb{R}^N_+))$:
\[
 \begin{array}{l}
     (M_1(\lambda)g)(x)= \int_0^\infty \mathcal{F}^{-1}_{\xi^\prime}[k_1(\lambda,\xi^\prime)e^{-L_1(x_N+y_N)}\widehat{g}(\xi^\prime,y_N)](x^\prime)dy_N,\\
    (M_2(\lambda)g)(x)= \int_0^\infty \mathcal{F}^{-1}_{\xi^\prime}[k_2(\lambda,\xi^\prime)Ae^{-L_1(x_N+y_N)}\widehat{g}(\xi^\prime,y_N)](x^\prime)dy_N,\\
    (M_3(\lambda)g)(x)= \int_0^\infty \mathcal{F}^{-1}_{\xi^\prime}[k_3(\lambda,\xi^\prime)Ae^{-L_2(x_N+y_N)}\widehat{g}(\xi^\prime,y_N)](x^\prime)dy_N,\\
    (M_4(\lambda)g)(x)= \int_0^\infty \mathcal{F}^{-1}_{\xi^\prime}[k_4(\lambda,\xi^\prime)|\lambda|^{1/2}e^{-L_2(x_N+y_N)}\widehat{g}(\xi^\prime,y_N)](x^\prime)dy_N,\\    
    (M_5(\lambda)g)(x)= \int_0^\infty \mathcal{F}^{-1}_{\xi^\prime}[k_2(\lambda,\xi^\prime)A^2\mathcal{M}(L_1,A,x_N+y_N)\widehat{g}(\xi^\prime,y_N)](x^\prime)dy_N, \\
    (M_6(\lambda)g)(x)= \int_0^\infty \mathcal{F}^{-1}_{\xi^\prime}[k_2(\lambda,\xi^\prime)A|\lambda|^\frac{1}{2}\mathcal{M}(L_1,A,x_N+y_N)\widehat{g}(\xi^\prime,y_N)](x^\prime)dy_N, \\  
    (M_7(\lambda)g)(x)= \int_0^\infty \mathcal{F}^{-1}_{\xi^\prime}[k_5(\lambda,\xi^\prime)A\mathcal{M}(L_2,L_1,x_N+y_N)\widehat{g}(\xi^\prime,y_N)](x^\prime)dy_N,\\
    (M_8(\lambda)g)(x)= \int_0^\infty \mathcal{F}^{-1}_{\xi^\prime}[k_6(\lambda,\xi^\prime)|\lambda|^{1/2}\mathcal{M}(L_2,L_1,x_N+y_N)\widehat{g}(\xi^\prime,y_N)](x^\prime)dy_N,\\
\end{array} 
\]
for $\lambda \in \Sigma_\epsilon$.
Then for $\ell=0,1$, for $j=1,\ldots, 8$ and for $m, n=1,\ldots, N$ the sets
\begin{align*}
& \left\{(\tau\partial_\tau)^\ell(\lambda M_j(\lambda))\mid \lambda\in\Sigma_{\epsilon, c_0}\right\}, &
&\left\{(\tau\partial_\tau)^\ell(\gamma M_j(\lambda))\mid \lambda\in\Sigma_{\epsilon, c_0}\right\}, 
\\
&  \left\{(\tau\partial_\tau)^\ell(|\lambda|^\frac{1}{2} \pd_m M_j(\lambda))\mid \lambda\in\Sigma_{\epsilon, c_0}\right\}, &
&\left\{(\tau\partial_\tau)^\ell(\pd_m\pd_n M_j(\lambda))\mid \lambda\in\Sigma_{\epsilon, c_0}\right\}
\end{align*}
are $\mathcal{R}$-bounded on $\CL(L_q(\R^N_+))$, where $\pd_n=\pd/\pd x_n$.
\end{lem}

\begin{proof}
The proof is similar to \cite[Lemma 5.6]{ShiS}, thus we mainly consider $M_8(\lambda)$.
For $j, k=1, \dots, N-1$, 
\begin{align*}
&(\lambda, \gamma)[M_8(\lambda)](x)\\
&\quad = \int ^\infty_0 \CF^{-1}_{\xi'}[(\lambda, \gamma) k_6(\lambda, \xi') |\lambda|^{1/2}\CM(L_2, L_1, x_N+y_N) \widehat g(\xi', y_N)](x')\,dy_N,\\
&\lambda^{1/2} \pd_j[M_8(\lambda)](x)\\
&\quad = \int ^\infty_0 \CF^{-1}_{\xi'}[(i\xi_j) \lambda^{1/2}|\lambda|^{1/2} A^{-1} k_6(\lambda, \xi') A \CM(L_2, L_1, x_N+y_N) \widehat g(\xi', y_N)](x')\,dy_N,\\
&\pd_j \pd_k[M_8(\lambda)](x)\\
&\quad = \int ^\infty_0 \CF^{-1}_{\xi'}[i\xi_j i\xi_kA^{-1}|\lambda|^{1/2} k_6(\lambda, \xi') A\CM(L_2, L_1, x_N+y_N) \widehat g(\xi', y_N)](x')\,dy_N.
\end{align*}
Lemma \ref{m} furnished that $\lambda k_6(\lambda, \xi')$, $\gamma k_6(\lambda, \xi') \in \BM_{0, 1}$ and $(i\xi_j) \lambda^{1/2} |\lambda|^{1/2} A^{-1} k_6(\lambda, \xi')$, $i\xi_j i\xi_k A^{-1} |\lambda|^{1/2} k_6(\lambda, \xi') \in \BM_{0, 2}$ provided by $|\lambda| \le c_0$.
Note that
\begin{align*}
\partial_N\mathcal{M}(L_2, L_1, x_N)&=-e^{-L_2x_N}-L_1\mathcal{M}(L_2, L_1, x_N),
\end{align*}
then
\begin{align*}
&\lambda^{1/2}\pd_N [M_8(\lambda)](x)\\
&\quad = -\int ^\infty_0 \CF^{-1}_{\xi'}[\lambda^{1/2} k_6(\lambda, \xi') |\lambda|^{1/2}e^{-L_2(x_N+y_N)}\widehat g(\xi', y_N)](x')\,dy_N\\
& \quad -\int ^\infty_0 \CF^{-1}_{\xi'}[\lambda^{1/2} L_1 k_6(\lambda, \xi') |\lambda|^{1/2}\CM(L_2, L_1, x_N+y_N)\widehat g(\xi', y_N)](x')\,dy_N,\\
&\pd_j \pd_N [M_8(\lambda)](x)\\
&\quad = -\int ^\infty_0 \CF^{-1}_{\xi'}[i\xi_j A^{-1} |\lambda|^{1/2} k_6(\lambda, \xi') Ae^{-L_2(x_N+y_N)}\widehat g(\xi', y_N)](x')\,dy_N\\
& \quad -\int ^\infty_0 \CF^{-1}_{\xi'}[i\xi_j A^{-1} |\lambda|^{1/2} L_1 k_6(\lambda, \xi') A\CM(L_2, L_1, x_N+y_N)\widehat g(\xi', y_N)](x')\,dy_N.
\end{align*}
Lemma \ref{m} implies that $\lambda^{1/2} k_6(\lambda, \xi')$, $\lambda^{1/2} L_1 k_6(\lambda, \xi') \in \BM_{0, 1}$ and
$i\xi_j A^{-1} |\lambda|^{1/2} k_6(\lambda, \xi')$, 
$i\xi_j A^{-1} |\lambda|^{1/2} L_1 k_6(\lambda, \xi') \in \BM_{0, 2}$ provided by $|\lambda| \le c_0$.
Furthermore, since
\begin{align*}
\partial_N^2 \mathcal{M}(L_2, L_1, x_N)&=(L_1+L_2)e^{-L_2x_N}+L_1^2\mathcal{M}(L_2, L_1, x_N)\\
&=(L_1+L_2)e^{-L_2x_N}+(z_1(\lambda)+A^2)\mathcal{M}(L_2, L_1, x_N),
\end{align*}
we can write
\begin{equation}\label{pd n}
\begin{aligned}
&\pd_N^2[M_8(\lambda)](x)\\
&\quad = \int ^\infty_0 \CF^{-1}_{\xi'}[(L_1+L_2) k_6(\lambda, \xi') |\lambda|^{1/2} e^{-L_2(x_N+y_N)}\widehat g(\xi', y_N)](x')\,dy_N\\
& \quad +\int ^\infty_0 \CF^{-1}_{\xi'}[z_1(\lambda) k_6(\lambda, \xi') |\lambda|^{1/2}\CM(L_2, L_1, x_N+y_N)\widehat g(\xi', y_N)](x')\,dy_N\\
& \quad +\int ^\infty_0 \CF^{-1}_{\xi'}[k_6(\lambda, \xi') |\lambda|^{1/2} A^2\CM(L_2, L_1, x_N+y_N)\widehat g(\xi', y_N)](x')\,dy_N,
\end{aligned}
\end{equation}
then we observe that
$\pd^2_N M_8(\lambda)$ satisfies the required properties by Lemma \ref{l.R-bound.int.1d}.

The other operators also satisfy the required properties by the same method as above, which completes the proof of Lemma \ref{l.R-bound.int.2d}.
\end{proof}
Thanks to the above lemma, it is sufficient to verify the following lemma to obtain the $\CR$-boundedness for the solution operator of the velocity.
\begin{lem}\label{lem:sol form}
Let $\epsilon \in (\epsilon_0, \pi/2)$ with $\tan \epsilon_0\ge |\beta|/\sqrt{2}$, and let $c_0$ be a small constant determined in Lemma \ref{lem:zl}.
For $\lambda \in \Sigma_{\epsilon, c_0}$, $M_j(\lambda)$ is the same operator in Lemma \ref{l.R-bound.int.2d} with 
$k_1 \in  \BM_{-1, 1}$, $k_2 \in \BM_{-2, 2}$, $k_3 \in  \widetilde\BM_{-2, 2}$, $k_4 \in \widetilde\BM_{-2, 1}$, $k_5 \in \widetilde \BM_{-1, 2}$, and $k_6 \in \widetilde \BM_{-1, 1}$.
Then the solution $\bu$ of \eqref{r} is given by a linear combination of $M_j$ for $j=1, \dots, 8$.
\end{lem}

\begin{proof}
According to \cite[proof of Theorem 3.4.5]{BM}, we know that the solution formula of $u_k$ for $k=1, \dots, N-1$,
\begin{align*} 
u_k&=\sum_{m=1}^{14}\mathcal{U}_k^m\\
&= -\int_0^\infty \mathcal{F}_{\xi^\prime}^{-1}\left[\partial_{y_N}\left(\widehat{h}_k(\xi^\prime,y_N) e^{-L_1(x_N+y_N)}\right)\right]dy_N\\
&-\int_0^\infty \mathcal{F}_{\xi^\prime}^{-1}\left[\partial_{y_N}\left(\frac{i\xi_k\CE(L_1-A)}{\lambda\mathcal{C}_aA}i\xi^\prime\cdot \widehat{h}(\xi^\prime,y_N) \mathcal{M}(L_1,A,x_N+y_N)\right)\right]dy_N\\
& - \int_0^\infty \mathcal{F}_{\xi^\prime}^{-1}\left[\partial_{y_N}\left(\frac{i\xi_k}{\lambda\mathcal{C}_aA}A^2(L_1-A)\widehat{H}_{NN}(\xi^\prime,y_N)\mathcal{M}(L_1,A,x_N+y_N)\right)\right]dy_N\\
& -\sum_{j=1}^{N-1}\int_0^\infty\mathcal{F}_{\xi^\prime}^{-1}\left[\partial_{y_N}\left(\frac{i\xi_k}{\lambda\mathcal{C}_aA}\frac{(L_1-A)(B_a^2+A^2)}{B_a}i\xi_j\widehat{H}_{jN}(\xi^\prime,y_N)\mathcal{M}(L_1,A,x_N+y_N)\right)\right]dy_N\\
& -\sum_{j,l=1}^{N-1}\int_0^\infty \mathcal{F}_{\xi^\prime}^{-1}\left[\partial_{y_N}\left(\frac{i\xi_k}{\lambda\mathcal{C}_aA}(L_1-A)i\xi_ji\xi_l\widehat{H}_{j\ell}(\xi^\prime,y_N)\mathcal{M}(L_1,A,x_N+y_N)\right)\right]dy_N\\
& -\int_0^\infty \mathcal{F}_{\xi^\prime}^{-1}\left[\partial_{y_N}\left(\frac{(B_a^2-L_1^2)(B_a^2-L_2^2)}{\mathcal{A}_a}E_k^hi\xi^\prime\cdot \widehat{h}(\xi^\prime,y_N)\mathcal{M}(L_2,L_1,x_N+y_N)\right)\right]dy_N\\
& - \int_0^\infty \mathcal{F}_{\xi^\prime}^{-1}\left[\partial_{y_N}\left(\frac{(B_a^2-L_1^2)(B_a^2-L_2^2)}{\mathcal{A}_a}E_k^{H_{NN}}\widehat{H}_{NN}(\xi^\prime,y_N)\mathcal{M}(L_2,L_1,x_N+y_N)\right)\right]dy_N\\
& - \sum_{j=1}^{N-1}\int_0^\infty \mathcal{F}_{\xi^\prime}^{-1}\left[\partial_{y_N}\left(\frac{(B_a^2-L_1^2)(B_a^2-L_2^2)}{\mathcal{A}_a}E_k^{H_{jN}}\widehat{H}_{jN}(\xi^\prime,y_N)\mathcal{M}(L_2,L_1,x_N+y_N)\right)\right]dy_N\\
& - \sum_{j,l=1}^{N-1}\int_0^\infty \mathcal{F}_{\xi^\prime}^{-1}\left[\partial_{y_N}\left(\frac{(B_a^2-L_1^2)(B_a^2-L_2^2)}{\mathcal{A}_a}E_k^{H_{j\ell}}\widehat{H}_{j\ell}(\xi^\prime,y_N)\mathcal{M}(L_2,L_1,x_N+y_N)\right)\right]dy_N\\
& +\int_0^\infty \mathcal{F}^{-1}_{\xi^\prime}\left[\partial_{y_N}\left(\frac{L_1(B_aL_1-A^2)(B_a^2-L_2^2)}{\mathcal{A}_a}\widehat{h}_k(\xi^\prime,y_N)\mathcal{M}(L_2,L_1,x_N+y_N)\right)\right]dy_N\\
& - \int_0^\infty \mathcal{F}_{\xi^\prime}^{-1}\left[\partial_{y_N}\left(\frac{i\xi_k\CE L_1(B_aL_1-A^2)(B_a^2-L_2^2)}{\lambda \mathcal{C}_aA\mathcal{A}_a}i\xi^\prime\cdot \widehat{h}(\xi^\prime,y_N) \mathcal{M}(L_2,L_1,x_N+y_N)\right)\right]dy_N\\
& - \int_0^\infty \mathcal{F}_{\xi^\prime}^{-1}\left[\partial_{y_N}\left(\frac{i\xi_k}{\lambda\mathcal{C}_aA}\frac{L_1(B_aL_1-A^2)(B_a^2-L_2^2)A^2}{\mathcal{A}_a}\widehat{H}_{NN}(\xi^\prime,y_N)\mathcal{M}(L_2,L_1,x_N+y_N)\right)\right]dy_N\\
& -\sum_{j=1}^{N-1}\int_0^\infty\mathcal{F}_{\xi^\prime}^{-1}\left[\partial_{y_N}\left(\frac{i\xi_k}{\lambda\mathcal{C}_aA}\frac{L_1(B_aL_1-A^2)(B_a^2-L_2^2)}{\mathcal{A}_a}\frac{(B_a^2+A^2)}{B_a}i\xi_j\widehat{H}_{jN}(\xi^\prime,y_N)\mathcal{M}(L_2,L_1,x_N+y_N)\right)\right]dy_N\\
& -\sum_{j,l=1}^{N-1}\int_0^\infty \mathcal{F}_{\xi^\prime}^{-1}\left[\partial_{y_N}\left(\frac{i\xi_k}{\lambda\mathcal{C}_aA}\frac{L_1(B_aL_1-A^2)(B_a^2-L_2^2)}{\mathcal{A}_a}i\xi_ji\xi_l\widehat{H}_{jl}(\xi^\prime,y_N)\mathcal{M}(L_2,L_1,x_N+y_N)\right)\right]dy_N. 
\end{align*}
Let us consider the first term:
\[
\mathcal{U}^1_k=-\int_0^\infty \mathcal{F}^{-1}_{\xi^\prime}\left[e^{-L_1(x_N+y_N)}\partial_{y_N}\widehat{h}_k(\xi^\prime,y_N)-L_1 e^{-L_1(x_N+y_N)}\widehat{h}_k\right]dy_N. 
\]
The identity
\begin{equation}\label{i}
1=\frac{z_1(\lambda)+A^2}{L_1^2}
\end{equation} 
gives us 
\begin{align*}
\mathcal{U}_k^1
&=-\int_0^\infty \mathcal{F}^{-1}_{\xi^\prime}
\left[\frac{z_1(\lambda)}{\lambda^{1/2}L_1^2} e^{-L_1(x_N+y_N)} \lambda^{1/2} \partial_{y_N}\widehat{h}_k(\xi^\prime,y_N)\right]dy_N\\
& + \sum^{N-1}_{j=1}\int_0^\infty \mathcal{F}^{-1}_{\xi^\prime}
\left[\frac{1}{L_1^2} \frac{i\xi_j}{A} A e^{-L_1(x_N+y_N)} i\xi_j \partial_{y_N}\widehat{h}_k(\xi^\prime,y_N)\right]dy_N\\
& +\int_0^\infty \mathcal{F}_{\xi^\prime}^{-1}\left[\frac{z_1(\lambda)}{L_1 \lambda} e^{-L_1(x_N+y_N)}\lambda \widehat{h}_k\right]dy_N
- \sum^{N-1}_{j=1}\int_0^\infty \mathcal{F}_{\xi^\prime}^{-1}\left[\frac{1}{L_1} e^{-L_1(x_N+y_N)} (i\xi_j)^2\widehat{h}_k\right]dy_N. 
\end{align*}
Lemma \ref{m} implies that
\[
\frac{z_1(\lambda)}{\lambda^{1/2}L_1^2}, ~\frac{z_1(\lambda)}{L_1 \lambda}, ~\frac{1}{L_1} \in \BM_{-1, 1},
\quad
\frac{1}{L_1^2}\frac{i\xi_j}{A} \in \BM_{-2, 2}.
\]


We will repeat a similar argument for all the terms. 
In particular, we consider the second term:
\begin{equation}\label{U2}
\begin{aligned}
\mathcal{U}^2_k
&=-\int_0^\infty\mathcal{F}^{-1}_{\xi^\prime}\left[\frac{i\xi_k \CE (L_1-A)}{\lambda\mathcal{C}_aA}i\xi^\prime\cdot \partial_{y_N}\widehat{h}(\xi^\prime,y_N)\CM(L_1,A,x_N+y_N)\right]dy_N\\
& - \int_0^\infty\mathcal{F}^{-1}_{\xi^\prime}\left[\frac{i\xi_k\CE(L_1-A)}{\lambda\CC_aA}i\xi^\prime\cdot \widehat{h}(\xi^\prime,y_N)\left(-e^{-L_1(x_N+y_N)}-A\mathcal{M}(L_1,A,x_N+y_N)\right)\right]dy_N,
\end{aligned}
\end{equation}
where we used the identity
\[
\partial_{x_N}\mathcal{M}(L_1,A,x_N)=-e^{-L_1x_N}-A\mathcal{M}(L_1,A,x_N). 
\]
The first term on the right-hand side of \eqref{U2} can be rewritten by 
Lemma \ref{lem:e} and
\begin{equation}\label{i3}
	L_1-A=\frac{z_1(\lambda)}{L_1+A}
\end{equation}
as follows: 
\begin{align*}
&-\int_0^\infty\mathcal{F}^{-1}_{\xi^\prime}\left[\frac{i\xi_k \CE (L_1-A)}{\lambda\mathcal{C}_aA}i\xi^\prime\cdot \partial_{y_N}\widehat{h}(\xi^\prime,y_N)\CM(L_1,A,x_N+y_N)\right]dy_N\\
&=
-\int_0^\infty\mathcal{F}^{-1}_{\xi^\prime}\left[\frac{1}{\lambda \CC_a} \frac{i \xi_k}{A} (Am_0 + m_1) \frac{\lambda+1}{\lambda} \frac{z_1(\lambda)}{L_1+A} \mathcal{M}(L_1,A,x_N+y_N) \partial_{y_N} i\xi' \cdot \widehat{h}(\xi^\prime,y_N)\right]dy_N\\
&=: I+II,
\end{align*}
where 
\begin{align*}
	I&=-\int_0^\infty\mathcal{F}^{-1}_{\xi^\prime}\left[\frac{1}{\lambda \CC_a} \frac{i \xi_k}{A} Am_0 \frac{\lambda+1}{\lambda} \frac{z_1(\lambda)}{L_1+A} \mathcal{M}(L_1,A,x_N+y_N) \partial_{y_N} i\xi' \cdot \widehat{h}(\xi^\prime,y_N)\right]dy_N,\\
	II&=-\int_0^\infty\mathcal{F}^{-1}_{\xi^\prime}\left[\frac{1}{\lambda \CC_a} \frac{i \xi_k}{A} m_1 \frac{\lambda+1}{\lambda} \frac{z_1(\lambda)}{L_1+A} \mathcal{M}(L_1,A,x_N+y_N) \partial_{y_N} i\xi' \cdot \widehat{h}(\xi^\prime,y_N)\right]dy_N
\end{align*}
with $m_0 \in \BM_{0, 2}$ and $m_1 \in \widetilde \BM_{1, 2}$.
Using the identity \eqref{i}
for $I$, it follows that
\begin{align*}
	I &=-\int_0^\infty\mathcal{F}^{-1}_{\xi^\prime}\bigg[\frac{1}{\lambda \CC_a} \frac{i \xi_k}{A} \frac{z_1(\lambda)}{|\lambda|^{1/2} L_1} m_0 \frac{(\lambda+1)z_1(\lambda)}{\lambda} \frac{1}{L_1(L_1+A)} \\
	&\enskip \times A |\lambda|^\frac{1}{2}\mathcal{M}(L_1,A,x_N+y_N) \partial_{y_N} i\xi' \cdot \widehat{h}(\xi^\prime,y_N)\bigg]dy_N\\
&\enskip -\int_0^\infty\mathcal{F}^{-1}_{\xi^\prime}\left[\frac{1}{\lambda \CC_a} \frac{i \xi_k}{A} \frac{A}{L_1} m_0 \frac{(\lambda+1)z_1(\lambda)}{\lambda} \frac{1}{L_1(L_1+A)} A^2
\mathcal{M}(L_1,A,x_N+y_N) \partial_{y_N} i\xi' \cdot \widehat{h}(\xi^\prime,y_N)\right]dy_N.
\end{align*}
In view of $m_1/L_2 \in \BM_{0, 2}$, we use the identity
\begin{equation}\label{i2}
	1=\frac{z_2(\lambda)+A^2}{L_2^2}
\end{equation}
for $II$, then it follows that
\begin{align*}
	II &=-\int_0^\infty\mathcal{F}^{-1}_{\xi^\prime}\left[\frac{1}{\lambda \CC_a} \frac{i \xi_k}{A} \frac{m_1}{L_2} \frac{(\lambda+1)z_1(\lambda)}{\lambda} \frac{z_2(\lambda)}{L_2}\frac{1}{L_1+A} \mathcal{M}(L_1,A,x_N+y_N) \partial_{y_N} i\xi' \cdot \widehat{h}(\xi^\prime,y_N)\right]dy_N\\
&\enskip -\int_0^\infty\mathcal{F}^{-1}_{\xi^\prime}\left[\frac{1}{\lambda \CC_a} \frac{i \xi_k}{A} \frac{m_1}{L_2} \frac{(\lambda+1)z_1(\lambda)}{\lambda} \frac{1}{L_2(L_1+A)} A^2 \mathcal{M}(L_1,A,x_N+y_N) \partial_{y_N} i\xi' \cdot \widehat{h}(\xi^\prime,y_N)\right]dy_N.
\end{align*}
The multipliers of $I$ and the second term of $II$
\begin{align*}
&\frac{1}{\lambda \CC_a} \frac{i \xi_k}{A} \frac{z_1(\lambda)}{|\lambda|^{1/2} L_1} m_0 \frac{(\lambda+1)z_1(\lambda)}{\lambda} \frac{1}{L_1(L_1+A)},
\\
&\frac{1}{\lambda \CC_a} \frac{i \xi_k}{A} \frac{A}{L_1} m_0 \frac{(\lambda+1)z_1(\lambda)}{\lambda} \frac{1}{L_1(L_1+A)},
\\
&\frac{1}{\lambda \CC_a} \frac{i \xi_k}{A} \frac{m_1}{L_2} \frac{(\lambda+1)z_1(\lambda)}{\lambda} \frac{1}{L_2(L_1+A)} 
\end{align*}
belong to $\BM_{-2, 2}$
by Lemma \ref{m2} \eqref{m2-3}, Corollary \ref{cor:ca}, and Lemma \ref{lem:zl} \eqref{z}.
Now, we consider the first term of $II$.
It seems difficult to obtain $z_2(\lambda)/(|\lambda|^{1/2}L_2) \in \BM_{0, 2}$ as well as $I$ because $|z_2(\lambda)| \sim |\lambda|+1$. 
Thus, we need the estimate of $\lambda^{3/2} \CC_a$ for the low frequency of $|\xi'|$ in Lemma \ref{lem:ca} \eqref{bound ca 2}.
Since $|\lambda| \neq 0$, there exists small number $\omega_0>0$ such that $0 < \omega_0 < |\lambda|/R$. 
Then we divide the region of $|\xi'|$ into two cases: $\omega_0 \le A^2$ and $0 < A^2 \le |\lambda|/R$.
In the case that $\omega_0 \le A^2$, we rewrite
\begin{align*}
&\int_0^\infty\mathcal{F}^{-1}_{\xi^\prime}\left[\frac{1}{\lambda \CC_a} \frac{i \xi_k}{A} \frac{m_1}{L_2} \frac{(\lambda+1)z_1(\lambda)}{\lambda} \frac{z_2(\lambda)}{L_2}\frac{1}{L_1+A}\mathcal{M}(L_1,A,x_N+y_N) \partial_{y_N} i\xi' \cdot \widehat{h}(\xi^\prime,y_N)\right]dy_N\\
&=\int_0^\infty\mathcal{F}^{-1}_{\xi^\prime}\left[\frac{1}{\lambda \CC_a} \frac{i \xi_k}{A} \frac{m_1}{L_2} \frac{(\lambda+1)z_1(\lambda)}{\lambda}\frac{z_2(\lambda)}{L_2}\frac{1}{L_1+A} \frac{1}{A^2} A^2 \mathcal{M}(L_1,A,x_N+y_N) \partial_{y_N} i\xi' \cdot \widehat{h}(\xi^\prime,y_N)\right]dy_N.
\end{align*}
Since $(L_1+A)^{-1} \in \BM_{0, 2}$ if $1/A^2 \le 1/\omega_0$, we observe that 
\[
\frac{1}{\lambda \CC_a} \frac{i \xi_k}{A} \frac{m_1}{L_2} \frac{(\lambda+1)z_1(\lambda)}{\lambda} \frac{z_2(\lambda)}{L_2}\frac{1}{L_1+A} \frac{1}{A^2} \in \BM_{-2, 2}.
\]
In the case that $A^2 \le |\lambda|/R$, using \eqref{i}
\begin{align*}
&\int_0^\infty\mathcal{F}^{-1}_{\xi^\prime}\left[\frac{1}{\lambda \CC_a} \frac{i \xi_k}{A} \frac{m_1}{L_2} \frac{(\lambda+1)z_1(\lambda)}{\lambda} \frac{z_2(\lambda)}{L_2}\frac{1}{L_1+A} \mathcal{M}(L_1,A,x_N+y_N) \partial_{y_N} i\xi' \cdot \widehat{h}(\xi^\prime,y_N)\right]dy_N\\
&=\sum^{N-1}_{j=1} \int_0^\infty\mathcal{F}^{-1}_{\xi^\prime}\bigg[\frac{1}{\lambda^{3/2} \CC_a} \frac{i \xi_k}{A} \frac{m_1}{L_2} \frac{(\lambda+1)z_1(\lambda)}{\lambda} \frac{z_2(\lambda)}{L_2}\frac{\lambda^{1/2}}{L_1+A} \frac{1}{L_1^2} \frac{z_1(\lambda)}{|\lambda|^{1/2}\lambda^{1/2}} \frac{i\xi_j}{A} \\
&\quad \times A |\lambda|^{1/2}\mathcal{M}(L_1,A,x_N+y_N) \lambda^{1/2} \partial_{y_N} \widehat{h}_j(\xi^\prime,y_N)\bigg]dy_N\\
&+\enskip \int_0^\infty\mathcal{F}^{-1}_{\xi^\prime}\left[\frac{1}{\lambda^{3/2} \CC_a} \frac{i \xi_k}{A} \frac{m_1}{L_2} \frac{(\lambda+1)z_1(\lambda)}{\lambda} \frac{z_2(\lambda)}{L_2}\frac{\lambda^{1/2}}{L_1+A} \frac{1}{L_1^2} A^2 \mathcal{M}(L_1,A,x_N+y_N) \partial_{y_N} i\xi' \cdot \widehat{h}(\xi^\prime,y_N)\right]dy_N,
\end{align*}
then Lemma \ref{lem:ca} implies that 
\begin{align*}
&\frac{1}{\lambda^{3/2} \CC_a} \frac{i \xi_k}{A} \frac{m_1}{L_2} \frac{(\lambda+1)z_1(\lambda)}{\lambda} \frac{z_2(\lambda)}{L_2}\frac{\lambda^{1/2}}{L_1+A} \frac{1}{L_1^2} \frac{z_1(\lambda)}{|\lambda|^{1/2}\lambda^{1/2}} \frac{i\xi_j}{A},
\\
&\frac{1}{\lambda^{3/2} \CC_a} \frac{i \xi_k}{A} \frac{m_1}{L_2} \frac{(\lambda+1)z_1(\lambda)}{\lambda} \frac{z_2(\lambda)}{L_2}\frac{\lambda^{1/2}}{L_1+A} \frac{1}{L_1^2} 
\end{align*}
belong to $\BM_{-2, 2}$ for $\lambda \in \Sigma_{\epsilon, c_0}$, which completes the proof for the first term on the right-hand side of \eqref{U2}.

Similarly, we consider the second term on the right-hand side of \eqref{U2}.
\begin{align*}
&\int_0^\infty\mathcal{F}^{-1}_{\xi^\prime}\left[\frac{i\xi_k\CE(L_1-A)}{\lambda\CC_aA}i\xi^\prime\cdot \widehat{h}(\xi^\prime,y_N) e^{-L_1(x_N+y_N)}\right]dy_N\\
&= \int_0^\infty\mathcal{F}^{-1}_{\xi^\prime}\bigg[\frac{1}{\lambda \CC_a} \frac{i \xi_k}{A} \frac{z_1(\lambda)}{\lambda^{1/2}L_1} m_0 \frac{(\lambda+1)z_1(\lambda)}{\lambda} \frac{1}{L_1(L_1+A)}
A e^{-L_1(x_N+y_N)} \lambda^{1/2} i\xi' \cdot \widehat{h}(\xi^\prime,y_N)\bigg]dy_N\\
&\enskip - \sum_{j=1}^{N-1} \int_0^\infty\mathcal{F}^{-1}_{\xi^\prime}\bigg[\frac{1}{\lambda \CC_a} \frac{i \xi_k}{A} \frac{A}{L_1}\frac{i\xi_j}{A} m_0 \frac{(\lambda+1)z_1(\lambda)}{\lambda}\frac{1}{L_1(L_1+A)}
A e^{-L_1(x_N+y_N)} i\xi_j i\xi' \cdot \widehat{h}(\xi^\prime,y_N)\bigg]dy_N\\
&\enskip + \int_0^\infty\mathcal{F}^{-1}_{\xi^\prime}\left[\frac{1}{\lambda \CC_a} \frac{i \xi_k}{A} \frac{m_1}{L_2} \frac{(\lambda+1)z_1(\lambda)}{\lambda} \frac{z_2(\lambda)}{L_2(L_1+A)} e^{-L_1(x_N+y_N)} i\xi' \cdot \widehat{h}(\xi^\prime,y_N)\right]dy_N\\
&\enskip -\sum_{j=1}^{N-1}\int_0^\infty\mathcal{F}^{-1}_{\xi^\prime}\left[\frac{1}{\lambda \CC_a} \frac{i \xi_k}{A} \frac{m_1}{L_2} \frac{(\lambda+1)z_1(\lambda)}{\lambda} \frac{1}{L_2(L_1+A)} \frac{i\xi_j}{A} A e^{-L_1(x_N+y_N)} i\xi_j i\xi' \cdot \widehat{h}(\xi^\prime,y_N)\right]dy_N,
\end{align*}
then
we see that the multipliers of the first, second, and fourth terms on the right hand side:
\begin{align*}
&\frac{1}{\lambda \CC_a} \frac{i \xi_k}{A} \frac{z_1(\lambda)}{\lambda^{1/2}L_1} m_0 \frac{(\lambda+1)z_1(\lambda)}{\lambda} \frac{1}{L_1(L_1+A)},
\\
&\frac{1}{\lambda \CC_a} \frac{i \xi_k}{A} \frac{A}{L_1}\frac{i\xi_j}{A} m_0 \frac{(\lambda+1)z_1(\lambda)}{\lambda}\frac{1}{L_1(L_1+A)},
\\
&\frac{1}{\lambda \CC_a} \frac{i \xi_k}{A} \frac{m_1}{L_2} \frac{(\lambda+1)z_1(\lambda)}{\lambda} \frac{1}{L_2(L_1+A)} \frac{i\xi_j}{A}  
\end{align*}
belong to $\BM_{-2, 2}$
by Lemma \ref{m2} \eqref{m2-3}, Corollary \ref{cor:ca}, and Lemma \ref{lem:zl} \eqref{z}.
The multiplier of the third term on the right-hand side also belongs to $\BM_{-2, 2}$ by the same way as the first term of \eqref{U2}.

The third term on the right-hand side of \eqref{U2}:
\[
\int_0^\infty\mathcal{F}^{-1}_{\xi^\prime}\left[\frac{i\xi_k\CE(L_1-A)}{\lambda\CC_aA}i\xi^\prime\cdot \widehat{h}(\xi^\prime,y_N) A\CM(L_1,A,x_N+y_N)\right]dy_N
\]
can be considered by the same way as above.

Similarly, we can treat $\CU^3_k$, $\CU^4_k$, and $\CU^5_k$.
Next, we consider $\CU^6_k$.
Note that
\[
\partial_{x_N}\mathcal{M}(L_2, L_1, x_N)=-e^{-L_2x_N}-L_1\mathcal{M}(L_2, L_1, x_N),
\]
then
we have
\begin{align*} 
\mathcal{U}_k^6&= -\int_0^\infty\mathcal{F}_{\xi'}^{-1}
\left[\frac{(B_a^2-L_1^2)(B_a^2-L_2^2)}{\CA_a} E^h_k \CM(L_2, L_1, x_N+y_N) i\xi' \cdot \pd_{y_N} \widehat{h}(\xi^\prime,y_N) \right]dy_N\\
&-\int_0^\infty\mathcal{F}_{\xi'}^{-1}
\bigg[\frac{(B_a^2-L_1^2)(B_a^2-L_2^2)}{\CA_a} E^h_k \left(-e^{-L_2(x_N+y_N)}- L_1 \CM(L_2, L_1, x_N)\right) i\xi' \cdot \widehat{h}(\xi^\prime,y_N) \bigg]dy_N.
\end{align*}
Lemma \ref{lem:e} and Lemma \ref{lem:eh} imply that
\[
E^h_k = i\xi_k \frac{1}{\lambda \CC_a}(Am_0+m_1)\frac{\lambda+1}{\lambda} +i\xi_k \frac{m_2}{\lambda+a},
\]
where $m_0 \in \BM_{0, 2}$, $m_1 \in \widetilde \BM_{1, 2}$, $m_2 \in \widetilde \BM_{1, 1}$.
Thus, the first term on the right-hand side of $\CU^6_k$ is written by
\begin{align*} 
&-\int_0^\infty\mathcal{F}_{\xi'}^{-1}
\left[\frac{(B_a^2-L_1^2)(B_a^2-L_2^2)}{\CA_a} E^h_k \CM(L_2, L_1, x_N+y_N) i\xi' \cdot \pd_{y_N} \widehat{h}(\xi^\prime,y_N) \right]dy_N\\
&=-\int_0^\infty\mathcal{F}_{\xi'}^{-1}
\bigg[\frac{(B_a^2-L_1^2)(B_a^2-L_2^2)}{\CA_a} \left(i\xi_k \frac{1}{\lambda \CC_a}(Am_0+m_1)\frac{\lambda+1}{\lambda} +i\xi_k \frac{m_2}{\lambda+a}\right)\\
&\qquad \times \CM(L_2, L_1, x_N+y_N) i\xi' \cdot \pd_{y_N} \widehat{h}(\xi^\prime,y_N) \bigg]dy_N\\
&=-\int_0^\infty\mathcal{F}_{\xi'}^{-1}
\left[\left(\frac{B_a^2-L_1^2}{\CA_a} i\xi_k\right) \frac{1}{\lambda \CC_a} \frac{(\lambda+1)(B_a^2-L_2^2)}{\lambda} m_0 A \CM(L_2, L_1, x_N+y_N) i\xi' \cdot \pd_{y_N} \widehat{h}(\xi^\prime,y_N) \right]dy_N\\
& -\int_0^\infty\mathcal{F}_{\xi'}^{-1}
\left[\left(\frac{B_a^2-L_1^2}{\CA_a}m_1\right) \frac{i\xi_k}{A} \frac{1}{\lambda \CC_a} \frac{(\lambda+1)(B_a^2-L_2^2)}{\lambda} A\CM(L_2, L_1, x_N+y_N) i\xi' \cdot \pd_{y_N} \widehat{h}(\xi^\prime,y_N) \right]dy_N\\
& -\int_0^\infty\mathcal{F}_{\xi'}^{-1}
\left[\left(\frac{B_a^2-L_1^2}{\CA_a} m_2 \right) \frac{i\xi_k}{A} \frac{B_a^2-L_2^2}{\lambda+a} A\CM(L_2, L_1, x_N+y_N) i\xi' \cdot \pd_{y_N} \widehat{h}(\xi^\prime,y_N) \right]dy_N.
\end{align*}
Lemma \ref{m}, Corollary \ref{cor:ca}, Lemma \ref{lem:aa}, and Lemma \ref{lem:zl} \eqref{z} furnish that
the multipliers
\begin{align*}
&\left(\frac{B_a^2-L_1^2}{\CA_a} i\xi_k\right) \frac{1}{\lambda \CC_a} \frac{(\lambda+1)(B_a^2-L_2^2)}{\lambda} m_0,
\\
&\left(\frac{B_a^2-L_1^2}{\CA_a}m_1\right) \frac{i\xi_k}{A} \frac{1}{\lambda \CC_a} \frac{(\lambda+1)(B_a^2-L_2^2)}{\lambda} ,\\
&
\left(\frac{B_a^2-L_1^2}{\CA_a} m_2 \right) \frac{i\xi_k}{A} \frac{B_a^2-L_2^2}{\lambda+a} 
\end{align*}
belong to $\widetilde \BM_{-1, 2}$.
The identity 
\eqref{i2}
gives us
\begin{equation}\label{U6-2}
\begin{aligned}
&\int_0^\infty\mathcal{F}_{\xi'}^{-1}
\bigg[\frac{(B_a^2-L_1^2)(B_a^2-L_2^2)}{\CA_a} E^h_k e^{-L_2(x_N+y_N)} i\xi' \cdot \widehat{h}(\xi^\prime,y_N) \bigg]dy_N\\
&= \int_0^\infty\mathcal{F}_{\xi'}^{-1}
\bigg[\frac{B_a^2-L_1^2}{\CA_a} \frac{1}{\lambda \CC_a} \frac{(\lambda+1)(B_a^2-L_2^2)}{\lambda} m_0
A e^{-L_2(x_N+y_N)} i\xi_k i\xi' \cdot \widehat{h}(\xi^\prime,y_N) \bigg]dy_N\\
&\enskip + \int_0^\infty\mathcal{F}_{\xi'}^{-1}
\bigg[\frac{B_a^2-L_1^2}{\CA_a}  \frac{1}{\lambda \CC_a} \frac{(\lambda+1)(B_a^2-L_2^2)}{\lambda} i\xi_k \frac{m_1}{L_2} \frac{z_2(\lambda)}{L_2} 
e^{-L_2(x_N+y_N)} i\xi' \cdot \widehat{h}(\xi^\prime,y_N) \bigg]dy_N\\
&\enskip + \int_0^\infty\mathcal{F}_{\xi'}^{-1}
\bigg[\frac{B_a^2-L_1^2}{\CA_a}  \frac{1}{\lambda \CC_a} \frac{(\lambda+1)(B_a^2-L_2^2)}{\lambda} \frac{m_1}{L_2} \frac{A}{L_2} 
A e^{-L_2(x_N+y_N)} i\xi_k i\xi' \cdot \widehat{h}(\xi^\prime,y_N) \bigg]dy_N\\
&\enskip + \int_0^\infty\mathcal{F}_{\xi'}^{-1}
\bigg[\frac{B_a^2-L_1^2}{\CA_a} \frac{B_a^2-L_2^2}{\lambda^{1/2}(\lambda+a)} \frac{i\xi_k}{A} \frac{m_2}{L_2} \frac{z_2(\lambda)}{L_2} 
A e^{-L_2(x_N+y_N)} \lambda^{1/2} i\xi' \cdot \widehat{h}(\xi^\prime,y_N) \bigg]dy_N\\
&\enskip + \int_0^\infty\mathcal{F}_{\xi'}^{-1}
\bigg[\frac{B_a^2-L_1^2}{\CA_a} \frac{B_a^2-L_2^2}{\lambda+a} \frac{m_2}{L_2} \frac{A}{L_2} 
A e^{-L_2(x_N+y_N)} i\xi_k i\xi' \cdot \widehat{h}(\xi^\prime,y_N) \bigg]dy_N.
\end{aligned}
\end{equation}
Dividing three cases for the second term on the right-hand side of \eqref{U6-2} in the same way as \eqref{U2}, 
the multipliers of \eqref{U6-2} belong to $\widetilde \BM_{-2, 2}$.
Similarly, the identity 
\eqref{i}
gives us
\begin{align*}
&\int_0^\infty\mathcal{F}_{\xi'}^{-1}
\bigg[\frac{(B_a^2-L_1^2)(B_a^2-L_2^2)}{\CA_a} E^h_k L_1 \CM(L_2, L_1, x_N+y_N) i\xi' \cdot \widehat{h}(\xi^\prime,y_N) \bigg]dy_N\\
&= \int_0^\infty\mathcal{F}_{\xi'}^{-1}
\bigg[\frac{(B_a^2-L_1^2)(B_a^2-L_2^2)}{\CA_a} L_1 \CM(L_2, L_1, x_N+y_N) i\xi' \cdot \widehat{h}(\xi^\prime,y_N) \bigg]dy_N\\
&= \int_0^\infty\mathcal{F}_{\xi'}^{-1}
\bigg[\left(\frac{B_a^2-L_1^2}{\CA_a}L_1\right) \frac{1}{\lambda \CC_a} \frac{(\lambda+1)(B_a^2-L_2^2)}{\lambda} m_0
A \CM(L_2, L_1, x_N+y_N) i\xi_k i\xi' \cdot \widehat{h}(\xi^\prime,y_N) \bigg]dy_N\\
&\enskip + \int_0^\infty\mathcal{F}_{\xi'}^{-1}
\bigg[\left(\frac{B_a^2-L_1^2}{\CA_a}m_1\right) \frac{1}{\lambda \CC_a} \frac{(\lambda+1)(B_a^2-L_2^2)}{\lambda} \frac{i\xi_k}{A} \frac{z_1(\lambda)}{L_1 \lambda^{1/2}} 
A \CM(L_2, L_1, x_N+y_N) \lambda^{1/2} i\xi' \cdot \widehat{h}(\xi^\prime,y_N) \bigg]dy_N\\
&\enskip + \int_0^\infty\mathcal{F}_{\xi'}^{-1}
\bigg[\left(\frac{B_a^2-L_1^2}{\CA_a}m_1\right) \frac{1}{\lambda \CC_a} \frac{(\lambda+1)(B_a^2-L_2^2)}{\lambda} \frac{A}{L_1} 
A \CM(L_2, L_1, x_N+y_N) i\xi_k i\xi' \cdot \widehat{h}(\xi^\prime,y_N) \bigg]dy_N\\
&\enskip + \int_0^\infty\mathcal{F}_{\xi'}^{-1}
\bigg[\left(\frac{B_a^2-L_1^2}{\CA_a}m_2\right) \frac{B_a^2-L_2^2}{\lambda+a} \frac{i\xi_k}{A} \frac{z_1(\lambda)}{L_1 \lambda^{1/2}} 
A \CM(L_2, L_1, x_N+y_N) \lambda^{1/2} i\xi' \cdot \widehat{h}(\xi^\prime,y_N) \bigg]dy_N\\
&\enskip + \int_0^\infty\mathcal{F}_{\xi'}^{-1}
\bigg[\left(\frac{B_a^2-L_1^2}{\CA_a}m_2\right) \frac{B_a^2-L_2^2}{\lambda+a} \frac{A}{L_1} 
A \CM(L_2, L_1, x_N+y_N) i\xi_k i\xi' \cdot \widehat{h}(\xi^\prime,y_N) \bigg]dy_N,
\end{align*}
then we observe that the multipliers belong to $\widetilde \BM_{-1, 2}$.

Thanks to Lemma \ref{lem:eh} and the identities \eqref{i}, \eqref{i2}, we can treat $\CU^7_k$ and the terms which are written as $m_4$ and $m_5$ in $\CU^8_k$ and $\CU^9_k$ by the same manner as $\CU^6_k$,
therefore we consider the term corresponding to $m_6 \in \widetilde \BM_{1, 1}$ in $\CU^8_k$, namely,
\begin{align*}
&\int_0^\infty \mathcal{F}_{\xi^\prime}^{-1}\left[\partial_{y_N}\left(\frac{(B_a^2-L_1^2)(B_a^2-L_2^2)}{\mathcal{A}_a} m_6 \widehat{H}_{kN}(\xi', y_N)\mathcal{M}(L_2,L_1,x_N+y_N)\right)\right]dy_N\\
&=\int_0^\infty \mathcal{F}_{\xi'}^{-1}\left[\frac{(B_a^2-L_1^2)(B_a^2-L_2^2)}{\mathcal{A}_a} m_6 \mathcal{M}(L_2,L_1,x_N+y_N) \pd_{y_N}\widehat{H}_{kN}(\xi^\prime,y_N) \right]dy_N\\
&-\sum^{N-1}_{j=1}\int_0^\infty \mathcal{F}_{\xi'}^{-1}\left[\frac{(B_a^2-L_1^2)(B_a^2-L_2^2)}{\mathcal{A}_a} m_6 \left(e^{-L_2(x_N+y_N)}+ L_1 \mathcal{M}(L_2,L_1,x_N+y_N) \right) \widehat{H}_{kN}(\xi',y_N) \right]dy_N.
\end{align*}
For the first term,
\begin{align*}
&\int_0^\infty \mathcal{F}_{\xi'}^{-1}\left[\frac{(B_a^2-L_1^2)(B_a^2-L_2^2)}{\mathcal{A}_a} m_6 \mathcal{M}(L_2,L_1,x_N+y_N) \pd_{y_N}\widehat{H}_{kN}(\xi^\prime,y_N) \right]dy_N\\
&=\int_0^\infty \mathcal{F}_{\xi'}^{-1}\left[\left(\frac{B_a^2-L_1^2}{\mathcal{A}_a} m_6 \right) \frac{B_a^2-L_2^2}{|\lambda|^{1/2} \lambda^{1/2}} |\lambda|^{1/2}\mathcal{M}(L_2,L_1,x_N+y_N) \lambda^{1/2}\pd_{y_N}\widehat{H}_{kN}(\xi^\prime,y_N) \right]dy_N,
\end{align*}
then 
\[
\left(\frac{B_a^2-L_1^2}{\mathcal{A}_a} m_6\right) \frac{B_a^2-L_2^2}{|\lambda|^{1/2} \lambda^{1/2}} \in \widetilde \BM_{-1, 1}.
\]
For the second term, by \eqref{i2}
\begin{align*}
&-\sum^{N-1}_{j=1}\int_0^\infty \mathcal{F}_{\xi'}^{-1}\left[\frac{(B_a^2-L_1^2)(B_a^2-L_2^2)}{\mathcal{A}_a} m_6 e^{-L_2(x_N+y_N)} \widehat{H}_{kN}(\xi',y_N) \right]dy_N\\
&=-\sum^{N-1}_{j=1}\int_0^\infty \mathcal{F}_{\xi'}^{-1}\left[\frac{B_a^2-L_1^2}{\mathcal{A}_a} \frac{m_6}{L_2} \frac{z_2(\lambda)}{L_2} \frac{B_a^2-L_2^2}{|\lambda|^{1/2}\lambda^{1/2}}|\lambda|^{1/2}e^{-L_2(x_N+y_N)} \lambda^{1/2} \widehat{H}_{kN}(\xi',y_N) \right]dy_N\\
&\enskip -\sum^{N-1}_{j=1}\int_0^\infty \mathcal{F}_{\xi'}^{-1}\left[\frac{B_a^2-L_1^2}{\mathcal{A}_a} \frac{m_6}{L_2} \frac{A}{L_2} \frac{B_a^2-L_2^2}{\lambda} Ae^{-L_2(x_N+y_N)} \lambda \widehat{H}_{kN}(\xi',y_N) \right]dy_N,
\end{align*}
then
\begin{align*}
\frac{B_a^2-L_1^2}{\mathcal{A}_a} \frac{m_6}{L_2} \frac{z_2(\lambda)}{L_2} \frac{B_a^2-L_2^2}{|\lambda|^{1/2}\lambda^{1/2}} \in \widetilde \BM_{-2, 1},
\enskip
\frac{B_a^2-L_1^2}{\mathcal{A}_a} \frac{m_6}{L_2} \frac{A}{L_2} \frac{B_a^2-L_2^2}{\lambda} \in \widetilde \BM_{-2, 2}
\end{align*}

The third term can be considered in the same way as the second term.

Lemma \ref{lem:bl} helps $\CU^{10}_k$, $\CU^{11}_k$, $\CU^{12}_k$, $\CU^{13}_k$, and $\CU^{14}_k$ to treat by the same way as before.
Therefore, we obtain Lemma \ref{lem:sol form} for $u_k$ with $k=1, \dots, N-1$.

Next, we consider the proof for $u_N$:
\begin{align*}
\widehat{u}_N&=-A_N^0(L_1-A)\mathcal{M}(L_1,A,x_N)+A_N^2(L_2-L_1)\mathcal{M}(L_2,L_1,x_N)\\
&=-\frac{AC(L_1-A)}{\lambda}\mathcal{M}(L_1,A,x_N)+\frac{A(L_1-A)}{\lambda}C\mathcal{M}(L_2,L_1,x_N)+i\xi^\prime\cdot \widehat{h}^\prime\mathcal{M}(L_2,L_1,x_N),
\end{align*}
where
\[
C=\frac{\CE}{\CC_a A}i\xi' \cdot \widehat{h}
+\frac{1}{\CC_a A}\left(A^2 \widehat H_{NN} -\frac{B_a^2+A^2}{B_a} \sum^{N-1}_{j=1} i\xi_j \widehat{H}_{jN}+\sum^{N-1}_{j, k=1} i\xi_j i\xi_k \widehat{H}_{jk}\right).
\]
Similarly to the tangential component of $\bu$, the following solution formula was obtained by \cite{BM},
\begin{align*} 
u_N
&= \int_0^\infty \mathcal{F}_{\xi^\prime}^{-1}
\left[\partial_{y_N}\left((L_1-A)\frac{\CE}{\lambda \CC_a} i\xi' \cdot \widehat{h}(\xi', y_N) \CM(L_1, A, x_N+y_N)\right)\right]dy_N\\
& +\int_0^\infty \mathcal{F}_{\xi^\prime}^{-1}
\left[\partial_{y_N}\left(\frac{L_1-A}{\lambda \CC_a}A^2 \widehat{H}_{NN}(\xi', y_N) \CM(L_1, A, x_N+y_N)\right)\right]dy_N\\
& -\sum^{N-1}_{j=1} \int_0^\infty \mathcal{F}_{\xi^\prime}^{-1}
\left[\partial_{y_N}\left(\frac{L_1-A}{\lambda \CC_a} \frac{B_a^2+A^2}{B_a}i\xi_j \widehat{H}_{jN}(\xi', y_N) \CM(L_1, A, x_N+y_N)\right)\right]dy_N\\
& +\sum^{N-1}_{j, k=1} \int_0^\infty \mathcal{F}_{\xi^\prime}^{-1}
\left[\partial_{y_N}\left(\frac{L_1-A}{\lambda \CC_a} i\xi_j i\xi_k\widehat{H}_{jk}(\xi', y_N) \CM(L_1, A, x_N+y_N)\right)\right]dy_N\\
&-\int_0^\infty \mathcal{F}_{\xi^\prime}^{-1}
\left[\partial_{y_N}\left((L_1-A)\frac{\CE}{\lambda \CC_a} i\xi' \cdot \widehat{h}(\xi', y_N) \CM(L_2, L_1, x_N+y_N)\right)\right]dy_N\\
& -\int_0^\infty \mathcal{F}_{\xi^\prime}^{-1}
\left[\partial_{y_N}\left(\frac{L_1-A}{\lambda \CC_a}A^2 \widehat{H}_{NN}(\xi', y_N) \CM(L_2, L_1, x_N+y_N)\right)\right]dy_N\\
& +\sum^{N-1}_{j=1} \int_0^\infty \mathcal{F}_{\xi^\prime}^{-1}
\left[\partial_{y_N}\left(\frac{L_1-A}{\lambda \CC_a} \frac{B_a^2+A^2}{B_a}i\xi_j \widehat{H}_{jN}(\xi', y_N) \CM(L_2, L_1, x_N+y_N)\right)\right]dy_N\\
& -\sum^{N-1}_{j, k=1} \int_0^\infty \mathcal{F}_{\xi^\prime}^{-1}
\left[\partial_{y_N}\left(\frac{L_1-A}{\lambda \CC_a} i\xi_j i\xi_k\widehat{H}_{jk}(\xi', y_N) \CM(L_2, L_1, x_N+y_N)\right)\right]dy_N\\
& -\int_0^\infty \mathcal{F}_{\xi^\prime}^{-1}
\left[\partial_{y_N}\left(i\xi'\cdot \widehat{h}(\xi', y_N) \CM(L_2, L_1, x_N+y_N)\right)\right]dy_N.
\end{align*}
Noting that 
\eqref{i3}, we can complete the proof of $u_N$ of Lemma \ref{lem:sol form} by the same way as the proof for $u_k$ with $k=1, \dots, N-1$.
\end{proof}
Lemma \ref{l.R-bound.int.2d} furnishes that $\CR$-boundedness for $u_k$ ($k=1, \dots, N-1$) and $u_N$, 
which completes the proof for the velocity part of Theorem \ref{thm:Rbdd}.

\subsection{$\CR$-boundedness for the solution operator of the order parameter}\label{Rbdd Q}
Thanks to the third equation of \eqref{r} and subsection \ref{subsection:v}, it is sufficient to consider the following problem:
\begin{equation}\label{rQ}
\left\{
\begin{aligned}
&(\lambda - \Delta + a) \bQ=\beta \bD(\bu)& \quad&\text{in $\R^N_+$},\\
&\pd_N \bQ=\bH &\quad &\text{on $\R^N_0$}.
\end{aligned}
\right.
\end{equation}
Let
\[
\CS'_\lambda\bu = (\nabla^2\bu, \lambda^{1/2}\nabla\bu) \in L_q(\R^N_+)^{N^3} \times L_q(\R^N_+)^{N^2}.
\]
Recall that
\begin{align*}
\CS_\lambda \bu = (\nabla^2 \bu, \lambda^{1/2}\nabla \bu, \lambda \bu),\quad
\CT_\lambda\bQ = (\nabla^2\bQ, \lambda^{1/2}\nabla\bQ, \lambda \bQ, \nabla \bQ, \lambda^{1/2}\bQ).
\end{align*}
Set 
\[
	E_{even}[\bF(x)] = \left\{
	\begin{aligned}
	&\bF(x) &\quad x_N > 0, \\
	&\bF(x', -x_N) &\quad x_N < 0,
	\end{aligned}
	\right. 
\]
where $x'=(x_1, \dots, x_{N-1})$.
Assuming that $\bQ_1$ satisfies
\begin{equation}\label{Q1}
(\lambda - \Delta +a)\bQ_1=\beta E_{even}[\bD(\bu)] \quad \text{in $\R^N$},
\end{equation}
there exists an operator family $\CH_1(\lambda) \in {\rm Hol} (\Sigma_{\epsilon, c_0}, \CL(L_q(\R^N_+)^{N^3} \times L_q(\R^N_+)^{N^2}, H^3_q(\R^N; \bS_0)))$ 
such that for any $\lambda=\gamma+i\tau \in \Sigma_{\epsilon, c_0}$,
$\bQ_1=\CH_1(\lambda)(\bff_1, \bff_2)$ is a solution of \eqref{Q1} with
\begin{equation}\label{R-bound Q1}
	\CR_{\CL(L_q(\R^N_+)^{N^3} \times L_q(\R^N_+)^{N^2}, B_q(\R^N_+)))}(\{(\tau \pd_\tau)^\ell \CT_\lambda\CH_1(\lambda) \mid \lambda \in \Sigma_{\epsilon, c_0}\}) \le r_1,
\end{equation}
where $\bff_1$ and $\bff_2$ correspond to $\nabla^2\bu$ and $\lambda^{1/2}\nabla \bu$, respectively.
This follows from \cite[Theorem 3.3]{ES}.
In fact, 
set $\bF = \beta E_{even}[\bD(\bu)]$. 
The solution formula of \eqref{Q1}:
\[
	\bQ_1 = \CF^{-1}[(\lambda+|\xi|^2+a)^{-1} \CF[\bF]]
\]
can be formally rewritten in two ways:
\begin{equation}\label{re sol}
\begin{aligned}
	\bQ_1 &= -\sum^N_{j=1} \CF^{-1}[(\lambda+|\xi|^2+a)^{-1} |\xi|^{-2} (i\xi_j)\CF[\pd_j \bF]],\\
	\bQ_1 &= \CF^{-1}[\lambda^{-1/2}(\lambda+|\xi|^2+a)^{-1} \CF[\lambda^{1/2} \bF]],
\end{aligned}
\end{equation}
where $\CF$ and $\CF^{-1}$ are the Fourier transform and its inverse transform, respectively.
In particular, we verify that the symbols of the operators corresponding to $\nabla^3 \bQ_1$ and $\lambda \bQ_1$ satisfy the assumption of \cite[Theorem 3.3]{ES}.
By \eqref{re sol} we have
\begin{equation}\label{sol Q1}
\begin{aligned}
	\pd_k \pd_\ell \pd_m \bQ_1 &= -\sum^N_{j=1} \CF^{-1}[(i\xi_k)(i\xi_\ell)(i\xi_m)(\lambda+|\xi|^2+a)^{-1} |\xi|^{-2} (i\xi_j)\CF[\pd_j \bF]],\\
	\lambda \bQ_1 &= \CF^{-1}[\lambda^{1/2}(\lambda+|\xi|^2+a)^{-1} \CF[\lambda^{1/2} \bF]].
\end{aligned}
\end{equation}
Let $\alpha \in \N^N_0$ be a multi-index, and let $n=0, 1$.
Lemma \ref{spectrum} and Bell's formula imply that there exists a positive constant $C$ such that
for any $(\xi, \lambda) \in \R^N \times \Sigma_\epsilon$ with $\lambda = \gamma+i\tau$
\[
	|D_\xi^\alpha (\tau \pd_\tau)^n \{(\lambda+|\xi|^2+a)^{-1}\}| \le C(|\lambda|^{1/2}+|\xi|+1)^{-1-|\alpha|}.
\]
Therefore, the symbols of \eqref{sol Q1} satisfy
\begin{align*}
	|D_\xi^\alpha (\tau \pd_\tau)^n \{(i\xi_k)(i\xi_\ell)(i\xi_m)(\lambda+|\xi|^2+a)^{-1} |\xi|^{-2} (i\xi_j)\}| &\le C |\xi|^{-|\alpha|},\\
	|D_\xi^\alpha (\tau \pd_\tau)^n \{\lambda^{1/2}(\lambda+|\xi|^2+a)^{-1}\}| &\le C |\xi|^{-|\alpha|}.
\end{align*}
Since the other symbols of $\CT_\lambda \CH_1(\lambda)$ 
are estimated by the same method as the above, we have \eqref{R-bound Q1}.
Thanks to subsection \ref{subsection:v}, there exists the operator family $\CA_2(\lambda) \in {\rm Hol} (\Sigma_{\epsilon, c_0}, \CL(\widetilde X_q(\R^N_+), H^2_q(\R^N_+)))$ 
such that for any $\lambda=\gamma+i\tau \in \Sigma_{\epsilon, c_0}$,
$(\nabla^2\bu, \lambda^{1/2}\nabla \bu)=\CS'_\lambda\CA_2(\lambda)(\CS_\lambda \bh, \CS_\lambda \bH, \lambda^{1/2}\bH)$ with
\[
\CR_{\CL(\widetilde X_q(\R^N_+), L_q(\R^N_+)^{N^3} \times L_q(\R^N_+)^{N^2})}(\{(\tau \pd_\tau)^\ell \CS'_\lambda\CA_2(\lambda) \mid \lambda \in \Sigma_{\epsilon, c_0}\}) \le r_2.
\] 
Therefore, setting $\CH_2(\lambda)=\CH_1(\lambda) \CS'_\lambda \CA_2(\lambda)$, the solution $\bQ_1$ satisfies
\begin{equation}\label{op Q1}
\bQ_1=\CH_2(\lambda)(\CS_\lambda \bh, \CS_\lambda \bH, \lambda^{1/2}\bH)
\end{equation}
with
\begin{equation}\label{R Q1}
\CR_{\CL(\widetilde X_q(\R^N_+), B_q(\R^N_+))}(\{(\tau \pd_\tau)^\ell \CT_\lambda\CH_2(\lambda) \mid \lambda \in \Sigma_{\epsilon, c_0}\}) \le r_1r_2.
\end{equation}

Set $\bQ=\bQ_1+\bQ_2$ and $\bP=\bH-\pd_N\bQ_1$, then $\bQ_2$ satisfies
\begin{equation}\label{Q2}
\left\{
\begin{aligned}
&(\lambda - \Delta + a) \bQ_2=0& \quad&\text{in $\R^N_+$},\\
&\pd_N \bQ_2=\bP& \quad&\text{on $\R^N_0$}.
\end{aligned}
\right.
\end{equation}
The solution formula of \eqref{Q2} is given by
\[
\bQ_2=-\CF^{-1}_{\xi'}\left[\frac{\widehat \bP}{B_a} e^{-B_a x_N}\right],
\]
then using the Volevich trick, 
\begin{equation}\label{V Q2}
\bQ_2=\int^\infty_0\CF^{-1}_{\xi'}\left[\frac{1}{B_a} e^{-B_a (x_N+y_N)} \pd_{y_N} \widehat \bP \right]\,dy_N
-\int^\infty_0\CF^{-1}_{\xi'}\left[ e^{-B_a (x_N+y_N)} \widehat \bP \right]\,dy_N.
\end{equation}
In order to prove $\CR$-boundedness for $\bQ_2$, we verify that the multipliers satisfy the following lemma, which is proved by replacing $|\lambda|^{1/2}$ with $|\lambda|^{1/2}+1$ in the proof of \cite[Lemma 5.4]{ShiS}.

\begin{lem}\label{R-bound Q}
Let $a>0$, $\beta \in \mathbb{R}\setminus\{0\}$, $\epsilon \in (\epsilon_0, \pi/2)$ with $\tan \epsilon_0\ge |\beta|/\sqrt{2}$, let and $q\in(1,\infty)$.
Let $m_1 \in \BM_{0,1}'$, $m_2 \in \BM_{0,2}$ and
let us define the following operators in $\CL(L_q(\mathbb{R}^N_+))$:
\begin{align*}
    (M_1(\lambda)g)(x)&= \int_0^\infty \mathcal{F}^{-1}_{\xi^\prime}[m_1(\lambda,\xi^\prime) (|\lambda|^{1/2}+1) e^{-B_a(x_N+y_N)}\widehat{g}(\xi^\prime,y_N)](x^\prime)dy_N,\\
    (M_2(\lambda)g)(x)&= \int_0^\infty \mathcal{F}^{-1}_{\xi^\prime}[m_2(\lambda,\xi^\prime) A e^{-B_a(x_N+y_N)}\widehat{g}(\xi^\prime,y_N)](x^\prime)dy_N
\end{align*}
Then for $\ell=0,1$ and $k=1, 2$, the sets $\{(\tau \pd_{\tau})^\ell M_k (\lambda) \mid \lambda \in \Sigma_\epsilon\}$ is $\mathcal{R}$-bounded on $\CL(L_q(\R^N_+))$.
\end{lem}
First, we consider $\lambda \bQ_2$ and the tangential derivative of $\bQ_2$. By \eqref{V Q2}, we have
\begin{align*}
\lambda^{1/2} \bQ_2&=\int^\infty_0\CF^{-1}_{\xi'}\left[\frac{\lambda^{1/2}}{B_a(|\lambda|^{1/2}+1)} (|\lambda|^{1/2}+1)e^{-B_a (x_N+y_N)} 
\pd_{y_N}\widehat{\bP}\right]\,dy_N\\
&\enskip -\int^\infty_0\CF^{-1}_{\xi'}\left[\frac{\lambda^{1/2}}{|\lambda|^{1/2}+1} (|\lambda|^{1/2}+1) e^{-B_a (x_N+y_N)} \widehat{\bP}\right]\,dy_N,\\
\lambda \bQ_2&=\int^\infty_0\CF^{-1}_{\xi'}\left[\frac{\lambda^{1/2}}{B_a(|\lambda|^{1/2}+1)} (|\lambda|^{1/2}+1)e^{-B_a (x_N+y_N)} 
\lambda^{1/2} \pd_{y_N}\widehat{\bP}\right]\,dy_N\\
&\enskip -\int^\infty_0\CF^{-1}_{\xi'}\left[\frac{1}{|\lambda|^{1/2}+1} (|\lambda|^{1/2}+1) e^{-B_a (x_N+y_N)} \lambda \widehat{\bP}\right]\,dy_N,\\
\lambda \pd_j \bQ_2
&=\int^\infty_0\CF^{-1}_{\xi'}\left[\frac{\lambda}{B_a(|\lambda|^{1/2}+1)}(|\lambda|^{1/2}+1) e^{-B_a (x_N+y_N)} 
i\xi_j \pd_{y_N}\widehat{\bP}\right]\,dy_N\\
&\enskip -\int^\infty_0\CF^{-1}_{\xi'}\left[\frac{i\xi_j}{A} A e^{-B_a (x_N+y_N)} \lambda \widehat{\bP}\right]\,dy_N,\\
\lambda^{1/2}\pd_j \bQ_2
&=\int^\infty_0\CF^{-1}_{\xi'}\left[\frac{\lambda^{1/2}}{B_a(|\lambda|^{1/2}+1)} (|\lambda|^{1/2}+1) e^{-B_a (x_N+y_N)} 
i\xi_j \pd_{y_N}\widehat{\bP}\right]\,dy_N\\
&\enskip- \int^\infty_0\CF^{-1}_{\xi'}\left[ \frac{1}{|\lambda|^{1/2}+1} (|\lambda|^{1/2}+1)e^{-B_a (x_N+y_N)}  \lambda^{1/2} i\xi_j \widehat{\bP}\right]\,dy_N,\\
\lambda^{1/2} \pd_j \pd_k \bQ_2
&=\int^\infty_0\CF^{-1}_{\xi'}\left[\frac{\lambda^{1/2}}{B_a}\frac{i\xi_j}{A} A e^{-B_a (x_N+y_N)} 
i\xi_k \pd_{y_N}\widehat{\bP}\right]\,dy_N\\
&\enskip -\int^\infty_0\CF^{-1}_{\xi'}\left[ \frac{\lambda^{1/2}}{|\lambda|^{1/2}+1} (|\lambda|^{1/2}+1) e^{-B_a (x_N+y_N)}i\xi_ji\xi_k \widehat{\bP}\right]\,dy_N,\\
\pd_j \bQ_2
&=\int^\infty_0\CF^{-1}_{\xi'}\left[\frac{1}{B_a (|\lambda|^{1/2}+1)} (|\lambda|^{1/2}+1) e^{-B_a (x_N+y_N)} 
i\xi_j \pd_{y_N}\widehat{\bP}\right]\,dy_N\\
&\enskip - \int^\infty_0\CF^{-1}_{\xi'}\left[\frac{1}{|\lambda|^{1/2}+1} (|\lambda|^{1/2}+1) e^{-B_a (x_N+y_N)} i\xi_j \widehat{\bP}\right]\,dy_N,\\ 
\pd_j \pd_k \bQ_2
&=\int^\infty_0\CF^{-1}_{\xi'}\left[\frac{i\xi_j}{B_a A} A e^{-B_a (x_N+y_N)} 
i\xi_k \pd_{y_N}\widehat{\bP}\right]\,dy_N\\
&\enskip- \int^\infty_0\CF^{-1}_{\xi'}\left[\frac{1}{|\lambda|^{1/2}+1} (|\lambda|^{1/2}+1) e^{-B_a (x_N+y_N)} i\xi_j i\xi_k \widehat{\bP}\right]\,dy_N,\\
\pd_j \pd_k \pd_\ell \bQ_2
&=\int^\infty_0\CF^{-1}_{\xi'}\left[\frac{i\xi_j i\xi_k}{B_a A} A e^{-B_a (x_N+y_N)} 
i\xi_\ell \pd_{y_N}\widehat{\bP}\right]\,dy_N\\
&\enskip- \int^\infty_0\CF^{-1}_{\xi'}\left[ \frac{i\xi_j}{A} A e^{-B_a (x_N+y_N)} i\xi_k i\xi_\ell \widehat{\bP}\right]\,dy_N
\end{align*}
for $j, k, \ell=1, \dots, N-1$.
Lemma \ref{m} implies that
\begin{align*}
&\frac{\lambda^m}{B_a(|\lambda|^{1/2}+1)},~\frac{1}{|\lambda|^{1/2}+1},~ \frac{\lambda^{1/2}}{|\lambda|^{1/2}+1} \in \BM_{0,1}',\\
&\frac{i\xi_j}{A},~\frac{\lambda^{1/2}}{B_a}\frac{i\xi_j}{A},~\frac{i\xi_j}{B_a A},~\frac{i\xi_j i\xi_k}{B_a A} \in \BM_{0, 2},
\end{align*}
where $m=0, 1/2, 1$.
Next, we consider the normal derivative of $\bQ_2$.
Let $\pd_N=\pd/\pd x_N$.
Since
\[
\pd_N \bQ_2
=- \int^\infty_0\CF^{-1}_{\xi'}\left[e^{-B_a (x_N+y_N)} 
\pd_{y_N} \widehat{\bP}\right]\,dy_N
+\int^\infty_0\CF^{-1}_{\xi'}\left[B_a e^{-B_a (x_N+y_N)} \widehat{\bP}\right]\,dy_N,
\]
$\lambda \pd_N \bQ_2$ is written by
\begin{align*}
\lambda \pd_N \bQ_2
&=- \int^\infty_0\CF^{-1}_{\xi'}\left[\frac{\lambda^{1/2}}{|\lambda|^{1/2}+1}(|\lambda|^{1/2}+1) e^{-B_a (x_N+y_N)} \lambda^{1/2} \pd_{y_N}\widehat{\bP}\right]\,dy_N\\
&\enskip +\int^\infty_0\CF^{-1}_{\xi'}\left[\frac{\lambda+a}{B_a(|\lambda|^{1/2}+1)} (|\lambda|^{1/2}+1) e^{-B_a (x_N+y_N)} \lambda \widehat{\bP}\right]\,dy_N\\
&\enskip +\int^\infty_0\CF^{-1}_{\xi'}\left[\frac{A}{B_a} A e^{-B_a (x_N+y_N)} \lambda \widehat{\bP}\right]\,dy_N,
\end{align*}
where we have used $1=(\lambda+a+A^2)/B_a^2$.
Lemma \ref{m} implies that
\[
\frac{\lambda^{1/2}}{|\lambda|^{1/2}+1},~
\frac{\lambda+a}{B_a(|\lambda|^{1/2}+1)} \in \BM_{0,1}', \quad
\frac{A}{B_a} \in \BM_{0, 2},
\]
Similarly,
\begin{align*}
\pd_N^3 \bQ_2
&=- \int^\infty_0\CF^{-1}_{\xi'}\left[B_a^2 e^{-B_a(x_N+y_N)} 
\pd_{y_N} \widehat{\bP}\right]\,dy_N
+\int^\infty_0\CF^{-1}_{\xi'}\left[B_a^3 e^{-B_a (x_N+y_N)} \widehat{\bP}\right]\,dy_N\\
&=-\int^\infty_0\CF^{-1}_{\xi'}\left[\frac{\lambda+a}{|\lambda|^{1/2}+1} (|\lambda|^{1/2}+1)e^{-B_a (x_N+y_N)} 
\pd_{y_N} \widehat{\bP}\right]\,dy_N\\
&\enskip + \sum_{j=1}^{N-1} \int^\infty_0\CF^{-1}_{\xi'}\left[\frac{i\xi_j}{A} Ae^{-B_a (x_N+y_N)} 
i\xi_j \pd_{y_N} \widehat{\bP}\right]\,dy_N\\
&\enskip +\int^\infty_0\CF^{-1}_{\xi'}\left[\frac{(\lambda+a)^2}{B_a(|\lambda|^{1/2}+1)} (|\lambda|^{1/2}+1) e^{-B_a (x_N+y_N)}\widehat{\bP}\right]\,dy_N\\
&\enskip -2\int^\infty_0\CF^{-1}_{\xi'}\left[\frac{\lambda+a}{B_a(|\lambda|^{1/2}+1)} (|\lambda|^{1/2}+1) e^{-B_a (x_N+y_N)} (i\xi_j)^2 \widehat{\bP}\right]\,dy_N\\
&\enskip -\int^\infty_0\CF^{-1}_{\xi'}\left[\frac{A}{B_a} A e^{-B_a (x_N+y_N)} (i\xi_j)^2 \widehat{\bP}\right]\,dy_N.
\end{align*}
Lemma \ref{m} implies that
\begin{align*}
\frac{\lambda+a}{|\lambda|^{1/2}+1},~\frac{(\lambda+a)^2}{B_a(|\lambda|^{1/2}+1)},~\frac{\lambda+a}{B_a(|\lambda|^{1/2}+1)} \in \BM_{0,1}', \quad
\frac{i\xi_j}{A},~\frac{A}{B_a} \in \BM_{0, 2}.
\end{align*}
Since the other terms $(1, \lambda^{1/2})\pd_N \bQ_2$, $\lambda^{1/2}(i\xi_j \pd_N, \pd_N^2) \bQ_2$, $(i\xi_j \pd_N, \pd_N^2) \bQ_2$,
and $(i\xi_j i\xi_k \pd_N, i\xi_j \pd_N^2)\bQ_2$ can be written by the operators $M_1(\lambda)$ and $M_2(\lambda)$ which defined in Lemma \ref{R-bound Q},
we observe that there exists an operator family 
\[
\CH_3(\lambda) \in {\rm Hol} (\Sigma_{\epsilon, c_0}, \CL(Z_q(\R^N_+), H^3_q(\R^N_+; \bS_0)))
\] 
such that for any $\lambda=\gamma+i\tau \in \Sigma_{\epsilon, c_0}$,
\[
\bQ_2=\CH_3(\lambda)((\CS_\lambda, 1, \nabla) \bP)
\] with
\begin{equation}\label{R Q2}
\CR_{\CL(Z_q(\R^N_+), B_q(\R^N_+))}(\{(\tau \pd_\tau)^\ell \CT_\lambda\CH_3(\lambda) \mid \lambda \in \Sigma_{\epsilon, c_0}\}) \le r_3,
\end{equation}
where
$Z_q(\R^N_+) = L_q(\R^N_+; \R^{N^4}) \times L_q(\R^N_+; \R^{N^3}) \times L_q(\R^N_+; \bS_0) \times L_q(\R^N_+; \bS_0) \times L_q(\R^N_+; \R^{N^3})$.
Since $\bP=\bH-\pd_N \bQ_1$, we define
\[
	\CH_4(\lambda) = \CH_3'(\lambda) - \CH_3(\lambda) \circ ((\CS_\lambda, 1, \nabla)\pd_N \CH_2'(\lambda)),
\]
where we have set $\CH_3'(\lambda)(\CS_\lambda \bh, \CT_\lambda \bH, \bH) = \CH_3(\lambda)((\CS_\lambda, 1, \nabla) \bH)$ and $\CH_2'(\lambda)(\CS_\lambda \bh, \CT_\lambda \bH, \bH) = \CH_2(\lambda)(\CS_\lambda \bh, \CS_\lambda \bH, \lambda^{1/2}\bH)$.
Then \eqref{R Q1} and \eqref{R Q2} imply that 
\[
	\CH_4(\lambda) \in {\rm Hol} (\Sigma_{\epsilon, c_0}, \CL(\widetilde Y_q(\R^N_+), H^3_q(\R^N_+; \bS_0)))
\] 
is an operator family such that
\[
	\bQ_2=\CH_4(\lambda)(\CS_\lambda \bh, \CT_\lambda \bH, \bH)
\] 
is a solution of \eqref{Q2} for any $\lambda=\gamma+i\tau \in \Sigma_{\epsilon, c_0}$ and $(\CS_\lambda \bh, \CT_\lambda \bH, \bH) \in \widetilde Y_q(\R^N_+)$ with
\[
	\CR_{\CL(\widetilde Y_q(\R^N_+), B_q(\R^N_+))}(\{(\tau \pd_\tau)^\ell \CT_\lambda\CH_4(\lambda) \mid \lambda \in \Sigma_{\epsilon, c_0}\}) \le r_4,
\]
where $r_4$ is determined by $r_1$, $r_2$, and $r_3$.
Setting $\CB_2(\lambda)=\CH_2'(\lambda)+\CH_4(\lambda)$,
$\CB_2(\lambda)$ is desired operator in Theorem \ref{thm:Rbdd}, namely,
\[
\bQ=\CB_2(\lambda)(\CS_\lambda \bh, \CT_\lambda \bH, \bH)
\]
is a solution of \eqref{rQ} with
\begin{equation*}
\CR_{\CL(\widetilde Y_q(\R^N_+), B_q(\R^N_+))}(\{(\tau \pd_\tau)^\ell \CT_\lambda \CB_2(\lambda) \mid \lambda \in \Sigma_{\epsilon, c_0}\}) \le r,
\end{equation*}
where $r$ is determined by $r_1$, $r_2$, and $r_4$.
This completes the proof of Theorem \ref{thm:Rbdd}.

\subsection{The proof of Theorem \ref{thm:Rbdd H}}\label{subsec:proof of main}
In this subsection, we prove the main theorem in this paper.
Recall that \eqref{r0}:
\begin{equation*}
\left\{
\begin{aligned}
&\lambda\bu -\Delta \bu + \nabla \fp + \beta \DV (\Delta \bQ -a \bQ)=\bff,
\enskip \dv \bu=0& \quad&\text{in $\R^N_+$},\\
&\lambda \bQ - \beta \bD(\bu) - \Delta \bQ + a \bQ =\bG& \quad&\text{in $\R^N_+$},\\
&\bu= \bh, \enskip \pd_N \bQ=\bH& \quad&\text{on $\R^N_0$}
\end{aligned}
\right.
\end{equation*}
with $h_N=0$.
According to Theorem \ref{thm:Rbdd}, it is sufficient to consider the following problem:
\begin{equation}\label{r2}
\left\{
\begin{aligned}
&\lambda\bu -\Delta \bu + \nabla \fp + \beta \DV (\Delta \bQ -a \bQ)=\bff,
\enskip \dv \bu=0& \quad&\text{in $\R^N_+$},\\
&\lambda \bQ - \beta \bD(\bu) - \Delta \bQ + a \bQ =\bG& \quad&\text{in $\R^N_+$},\\
&\bu= 0, \enskip \pd_N \bQ=0& \quad&\text{on $\R^N_0$}
\end{aligned}
\right.
\end{equation}
Let 
$E_v[\bff(x)]=(E_v[f_1(x)], \dots, E_v[f_N(x)])$ and 
the $(j, k)$ component of $E_M[\bG(x)]$
be $E_M[G_{jk}(x)]$
with
\begin{align*}
E_v[f_k(x)]&=\left\{
\begin{aligned}
        &E_{even}[f_k(x)]& \quad  & k=1,\ldots, N-1 \\
        &E_{odd}[f_N(x)]& \quad  & k=N
    \end{aligned}
\right. 
\\
E_M[G_{jk}(x)]&=\left\{
\begin{aligned}
        &E_{even}[G_{jk}(x)]& \quad  & j,k=1\ldots, N-1 \\
        &E_{odd}[G_{Nk}(x)]& \quad  & j=N, \enskip k=1,\ldots, N-1 \\
        &E_{odd}[G_{jN}(x)]& \quad  & j=1,\ldots, N-1, \enskip k=N \\
        &E_{even}[G_{NN}(x)]& \quad & j=k=N,
    \end{aligned}
\right. 
\end{align*}
where 
\[
E_{even}[f(x)] = \left\{
\begin{aligned}
&f(x) &\quad x_N > 0, \\
&f(x', -x_N) &\quad x_N < 0, 
\end{aligned}
\right.
\quad
E_{odd}[f(x)] = \left\{
\begin{aligned} 
f(x) &\quad x_N > 0, \\
-f(x', -x_N) &\quad x_N < 0, 
\end{aligned}
\right.
\]
with $x' = (x_1, \ldots, x_{N-1})$.
Theorem \ref{thm:Rbdd RN} implies that 
$\bu_1 = \CA_1(\lambda)(E_v[\bff], \nabla E_M[\bG])$, 
$\bQ_1 = \CB_1(\lambda)(E_v[\bff], \nabla E_M[\bG], E_M[\bG])$, and 
$\fp_1 = \CC_1(\lambda)(E_v[\bff], \nabla E_M[\bG])$
satisfy
\begin{equation}\label{r3}
\left\{
\begin{aligned}
&\lambda\bu_1 -\Delta \bu_1 + \nabla \fp_1 + \beta \DV (\Delta \bQ_1 -a \bQ_1)=\bff,
\enskip \dv \bu_1=0& \quad&\text{in $\R^N_+$},\\
&\lambda \bQ_1 - \beta \bD(\bu_1) - \Delta \bQ_1 + a \bQ_1 =\bG& \quad&\text{in $\R^N_+$},
\end{aligned}
\right.
\end{equation}
where by the definition of extension we have $u_{1N}|_{x_N=0}=0$.
Setting $\bu=\bu_1+\bu_2$, $\bQ=\bQ_1+\bQ_2$, and $\fp=\fp_1+\fp_2$ such that
\begin{equation}\label{r4}
\left\{
\begin{aligned}
&\lambda\bu_2 -\Delta \bu_2 + \nabla \fp_2 + \beta \DV (\Delta \bQ_2 -a \bQ_2)=0,
\enskip \dv \bu_2=0& \quad&\text{in $\R^N_+$},\\
&\lambda \bQ_2 - \beta \bD(\bu_2) - \Delta \bQ_2 + a \bQ_2 =0& \quad&\text{in $\R^N_+$},\\
&\bu_2 = -\bu_1, \enskip \pd_N \bQ_2=-\pd_N \bQ_1 &\quad & \text{in $\R^N_0$},
\end{aligned}
\right.
\end{equation}
$(\bu, \bQ, \fp)$ satisfies \eqref{r2}.
Note that $\pd_N \bQ_1 \in \bS_0$ for $\bQ_1 \in \bS_0$.
Since $(\CS_\lambda \bu_1, \CS_\lambda \pd_N \bQ_1, \lambda^{1/2}\pd_N \bQ_1) \in \widetilde X_q(\R^N_+)$ and
$(\CS_\lambda \bu_1, \CT_\lambda \pd_N \bQ_1, \pd_N \bQ_1) \in \widetilde Y_q(\R^N_+)$
 by Theorem \ref{thm:Rbdd RN} and Remark \ref{rem:Rbdd RN}, 
the solution $(\bu_2, \bQ_2)$ of
\eqref{r4} can be found by Theorem \ref{thm:Rbdd}.

The uniqueness of solutions of \eqref{r0} follows from a duality argument (cf. \cite[Corollary 3.4.8.]{BM}).
This completes the proof of Theorem \ref{thm:Rbdd H}.

\subsection{The proof of Corollary \ref{cor:resolvent}}
Theorem \ref{thm:Rbdd H} and Remark \ref{rem:def of rbdd} imply that the solution $(\bu, \bQ)$ for \eqref{r0} satisfies \eqref{est:uq}.
Therefore, we prove the existence of $\fp$ and its estimate \eqref{est:p}.
According to subsection \ref{subsec:proof of main}, 
it is sufficient to prove the existence of $\fp$ and its norm for \eqref{r}.
Note the fact that the weak Dirichlet Neumann problem:
\begin{equation}\label{DN}
(\nabla p, \nabla \varphi)=(\bff, \nabla \varphi) \quad \forall \varphi \in \widehat H^1_{q'}(\R^N_+) 
\end{equation}
is uniquely solvable in the half space.
Namely, \eqref{DN} has a unique solution $p \in \widehat H^1_q(\R^N_+)$
for any $\bff \in L_q(\R^N_+)^N$
satisfying 
\[
	\|\nabla p\|_{L_q(\R^N_+)} \le C \|\bff\|_{L_q(\R^N_+)}.
\]
Let $(\bu, \bQ) \in H^2_q(\R^N_+)^N \times H^3_q(\R^N_+; \bS_0)$ be a solution of \eqref{r}
satisfying the assertions of Theorem \ref{thm:Rbdd}.
The condition $h_N=0$ implies $u_N=0$ on $\R^N_0$.
Since $u_N=0$ on $\R^N_0$ and $\dv \bu=0$ in $\R^N_+$,
$\fp$ satisfies the variational equation:
\[
	(\nabla \fp, \nabla \vp)=(\Delta \bu- \beta \DV(\Delta \bQ-a\bQ), \nabla \vp)
\]
for any $\vp \in \widehat H^1_{q'}(\R^N_+)$.
The unique solvability of \eqref{DN} implies that the unique existence of $\fp \in \widehat H^1_q(\R^N_+)$ satisfying
\[
	\|\nabla \fp\|_{L_q(\R^N_+)} \le C\|\Delta \bu+ \beta \DV(\Delta \bQ-a\bQ)\|_{L_q(\R^N_+)}.
\]
Since the estimates for $(\bu, \bQ)$ are obtained by Theorem \ref{thm:Rbdd RN}, \eqref{rem:rbdd}, and Theorem \ref{thm:Rbdd}, together with Remark \ref{rem:def of rbdd},
we observe that the existence of $\fp$ satisfying \eqref{est:p},
which completes the proof of Corollary \ref{cor:resolvent}.

\section{The resolvent estimates in the homogeneous Sobolev spaces}

In this section, we prove Corollary \ref{cor:homo}.
Let us consider the resolvent problem with homogeneous boundary conditions:
\begin{equation}\label{r5}
\left\{
\begin{aligned}
&\lambda\bu -\Delta \bu + \nabla \fp + \beta \DV (\Delta \bQ -a \bQ)=\bff,
\enskip \dv \bu=0& \quad&\text{in $\R^N_+$},\\
&\lambda \bQ - \beta \bD(\bu) - \Delta \bQ + a \bQ =\bG& \quad&\text{in $\R^N_+$},\\
&\bu= 0, \enskip \pd_N \bQ=0& \quad&\text{on $\R^N_0$},
\end{aligned}
\right.
\end{equation}
where $a > 0$ and $\beta \neq 0$.
In view of Remark \ref{rem:Rbdd RN}, 
by replacing $\CB_1(\lambda)(E_v[\bff], \nabla E_M[\bG], E_M[\bG])$ with $\CB_1'(\lambda)(E_v[\bff], \nabla E_M[\bG])$
in subsection \ref{subsec:proof of main}, 
we can obtain the estimates for $(\nabla^2 \bQ, \lambda^{1/2}\nabla \bQ, \lambda \bQ)$ in $\dot H^1_q(\R^N_+; \R^{N^4}) \times \dot H^1_q(\R^N_+; \R^{N^3}) \times \dot H^1_q(\R^N_+; \bS_0)$.
In more detail, we can obtain the following result concerning the $\CR$-boundedness for the solution operators of \eqref{r5} in the homogeneous Sobolev spaces.

\begin{thm}\label{thm:homo}
Let $1 < q < \infty$. Let $\epsilon \in (\epsilon_0, \pi/2)$ with $\tan \epsilon_0 \ge |\beta|/\sqrt 2$, and let $c_0$ is a small constant depending on $\epsilon$, $\beta$, and $a$
chosen in Lemma \ref{lem:zl}. 
Let 
\[
	X_q(\R^N_+)=L_q(\R^N_+)^N \times L_q(\R^N_+; \R^{N^3})
\]
and let 
\[
	\bF=(\bff, \nabla \bG) \in X_q(\R^N_+).
\]
There exist 
operator families 
\begin{align*}
&\CA (\lambda) \in 
{\rm Hol} (\Sigma_{\epsilon, c_0}, 
\CL(X_q(\R^N_+), H^2_q(\R^N_+)^N))\\
&\CB (\lambda) \in 
{\rm Hol} (\Sigma_{\epsilon, c_0}, 
\CL(X_q(\R^N_+), H^3_q(\R^N_+; \bS_0)))
\end{align*}
such that 
for any $\lambda = \gamma + i\tau \in \Sigma_{\epsilon, c_0}$,
$\bu = \CA (\lambda) \bF$ and 
$\bQ = \CB (\lambda) \bF$
are unique solutions of \eqref{r5}
and 
\begin{align*}
&\CR_{\CL(X_q(\R^N_+), A_q(\R^N_+))}
(\{(\tau \pd_\tau)^\ell \CS_\lambda \CA (\lambda) \mid 
\lambda \in \Sigma_{\epsilon, c_0}\}) 
\leq r,\\
&\CR_{\CL(X_q(\R^N_+), B_q(\R^N_+))}
(\{(\tau \pd_\tau)^\ell \CS_\lambda \CB (\lambda) \mid 
\lambda \in \Sigma_{\epsilon, c_0}\}) 
\leq r
\end{align*}
for $\ell = 0, 1,$
where
$\CS_\lambda = (\nabla^2, \lambda^{1/2}\nabla, \lambda)$,
$A_q(\R^N_+) = L_q(\R^N_+)^{N^3 + N^2+N}$,
$B_q(\R^N_+) = \dot H^1_q(\R^N_+; \R^{N^4}) \times \dot H^1_q(\R^N_+; \R^{N^3}) \times \dot H^1_q(\R^N_+; \bS_0)$,
and $r$ is a constant independent of $\lambda$.
\end{thm}
Theorem \ref{thm:homo}, together with Remark \ref{rem:def of rbdd}, furnishes Corollary \ref{cor:homo}.
%


\begin{thebibliography}{9}
{\small 

\bibitem{A} H.~Amann,
{Linear and Quasilinear Parabolic Problems}, 
Birkh\"auser Basel (1995).  


\bibitem{BM} D.~Barbera and M.~Murata,
{\it The $L^p$-$L^q$ maximal regularity for the Beris-Edward model in the half-space},
Ann. Sc. Norm. Super. Pisa Cl. Sci., {\bf 56} (2024).

\bibitem{BE} A.~N. Beris and B.~J.~Edwards,
{Thermodynamics of Flowing Systems with Internal Microstructure}, 
Oxford Engrg. Sci. Ser. {\bf 36}, Oxford University Press, Oxford, New York,
(1994).

\bibitem{D} R.~Danchin, 
{\it Density-dependent Incompressible Fluids in Bounded Domain}, 
J. Math. Fluid Mech., {\bf 8} (2006), 333–381.

\bibitem{DHP} R.~Denk, M.~Hieber and J.~Pr\"u\ss, 
{$\CR$-boundedness, Fourier multipliers and problems of 
elliptic and parabolic type}. \newblock Memoirs of AMS. Vol 166. 
No. 788. (2003).

\bibitem{ES} Y.~Enomoto and Y.~Shibata, {\it On the $\CR$-sectoriality 
and its application to    
some mathematical study of the viscous compressible fluids}, 
Funk. Ekvac., {\bf 56} (2013), 441--505. 

\bibitem{GS} Y.~Giga and H.~Sohr, 
{\it Abstract $L^p$ estimates for the Cauchy problem with applications to
the Navier-Stokes equations in exterior domains}, 
J. Funct. Anal., {\bf 102 (1)} (1991), 72–94.

\bibitem{HHW} M.~Hieber, A.~Hussein, and M.~Wrona, 
{\it Strong well-posedness of the Q-tensor model for liquid crystals: the case of arbitrary ratio of tumbling and aligning effects $\xi$}, 
Arch. Ration. Mech. Anal., {\bf 248} (2024), no. 3, Paper No. 40.

\bibitem{HD} J.~Huang and S~.Ding,
{\it Global well-posedness for the dynamical $Q$-tensor model of liquid crystals},
Science China Mathematics, {\bf 58} (2015), 1349--1366.

\bibitem{K} T.~Kato,
{Perturbation Theory for Linear Operators},
Springer, Berlin (1995).

\bibitem{MZ} A.~Majumdar and A.~Zarnescu, 
{\it Landau-de Gennes theory of nematic liquid crystals: The
Oseen-Frank limit and beyond}, Arch. Ration. Mech. Anal., {\bf 196} (2010), 227--280.

\bibitem{MS} M.~Murata and Y.~Shibata, 
{\it Global well posedness for a Q-tensor model of nematic liquid crystals},
J. Math. Fluid Mech., {\bf 24 (1)} (2022), Paper No. 34.

\bibitem{SS1} M.~Schonbek and Y.~Shibata, 
{\it Global well-posedness and decay for a $\mathbb Q$ tensor model of incompressible nematic liquid crystals in $\R^N$},
J. Differential Equations, {\bf 266 (6)} (2019), 3034--3065. 

\bibitem{SS}
Y.~Shibata and S.~Shimizu, 
{\it On a resolvent estimate for the Stokes system with Neumann
boundary condition}, Differential Integral Equations, {\bf 16} (2003), 385-–426.

\bibitem{ShiS}
Y.~Shibata and S.~Shimizu, 
{\it On the maximal $L_p$-$L_q$ regularity of the Stokes problem
with first order boundary condition; model problems}, 
J. Math. Soc. Japan, {\bf 64 (2)} (2012), 561--626.

\bibitem{S} Y.~Shibata,
{\it $\CR$ boundedness, maximal regularity and free boundary problems for the Navier Stokes equations},
Mathematical analysis of the Navier-Stokes equations, 
Lecture Notes in Math., {\bf 2254}, Fond. CIME/CIME Found. Subser., Springer, 
(2020), 193--462. 

\bibitem{X} Y.~Xiao,
{\it Global strong solution to the three-dimensional liquid crystal flows of Q-tensor model},
J. Differential Equations {\bf 262 (3)} (2017), 1291--1316.

}
\end{thebibliography}
\end{document}